\documentclass[12pt,fleqn]{article}
\usepackage{mathtools}
\usepackage[margin=0.7in]{geometry}
\usepackage{amsmath,bm}
\usepackage{amsfonts}
\usepackage{amsthm}
\usepackage{graphicx}
\usepackage{tikz}
\usepackage{lscape}
\usetikzlibrary{decorations}
\usetikzlibrary{decorations.markings,arrows}

\newtheorem{theorem}{Theorem}
\newtheorem{prop}[theorem]{Proposition}
\newtheorem{lemma}[theorem]{Lemma}
\newtheorem{remark}[theorem]{Remark}
\theoremstyle{definition}

    \def\Re{{\rm Re \,}}
    \def\Im{{\rm Im \,}}

    \def\bigO{{\cal O}}

        \def\P2n{{\rm P}_{{\rm II}}^{(n)}}

\newcommand{\be}{\begin{equation}}
\newcommand{\ee}{\end{equation}}

\newcommand{\wt}{\widetilde}

\renewcommand{\th}{\theta}

\newcommand{\ep}{\varepsilon }
\newcommand{\ve}{\varepsilon }

\tikzset{middlearrow/.style={
       decoration={markings,
         mark= at position 0.5 with {\arrow{#1}} ,
        },
       postaction={decorate}    }
}

\title{Airy-kernel determinant on two large intervals}

\author{I. Krasovsky and T. Maroudas}

\date{}

\begin{document}

\maketitle

\bigskip

\centerline{\it{In memory of Harold Widom, 1932--2021}}

\bigskip
\bigskip

\begin{abstract}
We find the probability of two gaps of the form $(sc,sb)\cup (sa,+\infty)$, $c<b<a<0$, for large $s>0$, in the edge scaling limit of the Gaussian Unitary Ensemble of random matrices, including the multiplicative constant in the asymptotics. 
\end{abstract}

\section{Introduction}
For $c<b<a<0$, set $J=(c,b)\cup(a,+\infty)$. Let $\text{Ai}(x)$, $\text{Ai}'(x)$ be the Airy function and its derivative, respectively.
Consider the (trace class) operator $K^{\mathrm{Ai}}$, acting on $L^2(sJ)$, $s>0$, with kernel
\[
K^{\mathrm{Ai}}(z,z')=\frac{\text{Ai}(z)\text{Ai}'(z')-\text{Ai}'(z)\text{Ai}(z')}{z-z'}=\int_{0}^{\infty}\text{Ai}(z+\zeta)\text{Ai}(z'+\zeta)d\zeta.
\]
We are interested in the large-$s$ behaviour of the corresponding Fredholm determinant
\begin{equation}\label{AIRYDET}
P^{\mathrm{Ai}}(sJ)=\det(I-K^{\mathrm{Ai}})_{(sc,sb)\cup(sa,+\infty)}.
\end{equation}
The determinant $P^{\mathrm{Ai}}(sJ)$ is the probability of 2 gaps $(sc,sb)$ and $(sa,+\infty)$ in the edge scaling limit of the Gaussian Unitary Ensemble (GUE), see, e.g, \cite{Mehta, Dbook}.

Clearly, by rescaling $s$, we can consider the second interval to be $(-s,+\infty)$. 
In the case of one gap, $(-s,+\infty)$,
the determinant \[P^{\mathrm{Ai}}(-s,+\infty)=\det(I-K^{\mathrm{Ai}})_{(-s,+\infty)}\] is the Tracy-Widom distribution \cite{TW} --- the 
distribution of the largest eigenvalue of the GUE. The same determinant also describes the distribution of the longest increasing subsequence in a random permutation \cite{BBRD}. Its large $s$ asymptotics were first considered by Tracy and Widom \cite{TW} in 1994, who observed that
\begin{equation}\label{TW}
 P^{\mathrm{Ai}}(-s,+\infty)= \exp \left \{ -\int_{s}^{\infty}(x-s)u^{2}(x)dx\right \},
\end{equation}
where $u(x)$ is the Hastings-McLeod solution of the Painlev\'e II
equation
\begin{equation}\label{ds:30}
u''(x)=xu(x)+2u^3(x)\,,
\end{equation}
specified by the following asymptotic condition:
\begin{equation}\label{ds:40}
u(x)\sim \mathrm{Ai}(x)\qquad \mbox{as}\quad x\to+\infty.
\end{equation}
The asymptotics of the logarithmic derivative $(d/ds)\log P^{\mathrm{Ai}}(-s,+\infty)$ follow,
up to a constant (which is in fact zero), from (\ref{ds:40})
and the known asymptotics of the Hastings-McLeod solution at $-\infty$. 
Integrating, Tracy and Widom obtained
\begin{equation}\label{AirydetoneGAP}
    \log \det(I-K^{\mathrm{Ai}})_{(-s,+\infty)}=-\frac{1}{12}s^3-\frac{1}{8}\log s+\chi_{\mathrm{Airy}}+\bigO\left( \frac{1}{s^{3/2}}\right), \quad s\to \infty,
\end{equation}
up to an undetermined constant $\chi_{\mathrm{airy}}$. Tracy and Widom did however, conjecture its value to be
\begin{equation}\label{constAiry}
\chi_{\mathrm{Airy}}=\frac{1}{24}\log 2 + \zeta^{\prime}(-1), 
\end{equation}
where $\zeta$ denotes Riemann's $\zeta$-function. 
The proof of \eqref{AirydetoneGAP} with \eqref{constAiry} was given in \cite{DIKAiry},
and, by a different method, in \cite{BBD}.

On the other hand, in the {\it bulk} of the spectrum of GUE, the probability of a gap $(-s,s)$
is given by the Fredholm determinant 
$\det(I-K^{\mathrm{sine}})_{(-s,s)}$
of the trace class operator $K^{\mathrm{sine}}$ on $L^2(-s,s)$ with the sine kernel
\[
K^{\mathrm{sine}}(x,y)=\frac{\sin (x-y)}{\pi (x-y)}.
\]
In this case, we have the following large $s$ asymptotics
\begin{equation}\label{1gap}
\log \det(I-K^{\mathrm{sine}})_{(-s,s)}=-\frac{s^2}{2}-\frac{1}{4}\log s+c_{\mathrm{sine}}+\mathcal O\left(\frac{1}{s}\right),\qquad s\to \infty,
\end{equation}
where
\begin{equation}\label{constsine}
c_{\mathrm{sine}}=\frac{1}{12}\log 2+3\zeta'(-1).
\end{equation}

The leading term $-\frac{s^2}{2}$ in \eqref{1gap} was found by Dyson in 1962 in \cite{Dyson62}. 
The terms $-\frac{s^2}{2}-\frac{1}{4}\log s$ were then computed by des Cloizeaux and Mehta \cite{CM} in 1973.
The constant \eqref{constsine}, known as the Widom-Dyson constant, was identified
by Dyson \cite{Dyson} in 1976. The works \cite{Dyson62}, \cite{CM}, and \cite{Dyson} are not fully rigorous.
The first rigorous confirmation of the main term, i.e. the fact that $\log \det(I-K^{\mathrm{sine}})_{(-s,s)}=-\frac{s^2}{2}(1+o(1))$, was given by Widom \cite{W2}  in 1994. The full asymptotics (up to an undetermined value of $c_{\mathrm{sine}}$) was justified by Deift, Its, and Zhou in 1997 in \cite{DIZ}.
The value \eqref{constsine} of $c_{\mathrm{sine}}$ in \eqref{1gap} was justified in 3 different ways 
in  \cite{Ksine}, \cite{DIKZ}, \cite{Ehrhardt} (see
\cite{DIKreview} and \cite{IB} for more historical details).

As we will see, the present work relates the results \eqref{AirydetoneGAP} and \eqref{1gap} in some sense. 

Since $J$ consists of 2 intervals, we expect the appearance of Jacobi $\theta$-functions in the asymptotics. This phenomenon was first observed
in \cite{DIZ}, where the authors considered the sine-kernel determinant on several large intervals, and found its asymptotics up to a multiplicative constant.
The constant in the case of 2 intervals for the sine-kernel determinant was recently found in \cite{IB}. In the present work, 
following many
ideas from \cite{IB}, we establish the asymptotics of \eqref{AIRYDET} including the relevant multiplicative constant.
To describe our results, we introduce some notation.

Let
\[p(z)=(z-a)(z-b)(z-c),\]
and consider the Riemann surface of the function $p(z)^{1/2}$ with branch cuts on $\mathbb{R}\setminus J$
(see Figure \ref{Figcycles}). Fix the first sheet of the surface by the condition $p(z)^{1/2}>0$ for $z>a$.
\begin{figure}
	\begin{center}
		\begin{tikzpicture}
		\node [above] at (0,0){$-\infty$};
		\node [above] at (2,0){$c$};
		\node [above] at (5,0) {$b$};
		\node [above] at (7,0) {$a$};

		\draw[black,fill=black]  (0,0) circle [radius=0.04];
		\draw[black,fill=black]  (2,0) circle [radius=0.04];
		\draw[black,fill=black]  (5,0) circle [radius=0.04];
		\draw[black,fill=black]  (7,0) circle [radius=0.04];

		\node [below] at (6.5,-1) {$A_1$};
		\node [below] at  (0.6,-1){$A_0$};
		\node [below] at (3.7,-0.88){$B_1$};

		\draw  (5,0)--(7,0);
		\draw  (0,0)--(2,0);

		\draw[dashed,decoration={markings, mark=at position 0.5 with {\arrow[thick]{>}}},
		postaction={decorate}] (1,0) to [out=90,in=180] (3.5,1.5) to [out=0,in=90]  (6,0);
		
		\draw[decoration={markings, mark=at position 0.5 with {\arrow[thick]{<}}},
		postaction={decorate}] (1,0) to [out=270,in=180] (3.5,-1.5) to [out=0,in=270]  (6,0);

		\draw[decoration={markings, mark=at position 0.5 with {\arrow[thick]{>}}},
		postaction={decorate}] (-1,0) to [out=90,in=180] (1,1) to [out=0,in=90]  (3,0);
		\draw[decoration={markings, mark=at position 0.5 with {\arrow[thick]{<}}},
		postaction={decorate}] (-1,0) to [out=270,in=180] (1,-1) to [out=0,in=270]  (3,0);     
		
		\draw[decoration={markings, mark=at position 0.5 with {\arrow[thick]{>}}},
		postaction={decorate}] (4,0) to [out=90,in=180] (6,1) to [out=0,in=90]  (8,0);
		\draw[decoration={markings, mark=at position 0.5 with {\arrow[thick]{<}}},
		postaction={decorate}] (4,0) to [out=270,in=180] (6,-1) to [out=0,in=270]  (8,0);  
		\end{tikzpicture} 
		\caption{Cycles on the Riemann surface.}\label{Figcycles}\end{center}
\end{figure}
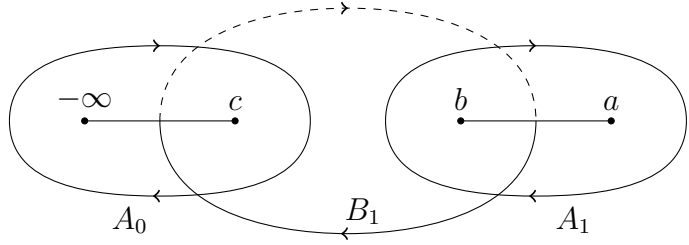
Define the elliptic integrals around cycles
\begin{equation}\label{Jkdefinition}
I_k= \frac{1}{2}\oint_{A_1} \frac{\zeta^kd\zeta}{p(\zeta)^{1/2}}=
 \frac{1}{i}\int_{b}^{a}\frac{\zeta^kd\zeta}{|p(\zeta)^{1/2}|},\qquad 
     J_k=  -\frac{1}{2}\oint_{B_1} \frac{\zeta^kd\zeta}{p(\zeta)^{1/2}}
     =-\int_{c}^{b}\frac{\zeta^kd\zeta}{|p(\zeta)^{1/2}|},  \qquad k=0,1,2,
\end{equation}
where the cycles $A_0$, $A_1$, $B_1$ are depicted in Figure \ref{Figcycles}:
parts represented by solid lines are on first sheet of the surface, while dotted line is on the second.

Consider the function
\begin{equation} \label{definitionofg}
g(z)=\int_{a}^z\frac{q(\zeta)}{p(\zeta)^{1/2}}d\zeta,\qquad \mathbb{C}\setminus (-\infty,a],
\end{equation}
on the first sheet. Here 
\[
q(z)=z^2+q_1z+q_0
\]
is a polynomial such that 
\begin{equation}\label{gcond1}
\int_a^z \frac{q(\zeta)}{p(\zeta)^{1/2}}d\zeta =\frac{2}{3}z^{3/2}+\bigO\left(\frac{1}{z^{1/2}}\right), \quad z\to \infty,
\end{equation}and
\begin{equation} \label{gcond2}
\int_{c}^{b}\frac{q(\zeta)}{p(\zeta)^{1/2}}d\zeta=0.
\end{equation}

As we verify in Lemma \ref{lemmag}  below, these conditions determine the coefficients of $q(z)$:
\begin{equation}\label{defq1}
    q_1=-\frac{a+b+c}{2},\qquad 
    q_0=-\frac{J_2+q_1J_1}{J_0}=\frac{1}{3}(ab+ac+bc)+\frac{1}{3}q_1\frac{J_1}{J_0}.
\end{equation}
Furthermore, we will see that the function $g(z)$ admits a large-$z$ asymptotic expansion of the form
\begin{equation}
    g(z)=\frac{2}{3}z^{3/2}+\frac{\alpha_1}{z^{1/2} }+\frac{\alpha_2}{z^{3/2}}+O\left( z^{-5/2}\right),\qquad z\to \infty,
\end{equation}
where, in particular,
\begin{equation}\label{definitionofalpha2}
    \alpha_2=    -\frac{1}{12}\left( a^3+b^3+c^3-(a+b)(a+c)(b+c)-8q_0q_1\right).
\end{equation}

Let
\begin{equation}\label{definitionofV}
\Omega=\frac{g_+(b)}{\pi i}\in \mathbb{R},\qquad \tau=\frac{I_0}{J_0},\quad \Re\tau=0,\quad \Im\tau>0,
\end{equation}
and recall the third Jacobi $\theta$-function given by
\begin{equation}\label{definitionoftheta3}
\theta(z)=\theta_3(z)=\theta_3(z;\tau)=\sum_{m\in \mathbb{Z}}e^{2 \pi i m z+\pi i\tau m^2}.
\end{equation}

The $\theta$-function satisfies the following periodicity relations, see e.g. \cite{WW},
\begin{equation}\label{qprelations}
\theta_3(z+1)=\theta_3(z)\quad \mathrm{ and }\quad \theta_3(z+\tau)=e^{-2 \pi i z-\pi i \tau}\theta_3(z).
\end{equation}

We now state our result.

\begin{theorem}\label{Mainthm}
The following asymptotics hold 
\begin{equation}\label{asP}
    \log P^{\mathrm{Ai}}(sJ)= -\alpha_2 s^3 -\frac{1}{2} \log s+\log \frac{\theta_3(s^{3/2}\Omega )}{\theta_3(0)}+\chi+o(1), \qquad s\to +\infty,
\end{equation}
with \be \chi=\frac{1}{4}\log(a-c)-\frac{1}{8}\log \lvert  2q(a)q(b)q(c)\rvert+ c_{\mathrm{sine}}+ \chi_{\mathrm{Airy}},\ee
where the constants $c_{\mathrm{sine}}$, $\chi_{\mathrm{Airy}}$ are given by \eqref{constsine} and \eqref{constAiry}, respectively.
\end{theorem}

\begin{remark}
Alternatively, by the identity for $\theta_3(0)$ in \eqref{id123} below, 
\eqref{asP} can be written as
\begin{equation} \label{RESULT1}
    \log P^{\mathrm{Ai}}(sJ)= -\alpha_2 s^3 -\frac{1}{2} \log s+\log \theta_3(s^{3/2}\Omega ;\tau)+\chi_1+o(1),\qquad s\to +\infty,
\end{equation}
with \be \chi_1=-\frac{1}{2}\log \left| \frac{J_0}{\pi} \right|-\frac{1}{8}\log  \lvert q(a)q(b)q(c)\rvert+ 4\zeta^{\prime}(-1) \ee
This expression for the determinant exhibits a certain $duality$ in $I_0$ and $J_0$. By means of the relation
\begin{equation}
    \theta_3(z;\tau)=\frac{e^{-i \pi z^2/\tau}}{\sqrt{-i\tau}}\theta_3\left(\frac{z}{\tau};-\frac{1}{\tau}\right),
\end{equation}
we may write the determinant in a third way. For $\alpha_2^*=\frac{i\pi \Omega ^2}{\tau}+\alpha_2$, we have that
\begin{equation}\label{ThirdWayDET}
   \log P^{Ai}(sJ)= -\alpha_2^* s^3+\log   \theta_3\left(s^{3/2}\frac{\Omega}{\tau};-\frac{1}{\tau}\right)-\frac{1 }{2}\log s +\chi_1^*+o(1),\qquad 
   s\to +\infty,
\end{equation}
with \be \chi_1^*=-\frac{1}{2}\log \left| \frac{I_0}{\pi} \right|-\frac{1}{8}\log \lvert q(a)q(b)q(c)\rvert +4\zeta^{\prime}(-1). \ee

\end{remark}


In the recent work \cite{BCL}, Blackstone, Charlier and Lenells have simultaneously and independently analyzed the large-$s$ asymptotics of  $\log P^{\mathrm{Ai}}(sJ)$. They found the expansion $-\alpha_2 s^3 -\frac{1}{2}\log s+\log \theta_3(s^{3/2}\Omega)+\chi'+\bigO(1/s)$ with
an undetermined  constant term $\chi'=\chi'(a,b,c)$. (This analysis was then extended by the authors
to the case of $n$ gaps in the bulk of the Airy process in \cite{BCLairybulk}, and 
in the Bessel process in \cite{BCLbessel}.) 
They followed the approach of \cite{DIZ}, and used Riemann-Hilbert analysis to obtain the asymptotics 
of the derivative $\frac{d}{ds} \log P^{\mathrm{Ai}}(sJ)$. To determine the multiplicative constant in the asymptotics of the determinant, 
one would need to integrate the logarithmic derivative over $s$. However, there is no appropriate initial point $s_0$ where the asymptotics
of $P^{\mathrm{Ai}}(s_0J)$ would be independently known. For this reason the problem of determining the constant is different and requires additional ideas. As mentioned above, it was first solved for a single interval in the bulk of the spectrum of GUE.
The solution in \cite{Ksine} and also the one in \cite{DIKZ} involved representing the sine-kernel determinant as a double-scaling limit of a 
Toeplitz determinant (whose asymptotics at certain points are either known or can be independently determined)
and then integrating a differential identity for Toeplitz determinants starting from this point. The differential identity can be found in the asymptotic form by a Riemann-Hilbert analysis. Similarly, the constant $c_{\mathrm{Airy}}$ in \eqref{AirydetoneGAP} was determined in \cite{DIKAiry} by integrating a differential identity for a Hankel determinant and using the fact that the asymptotics of the Hankel determinant can be independently established at a certain point, thus providing an initial point for the integration.

It was observed in \cite{IB} that for the sine-kernel determinant on 2 intervals, one does not need to reduce the problem to Toeplitz determinants
(although, of course, one still can), but rather it is easier to notice first that if the 2 intervals are far apart from each other in comparison with their width, then the determinant splits (to the main orders in $s$) into a product of 2 determinants, each on a single interval so that we can use the asymptotics \eqref{1gap}. Then to complete the solution, one determines the asymptotic form of a differential identity with respect to the edges of the intervals and performs integration starting from a point where splitting into the product occurs.

 Here we follow a similar approach to \cite{IB}. First, in Section \ref{secseplemma}, we establish a separation lemma, which states that if the 
 length of the interval $(c,b)$ is relatively small compared to $b-a$, and $a$ is close to zero from the left then $P^{\mathrm{Ai}}(sJ)$
 is written (to main orders in $s$) as a product of a sine-kernel and Airy-kernel determinant for which we can use the asymptotics \eqref{1gap},
 \eqref{AirydetoneGAP}, respectively. More precisely, we prove 
 
\begin{lemma}[Separation of gaps]\label{seplemma}
Set $b=c+\frac{2t_0}{s^{3/2}}$ and $a=-\frac{t_1}{s}$, where $t_0=t_1=(\log s)^{1/8}$. Then as $s\to +\infty$, 
\begin{align}
    \log P^{\mathrm{Ai}}(sJ)&= \log \det(I-K^{\mathrm{sine}})_{(-t_0\sqrt{|c|},t_0\sqrt{|c|})}+
    \log \det(I-K^{\mathrm{Ai}})_{(-t_1,+\infty)}+o(1)\\
    &= -\frac{\lvert c \rvert}{2}t_0^2-\frac{1}{4}\log(\sqrt{\lvert c \rvert} t_0) +c_{\mathrm{sine}}
     -\frac{1}{12}t_1^3-\frac{1}{8}\log t_1 + \chi_{\mathrm{Airy}}+o(1).\label{seplemmaexpansion}
\end{align}
\end{lemma} 

\begin{remark}
We can also choose different values for $t_0$, $t_1$, and slightly larger in $s$ than $(\log s)^{1/8}$.
\end{remark}

To solve the problem for arbitrary fixed $c<b<a<0$, we proceed as follows. In Section \ref{secdiffid}, we formulate differential identities which express
the derivatives with respect to the edges, $\frac{d}{da}\log P^{\mathrm{Ai}}(sJ)$, $\frac{d}{db} \log P^{\mathrm{Ai}}(sJ)$,  in terms of a solution of a certain Riemann-Hilbert problem. In Section \ref{secRHP}, this problem is solved asymptotically for large $s$ 
and fixed $c<b<a<0$, by
the Deift-Zhou steepest descent method (as for the sine-kernel on $2$ intervals, the solution involves $\theta$-functions).
In Section \ref{secRHext} we verify that this solution is extendable for variable edges up to the scaling regime of Lemma \ref{seplemma}. We then substitute the solution into the differential identities and perform integration: first setting $a=\alpha-x$, $b=\beta+x$, where $\alpha$ is close to zero and
$\beta$ is close to $c$ in the sense of Lemma \ref{seplemma}, we integrate 
$\frac{d}{dx} \log P^{\mathrm{Ai}}(sJ)$ from $x=0$ to $x_0$ such that $a=\alpha-x_0$.
This fixes the desired value of $a$.
Then we integrate the identity $\frac{d}{db} \log P^{\mathrm{Ai}}(sJ)$ to the general position of $b$ and thus conlcude the proof of Theorem \ref{Mainthm}. 
Note that unlike the cases of just 1 gap, the integration of the differential identities for 2 gaps is technically involved:
we have to use various identities for $\theta$-functions (see Section \ref{secidentities}) and averaging over fast oscillations. The details 
of this computation are presented in Section \ref{proofofthm}.


\section{Separation of gaps: Proof of Lemma \ref{seplemma}}\label{secseplemma}

In this section $C_j$, $j=1,2,\dots$, will denote positive constants whose value may change from line to line.

Recall the kernel
\be\label{airykernel1}
K^{\textmd{Ai}}(z,z')=\frac{\text{Ai}(z)\text{Ai}'(z')-\text{Ai}'(z)\text{Ai}(z')}{z-z'},
\ee
and denote
\be
K^{\textmd{sine}}_{\alpha}(z,z')=\frac{\sin(\alpha(z-z'))}{\pi(z-z')}.
\ee
Thus $K^{\textmd{sine}}(z,z')=K^{\textmd{sine}}_{1}(z,z')$.
Let
\be
b=c+\frac{2t_0}{s^{3/2}},\qquad t_0/s^{1/2}\to 0,;\qquad a=-\frac{t_1}{s},\qquad  t_1/s\to 0.
\ee

We first observe the following.
\begin{prop}\label{Ksineinternal} For $z,z^{\prime}\in(sc,sb)$, let
\[z=cs+\frac{x t_0}{\sqrt{s}}, \qquad z^{\prime}=cs+\frac{yt_0}{\sqrt{s}},\qquad x,y\in(0,2). \]
Then, uniformly for  $x,y\in (0,2)$,
\begin{align}\label{kernelestimate}
    K^{\textmd{Ai}}(z,z^{\prime})=\frac{\sqrt{s}}{t_0}\left(K^{\textmd{sine}}_{t_0\sqrt{|c|}}(x,y)+\bigO\left(\frac{t_0^2}{s^{3/2}}\right)\right),\qquad t_0^2/s^{3/2}\to 0, \qquad s \to +\infty.
\end{align}
\end{prop}
\begin{proof}
We will make use of the following expansions of the Airy function for large, negative argument (see, e.g., \cite{AS}):
\begin{equation}\label{airysine1}
    \textmd{Ai}(-z)=\frac{1}{\sqrt{\pi}z^{1/4}}\left[ \sin\left(\frac{2}{3}z^{3/2}+\frac{\pi}{4}\right)+\bigO\left(\frac{1}{z^{3/2}}\right)\right],
\end{equation}
and for the derivative ($\textmd{Ai}^{\prime}(u)=\frac{d}{du}\textmd{Ai}(u)$)
\begin{equation}\label{airysine2}
    \textmd{Ai}^{\prime}(-z)=-\frac{z^{1/4}}{\sqrt{\pi}}\left[ \cos\left(\frac{2}{3}z^{3/2}+\frac{\pi}{4}\right)+\bigO\left(\frac{1}{z^{3/2}}\right)\right],
\end{equation}
as $z\to +\infty$, where the principal branches of the roots are taken with cuts along $(-\infty,0)$. 
Let us define, for $\Re (z)>0$,
\be \textmd{Err}_{\textmd{Ai}}(z)=\sqrt{\pi}z^{1/4} \textmd{Ai}(-z)-\sin\left(\frac{2}{3}z^{3/2} +\frac{\pi}{4}\right), \ee
so that
\begin{equation}\label{airyerror} \textmd{Ai}(-z)=\frac{1}{\sqrt{\pi} z^{1/4}}\left[ \sin\left(\frac{2}{3}z^{3/2}+\frac{\pi}{4}\right)+\textmd{Err}_{\textmd{Ai}}(z)\right]. 
\end{equation}
Since the Airy function is entire it follows, by definition, that $\textmd{Err}_{\textmd{Ai}}(z)$ is analytic in $\Re (z) >0$. Moreover it follows from differentiating $\textmd{Err}_{\textmd{Ai}}$, making use of the expansions (\ref{airysine1}) and (\ref{airysine2}), that
\begin{equation}\label{KernelErrorEstimate}
    \textmd{Err}_{\textmd{Ai}}(z)=\bigO\left( \frac{1}{z^{3/2}}\right),\quad \textmd{Err}_{\textmd{Ai}}^{\prime}(z)=\bigO\left(\frac{1}{z}\right), \quad \textmd{Err}_{\textmd{Ai}}^{\prime\prime}(z)=\bigO\left(\frac{1}{z^{1/2}}\right), 
\end{equation} 
where for the second derivative we have used, in addition, the Airy equation $\textmd{Ai}^{\prime\prime}(z)=z\textmd{Ai}(z)$.

Let us denote 
\[ S(z)=\sin\left(\frac{2}{3}(-z)^{3/2}+\frac{\pi}{4}\right),\quad \textmd{and} \quad C(z)=\cos\left(\frac{2}{3}(-z)^{3/2}+\frac{\pi}{4}\right). \]
Noting that
\[
\frac{(-z')^{1/4}}{(-z)^{1/4}}=1+ \frac{(-z')^{1/4}-(-z)^{1/4}}{(-z')^{1/4}}=1+(z-z')\bigO(1/s)=1+\bigO\left(\frac{t_0(x-y)}{s^{3/2}}\right),
\]
we obtain, by means of (\ref{airyerror}) and its derivative, that
\be\label{kernelexp}
\begin{aligned}
    &\pi (z-z') K^{\textmd{Ai}}(z,z')=
    \pi\left[\textmd{Ai}(z)\textmd{Ai}^{\prime}(z^{\prime})-\textmd{Ai}(z^{\prime})\textmd{Ai}^{\prime}(z)\right]=\\
    & -\left[S(z)C(z^{\prime})-S(z^{\prime})C(z)\right]+C(z)\textmd{Err}_{\textmd{Ai}}(-z^{\prime})- C(z^{\prime})\textmd{Err}_{\textmd{Ai}}(-z)\\
    & +\frac{1}{(-z')^{1/4}}\left[S(z')+\textmd{Err}_{\textmd{Ai}}(-z')\right]\textmd{Err}_{\textmd{Ai}}^{\prime}(-z)
    -\frac{1}{(-z)^{1/4}}\left[S(z)+\textmd{Err}_{\textmd{Ai}}(-z)\right]\textmd{Err}_{\textmd{Ai}}^{\prime}(-z')\\
    & +\bigO\left(\frac{t_0(x-y)}{s^{3/2}}\right).
\end{aligned}
\ee

For $f,g$ denoting analytic functions, there are $ z_1^*,z_2^*\in(z,z^{\prime})$ such that
\be\label{UniformKernelEstimate}
\begin{aligned}
    &\frac{f(z)g(z^{\prime})-f(z^{\prime})g(z)}{z-z^{\prime}}=\frac{f(z)\left(g(z)+g^{\prime}(z_1^*)(z-z^{\prime})\right)-\left(f(z)+f^{\prime}(z_2^*)(z-z^{\prime})\right)g(z)}{z-z^{\prime}}\\
    &=f(z)g^{\prime}(z_1^*)-f^{\prime}(z_2^*)g(z).
\end{aligned}
\ee
Applying this to the pairs $f(z)=C(z)$, $g(z)=\textmd{Err}_{\textmd{Ai}}(-z)$ and 
$f(z)=\frac{1}{(-z)^{1/4}}\left[S(z)+\textmd{Err}_{\textmd{Ai}}(-z)\right]$, $g(z)=\textmd{Err}_{\textmd{Ai}}^{\prime}(-z)$, in \eqref{kernelexp},
and using the estimates \eqref{KernelErrorEstimate},
we obtain \eqref{kernelestimate},  uniformly in $x,y \in (0,2)$.
\end{proof}

With the assistance of the above result we may now prove Lemma 1. 

Consider the kernel 
\begin{equation}
\widehat{K^{\textmd{Ai}}}(z,z^{\prime})=\begin{cases} K^{\textmd{Ai}}(z,z^{\prime}), \quad z,z^{\prime} \in (sc,sb) \textmd{ or } z,z^{\prime}\in \left(sa,+\infty\right) \\ 0,\quad \textmd{ otherwise.}
\end{cases}
\end{equation}
The corresponding determinant splits into the desired product of determinants up to a small error. 
We have 
\begin{prop}\label{HATANDPRODEST}There exist constants $C_0,C_1>0$ depending only on $c$ such that
\be
\Biggl\lvert \det(I-\widehat{K^{\textmd{Ai}}})_{L^2(sJ)}- \det(I-K^{\textmd{sine}})_{L^2(0,2\lvert c \rvert^{1/2}t_0)}
\det(I-K^{\textmd{Ai}})_{L^2(-t_1,+\infty)}\Biggr\rvert \le C_0 \frac{e^{C_1 t_0^2}}{s^{3/2}}.
\ee
\end{prop}
\begin{proof}
First we note that
\be\label{41}
\begin{aligned}
    &\det(I-\widehat{K^{\textmd{Ai}}})_{L^2(sJ)}=1+\sum_{m=1}^{\infty}\sum_{k=0}^m \frac{(-1)^m}{(m-k)!k!}\\
   &\times \int_{\begin{matrix}z_1,\dots,z_k \in(cs,bs) \\ z_{k+1},\dots,z_m\in(as,+\infty)\end{matrix}} \det(K^{\textmd{Ai}}(z_i,z_j))_{i,j=1}^k
   \det(K^{\textmd{Ai}}(z_i,z_j))_{i,j=k+1}^m dz_1\cdots dz_m \\    
   &= \det(I-K^{\textmd{Ai}})_{L^2(sc,sb)}\det(I-K^{\textmd{Ai}})_{L^2(-t_1,+\infty)}.
\end{aligned}
\ee
Furthermore, using the change of variables,
\be\label{changevar}
z_j=sc+\frac{t_0 x_j}{\sqrt{s}},\quad j=1,\dots,k,
\ee
and Proposition \ref{Ksineinternal}, we write
\be\label{42}
\begin{aligned}
&\det(I-K^{\textmd{Ai}})_{L^2(sc,sb)}=1+\sum_{k=1}^{\infty}\frac{(-1)^k}{k!}\\
   &\times \int_{x_1,\dots,x_k \in(0,2)}\det\left(K^{\textmd{sine}}_{t_0\sqrt{|c|}}(x_i,x_j)+\bigO\left(\frac{ t_0^2}{s^{3/2}}\right)\right)_{i,j=1}^k dx_1\cdots dx_k.
\end{aligned}
\ee

Let $r_j$ denote the $j$'th row of $\left(K^{\textmd{sine}}_{t_0\sqrt{|c|}}(x_i,x_j)\right)_{i,j=1}^k$ and let $e_j$ denote the $j$'th row of the error matrix, so that $e_j=\bigO\left(\frac{t_0^2}{s^{3/2}}\right)$, $j=1,\dots,k$. Then we have 
\begin{equation*}
\begin{aligned}
&\det\left(K^{\textmd{sine}}_{t_0\sqrt{|c|}}(x_i,x_j)+\bigO\left(\frac{t_0^2}{s^{3/2}}\right)\right)_{i,j=1}^k=\det\begin{pmatrix}r_1+e_1\\ r_2+e_2\\ .. \\ r_k+e_k\end{pmatrix}=\\
&=\det \begin{pmatrix}r_1\\ r_2\\ ..\\ r_k\end{pmatrix}+\sum_{j=1}^k \det \begin{pmatrix}r_1\\ .. \\ r_{j-1}\\ e_j \\ r_{j+1}+e_{j+1}\\.. \\ r_k+e_k\end{pmatrix}=\det\left(K^{\textmd{sine}}_{t_0\sqrt{|c|}}(x_i,x_j)\right)_{i,j=1}^k+\frac{t_0^2}{s^{3/2}}\sum_{j=1}^k \det \begin{pmatrix}r_1\\ .. \\ r_{j-1}\\ \bigO(1) \\ r_{j+1}+e_{j+1}\\.. \\ r_k+e_k\end{pmatrix}.
\end{aligned}
\end{equation*}
Note that the determinants in the sum over $j$ on the right hand side may be estimated by Hadamard's inequality.
 Let $v_{\ell}$ denote the rows of a matrix. If all the matrix elements $\| v_{\ell\,m} \| \le C$,
\be\label{HadamardINEQ}
\lvert \det\left(v_1,v_2,\dots,v_k\right)^T \rvert \leq \prod_{\ell=1}^k \| v_{\ell} \| \leq C_1^k k^{k/2} \leq (100\cdot C_1)^k \sqrt{k!}. 
\ee
Using uniformity of the error term and the estimate 
\be\label{Ksineest}
\left| K^{\textmd{sine}}_{t_0\sqrt{|c|}}(x_i,x_j)\right| \le C_1 t_0,
\ee
we obtain
\be \label{RKSEstimate}
\begin{aligned}
\Biggl \lvert &\det\left(K^{\textmd{sine}}_{t_0\sqrt{|c|}}(x_i,x_j)+\bigO\left(\frac{t_0^2}{s^{3/2}}\right)\right)_{i,j=1}^k -\det\left(K^{\textmd{sine}}_{t_0\sqrt{|c|}}(x_i,x_j)\right)_{i,j=1}^k\Biggr\rvert \leq \\
&\leq  \frac{t_0^2}{s^{3/2}}k (C_2t_0)^k \sqrt{k!}\leq \frac{t_0^2}{s^{3/2}}(C_3t_0)^{k}\sqrt{k!}.
\end{aligned}
\ee

This estimate and \eqref{42} imply (note that $\det(I-K^{\textmd{sine}})_{L^2(0,2\lvert c \rvert^{1/2}t_0)}=
\det(I-K^{\textmd{sine}}_{t_0\sqrt{|c|}})_{L^2(0,2)}$)
\be\label{finalestlemma}
\begin{aligned}
&\left| \det(I-K^{\textmd{Ai}})_{L^2(sc,sb)}- \det(I-K^{\textmd{sine}})_{L^2(0,2\lvert c \rvert^{1/2}t_0)}\right|\\
&\le
\frac{t_0^2}{s^{3/2}}\sum_{k=1}^{\infty}\frac{(C_1 t_0)^k}{\sqrt{k!}}\le
\frac{t_0^2}{s^{3/2}}\left(\sum_{k=1}^{\infty}\frac{(C_1t_0)^{2k}k^2}{k!}\right)^{1/2}
\left(\sum_{k=1}^{\infty}\frac{1}{k^2}\right)^{1/2}\le C_0 \frac{e^{C_2 t_0^2}}{s^{3/2}}.
\end{aligned}
\ee
Since by \eqref{AirydetoneGAP},
\[
\left|\det(I-K^{\textmd{Ai}})_{L^2(-t_1,+\infty)}\right|\le C_0,
\]
the estimate \eqref{finalestlemma} implies the statement of the proposition.
\end{proof}

On the other hand, the difference between the determinants corresponding to the kernels 
$K^{\textmd{Ai}}$ and $\widehat{K^{\textmd{Ai}}}$ is small:

\begin{prop}\label{AIRYANDHATEST} There exist constants $C_2,C_3>0$ depending only on $c$ such that
\be
\Biggl \lvert \det(I-K^{\textmd{Ai}})_{L^2(sJ)}-\det(I-\widehat{K^{\textmd{Ai}}})_{L^2(sJ)} \Biggr \rvert \leq  
\frac{C_2}{s^{3/4}}e^{C_3 t_1^2\max\{t_0^2,t_1\}}.
\ee
\end{prop}
\begin{proof}
We use another representation of the Airy-kernel,
\[
K^{\textmd{Ai}}(z,z')=\int_0^{\infty} \text{Ai}(z+x)\text{Ai}(z'+x)dx,
\]
and the asymptotics of the Airy function at $+\infty$, see, e.g., \cite{AS}, 
\be\label{Aiinfty}
    \textmd{Ai}(z)=\frac{C_0}{z^{1/4}}\exp\left(-\frac{2}{3}z^{3/2}\right)\left[1+\bigO\left(\frac{1}{z^{3/2}}\right)\right],\qquad z\to +\infty,
\ee
and also the arguments similar to those in Proposition \ref{Ksineinternal}
to conclude that
\be
K^{\textmd{Ai}}(z,z')=\kappa(z,z')f(z)f(z'),\qquad z,z'\in sJ,
\ee
where
\be\label{53est}
|\kappa(z,z')|\le C_0 t_1^{1/2},\qquad |f(z)|\le 
\begin{cases}
C_1,& z\in(-t_1,t_1)\cr
C_1 e^{-C_2 z^{3/2}},& z\in (t_1,\infty)
\end{cases},\qquad z,z'\in (sa,+\infty).
\ee
Moreover, on the interval $(sc,sb)$,
\be
|\kappa(z,z')|\frac{t_0}{s^{1/2}}\le C_0 t_0,\qquad |f(z)|\le C_1,\qquad z,z'\in (sc,sb).
\ee

We similarly have
\be
\widehat{K^{\textmd{Ai}}}(z,z')=\widehat{\kappa}(z,z')f(z)f(z'),\qquad z,z'\in sJ,\qquad |\widehat{\kappa}(z,z')|\le 
\begin{cases}
C_1 s^{1/2} ,& z,z'\in(sc,sb)\cr
C_1 t_1^{1/2},& z,z'\in  (sa,+\infty)
\end{cases}.
\ee

Consider the case when $z\in (sc,sb)$ and $z^{\prime}\in (sa,+\infty)$. Here we have that $z-z^{\prime}>(a-b)s$. 
By representation \eqref{airykernel1}, it follows, from (\ref{airysine1}), \eqref{Aiinfty}, and the asymptotics for the derivatives, that
\begin{equation}\label{86kernelestimate}
    |\kappa(z,z^{\prime})|\le\frac{C_1}{s^{3/4}},\qquad z\in (sc,sb),\quad z^{\prime}\in (sa,+\infty).
\end{equation}

After making the change of variable \eqref{changevar} in the integrals over $(cs,bs)$, 
we can use Proposition 12 in \cite{IB} (which is based on Hadamard's inequality) to estimate the difference
$|\det\left(\kappa(z_j,z_k)\right)_{j,k=1}^{m+1}-\det\left(\widehat{\kappa}(z_j,z_k)\right)_{j,k=1}^{m+1}|$ below and obtain
\be
\begin{aligned}
&\left| \det(I-K^{\textmd{Ai}})_{L^2(sJ)}-\det(I-\widehat{K^{\textmd{Ai}}})_{L^2(sJ)}\right|\\
&\le \sum_{m=0}^{\infty} \frac{1}{(m+1)!}
   \int_{z_1,\dots,z_{m+1} \in sJ}
   \left| \det\left(\kappa(z_i,z_j)\right)_{i,j=1}^{m+1}- \det\left(\widehat{\kappa}(z_i,z_j)\right)_{i,j=1}^{m+1}\right| f(z_1)^2dz_1\cdots f(z_{m+1})^2dz_{m+1}\\
   &=\sum_{m=0}^{\infty}\sum_{k=0}^{m+1} \frac{(-1)^{m+1}}{(m+1-k)!k!}\\
   &\times \int_{\begin{matrix}z_1,\dots,z_k \in(cs,bs) \\ z_{k+1},\dots,z_{m+1}\in(as,+\infty)\end{matrix}} 
   \left| \det\left(\kappa(z_i,z_j)\right)_{i,j=1}^{m+1}- \det\left(\widehat{\kappa}(z_i,z_j)\right)_{i,j=1}^{m+1}\right| f(z_1)^2dz_1\cdots f(z_{m+1})^2dz_{m+1} \\    
  &\le \frac{C_0}{s^{3/4}}\sum_{m=0}^{\infty} \frac{(C_1t_1\max\{t_0,t_1^{1/2}\})^{m}}{\sqrt{m!}}\le \frac{C_2}{s^{3/4}}e^{C_3 t_1^2\max\{t_0^2,t_1\}}.
\end{aligned}
\ee
Indeed, by the estimates \eqref{53est}--\eqref{86kernelestimate} and Proposition 12 in \cite{IB}, if $z_1,\dots,z_k\in (cs,bs)$,
\be
\left|\det\left(\kappa(z_i,z_j)\right)_{i,j=1}^{m+1}-\det\left(\widehat{\kappa}(z_i,z_j)\right)_{i,j=1}^{m+1}\right|\left(\frac{t_0}{s^{1/2}}\right)^k
\le\frac{C_0}{s^{3/4}}(C_1\max\{t_0,t_1^{1/2}\})^m\sqrt{m!}.
\ee

\end{proof}

Combining the statements of Propositions \ref{HATANDPRODEST} and \ref{AIRYANDHATEST} we obtain
\be
\begin{aligned}
&\Biggl\lvert \det(I-K^{\textmd{Ai}})- \det(I-K^{\textmd{sine}})_{L^2(0,2\lvert c \rvert^{1/2}t_0)}\det(I-K^{\textmd{Ai}})_{L^2(-t_1,+\infty)}\Biggr\rvert \leq\\
& \leq \frac{C_1}{s^{3/2}}\exp\left(C_2 t_0^2\right)+\frac{C_3}{s^{3/4}}\exp(C_4t_1^2\max\{t_0^2,t_1\}).
\end{aligned}
\ee
We may thus choose, for example, $t_0=t_0(s)=t_1=t_1(s)=\log(s)^{1/8}$, which proves the lemma.

\section{Differential Identity}\label{secdiffid}
In this section we express, for $p=a,b,c$, the derivative $\frac{d}{dp}\log \det (I-K^{\mathrm{Ai}})$ in terms of a solution of a certain Riemann-Hilbert problem (RHP). 
First let us write the kernel of the operator $K^{\mathrm{Ai}}$ in the form
\begin{equation}
    K^{\mathrm{Ai}}(z,z^{\prime})=\frac{\vec{f}(z)^T\vec{g}(z^{\prime})}{z-z^{\prime}}, \qquad \vec{f}(z)=\begin{pmatrix}\mathrm{Ai}(z) \\ \mathrm{Ai}^{\prime}(z)\end{pmatrix},\quad \vec{g}(z)=\begin{pmatrix}\mathrm{Ai}^{\prime}(z) \\ -\mathrm{Ai}(z)\end{pmatrix}.
\end{equation}
Note that $\sum_{k=1}^2 f_k(z)g_k(z)=0$.
The operators of this form are known as integrable operators. They possess the following crucial properties (Lemmata 2.8 and 2.12 in \cite{DIZ}).

The resolvent $R_s$ of the operator $K^{\mathrm{Ai}}$ is given by
\begin{equation}\label{resolvent}
    (I-K^{\mathrm{Ai}})^{-1}=I+R_s,
\end{equation}
where $R_s$ can be expressed as 
\begin{equation}\label{resolventexpr}
    R_s(z,z^{\prime})=\frac{\vec{F}^T(z)\vec{G}(z^{\prime})}{z-z^{\prime}},\qquad \vec{F}(z)=(I-K^{\mathrm{Ai}})^{-1}\vec{f}(z),\quad 
    \vec{G}(z)=(I-K^{\mathrm{Ai}})^{-1}\vec{g}(z),
\end{equation}
and $\sum_{k=1}^2 F_k(z)G_k(z)=0$. 
Furthermore,
\be\label{mFG}
\vec{F}(z)=m_+(z)\vec{f}(z),\qquad \vec{G}(z)=(m^{-1}_+(z))^T\vec{g}(z),
\ee
where $m(z)$ is the $2\times 2$ matrix, the solution to the following RHP:

\noindent
(a) $m(z) \text{ is analytic in } \mathbb{C} \setminus [c,b]\cup [a,\infty]$.

\noindent (b)
The function $m(z)$ possesses $L^2$ boundary values $m_+$ and $m_-$ as $z$ approaches 
$ (c,b)\cup (a,\infty)$ from above and below, respectively. These boundary values are related by the condition
\be
m_{+}(z)=m_{-}(z)V^{-1}(sz),\qquad z\in (c,b )\cup (a ,\infty),
\ee
where the jump matrix
\begin{equation}V^{-1}(z)=I-2\pi i \vec{f}(z)\vec{g}(z)^T=
\begin{pmatrix}1-2\pi i \mathrm{Ai}(z)\mathrm{Ai}^{'}(z) & 2 \pi i\mathrm{Ai}(z)^2 \\ -2 \pi i \mathrm{Ai}^{'}(z)^2 & 1+ 2\pi i\mathrm{Ai}(z)\mathrm{Ai}^{'}(z) \end{pmatrix}.
\end{equation}

\noindent (c)  $m(z)$ satisfies
\be
m(z) = I+\bigO\left(\frac{1}{z}\right), \qquad z\to \infty.
\ee

Before proceeding with the derivation of the differential identity, we will now reduce this RHP to another one, 
with constant jumps. For that, we need the following model problem.

\subsection{Airy model RH problem}
Let
\be
y_j(z)=\omega^j \text{Ai}(\omega^j z),\qquad \omega =e^{\frac{2\pi i}{3}},\qquad j=0,1,2.
\ee
In each region $I-IV$ as per Figure \ref{Fig1},
\begin{figure}
\begin{center}
    \begin{tikzpicture}
    \draw[thick,middlearrow={>}] (-2,3.46)--(0,0) node[anchor=north west] {$0$} ;
     \node[above] at (4.5, 0)  (c)     {$\Gamma$};
     \node[above] at (1.5, 2)  (c)     {$I$};
     \node[above] at (1.5, -2)  (c)     {$IV$};
     \node[above] at (-2.5, 2)  (c)     {$II$};
     \node[above] at (-2.5, -2)  (c)     {$III$};
    \draw[thick,middlearrow={>}] (-2,-3.46)--(0,0) ;
    \draw[thick,middlearrow={>}] (-4,0)--(0,0) ;
    \draw[thick,middlearrow={>}] (0,0)--(4,0) ;
    \end{tikzpicture}
    \caption{Jump contour for the $\Phi$-RH problem.}\label{Fig1}
\end{center} 
\end{figure}
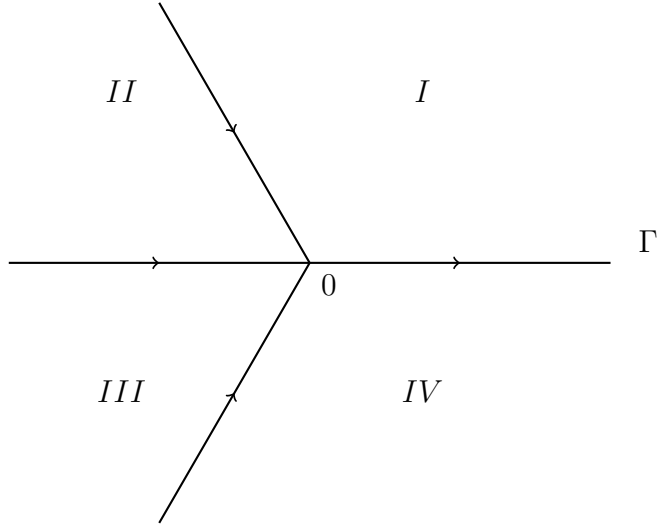
consider the function $\Phi$ defined by
\begin{equation}
\Phi(z)=\sqrt{2 \pi}e^{-\frac{\pi i}{4}} \begin{cases}\begin{pmatrix}y_0 & -y_2 \\ y_0' & -y_2'\end{pmatrix} \text{, for } z \in I, \\\begin{pmatrix}-y_1 & -y_2 \\ -y_1' & -y_2'\end{pmatrix} \text{, for } z \in II, \\ \begin{pmatrix}-y_2 & y_1 \\ -y_2' & y_1'\end{pmatrix} \text{, for } z \in III, \\ \begin{pmatrix}y_0 & y_1 \\ y_0' & y_1'\end{pmatrix} \text{, for } z \in IV. \\\end{cases}
\end{equation}
The pre-factor normalises the determinant to unity at all points. For $\Gamma=\mathbb{R}\cup \mathbb{R}^{+}e^{\pm\frac{2\pi i}{3}}$ 
(see Figure \ref{Fig1}) let
\begin{equation}
v_{\Gamma}(z)=\begin{cases}\begin{pmatrix}1&0\\1&1\end{pmatrix} \text{, } z\in \mathbb{R}^{+}e^{\pm \frac{2\pi i}{3}},\\ \begin{pmatrix}0&1\\-1&0\end{pmatrix} \text{, } z\in (-\infty,0),\\ \begin{pmatrix}1&1\\0&1\end{pmatrix} \text{, } z\in (0,\infty), \end{cases}
\end{equation}
it is known, see e.g. \cite{CIK}, \cite{Dbook}, that $\Phi$ satisfies the following RHP with $L^2$ boundary values:
\begin{equation}\label{PHIRHP}
\begin{aligned}
&\Phi(z)\;\mathrm{ is\; analytic\; in\; } \mathbb{C}\setminus \Gamma ,\\ &\Phi_{+}(z)=\Phi_{-}(z)v_{\Gamma}(z),\qquad z \in \Gamma\setminus \{0\}, \\ &\Phi(z) =z^{-\frac{1}{4}\sigma_3}N_0\left(I+\begin{pmatrix}1&6i\\6i&-1\end{pmatrix}\frac{1}{48z^{3/2}}+\bigO\left(\frac{1}{z^{3}}\right)\right)e^{-\frac{2}{3}z^{\frac{3}{2}}\sigma_3},\quad z \to \infty, \\ &\Phi(z) =O(1),\qquad z \to 0,
\end{aligned}
\end{equation}
where $N_0$ is given by
\begin{equation}\label{N0def}
    N_0=\frac{1}{\sqrt{2}}\begin{pmatrix}1 & 1 \\ -1 & 1 \end{pmatrix}e^{-\frac{i \pi}{4}\sigma_3}.
\end{equation}
The calculation to verify that $\Phi$ indeed satisfies (\ref{PHIRHP}) rests upon the well-known facts
\[\text{Ai}(z)+e^{\frac{2\pi i}{3}}\text{Ai}(ze^{\frac{2\pi i}{3}})+e^{-\frac{2\pi i}{3}}\text{Ai}(ze^{-\frac{2\pi i}{3}})=0,\] and
\[\text{Ai}(z)=\frac{z^{-1/4}e^{-\frac{2}{3}z^{3/2}}}{2\sqrt{\pi}}\left(1-\frac{5}{48z^{3/2}}+\bigO\left(\frac{1}{z^3}\right)\right),\qquad z\to \infty.\]
For further details we refer the reader to p.216 in \cite{Dbook}.

The reason for introducing $\Phi$ is because the jump, $V^{-1}(sz)$, of the $m$-RHP may be factorised in terms of 
$\Phi(sz)$ as follows:
\begin{equation}\label{Vfact}
V^{-1}(sz)=
\begin{cases}
\Phi_{+}(sz)\begin{pmatrix}1&-1\\0&1\\ \end{pmatrix} \Phi_{+}^{-1}(sz),\qquad z \in (0,\infty), \\
\Phi_{+}(sz)\begin{pmatrix}2&-1\\1&0\\ \end{pmatrix} \Phi_{+}^{-1}(sz),\qquad z \in (c,b)\cup(a,0).
\end{cases}
\end{equation}

\subsection{RH problem for $X$}
Making use of the Airy model problem, we transform 
 the $m$-RHP into a form with constant jump matrices. We define the matrix $X(z)$ in each region of $\mathbb{C}$ as per Figure \ref{Fig2}. 
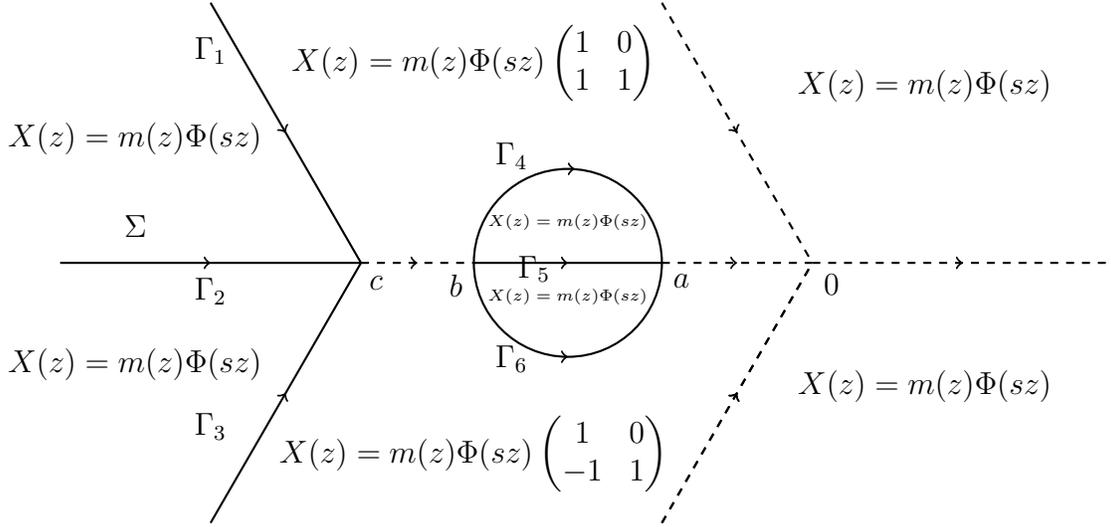
\begin{figure}
    \begin{center}
    \begin{tikzpicture}

   \node[above] at (-9, 0.2)  (c)     {$\Sigma$};
     \node[above] at (1.5, 2)  (c)     {$X(z)=m(z)\Phi(sz)$};
     \node[above] at (1.5, -2)  (c)     {$X(z)=m(z)\Phi(sz)$};
     \node[above] at (-4.5, 2)  (c)     {$X(z)=m(z)\Phi(sz)\begin{pmatrix}1& 0 \\ 1& 1\end{pmatrix}$};
     \node[above] at (-4.5, -3.2)  (c)     {$X(z)=m(z)\Phi(sz)\begin{pmatrix}1& 0 \\ -1& 1\end{pmatrix}$};     

      \node[above] at (-5.8, -0.5)  (c)     {$c$};
      
      \node[above] at (-9, 1.3)  (c)     {$X(z)=m(z)\Phi(sz)$};
      \node[above] at (-9, -1.7)  (c)     {$X(z)=m(z)\Phi(sz)$};

       \node[above] at (-3.25, 0.3)  (c)     {\tiny $X(z)=m(z)\Phi(sz)$};
      \node[above] at (-3.25, -0.7)  (c)     {\tiny $X(z)=m(z)\Phi(sz)$};
      
    \node[above] at (-8, 2.5)  (c)     {$\Gamma_1$};  
    \node[above] at (-8, -0.7)  (c)     {$\Gamma_2$};
    \node[above] at (-8, -2.5)  (c)     {$\Gamma_3$}; 
      
    \node[above] at (-4, 1.1)  (c)     {$\Gamma_4$};  
    \node[above] at (-3.7, -0.4)  (c)     {$\Gamma_5$};
    \node[above] at (-4, -1.6)  (c)     {$\Gamma_6$};

        \draw[thick,middlearrow={>}] (-8,3.46)--(-6,0) ;
        \draw[thick,middlearrow={>}] (-8,-3.46)--(-6,0);
        \draw[thick,middlearrow={>}] (-10,0)--(-6,0);
        \draw[thick,dashed,middlearrow={>}] (-6,0)--(-4.5,0) node[anchor=north east] {$b$}; 
        \draw[thick,middlearrow={>}] (-4.5,0)--(-2,0) node[anchor=north west] {$a$}; 
        
         \draw[thick,middlearrow={<}] (-2,0) arc (0:180:1.25);
        \draw[thick,middlearrow={>}] (-4.5,0) arc (180:360:1.25);

        \draw[thick,dashed,middlearrow={>}] (-2,0)--(0,0) ;

    \draw[thick,dashed,middlearrow={>}] (-2,-3.46)--(0,0) ;
    \draw[thick,dashed,middlearrow={>}] (-2,-3.46)--(0,0) ;
    \draw[thick,dashed,middlearrow={>}] (0,0)--(4,0) ;
    \draw[thick,dashed,middlearrow={>}] (-2,3.46)--(0,0) node[anchor=north west] {$0$} ;

    \end{tikzpicture}
     \caption{Definition of $X$ in various regions.}\label{Fig2}
\end{center} 
\end{figure}

Using \eqref{Vfact} and the $m$ and $\Phi$-RH problems, we find that $X(z)$ satisfies the following problem, where 
$\Sigma$ denotes the union of bold curves on Figure \ref{Fig2}.
\begin{equation}\label{XRHP}
\begin{aligned}
&X(z)\;\mathrm{ is\; analytic\; in\; } \mathbb{C}\setminus \Sigma,\\ 
&X_{+}(z)=X_{-}(z)\begin{pmatrix}0&1\\-1&0\\ \end{pmatrix}, \mathrm{\; for\; } z \in \Gamma_2 \cup \Gamma_5=(-\infty,c)\cup(b,a), \\
&X_{+}(z)=X_{-}(z)\begin{pmatrix}1&0\\1&1\\ \end{pmatrix}, \mathrm{\; for\; } z \in  \Gamma_1\cup \Gamma_3 \cup \Gamma_4 \cup \Gamma_6,\\  &X(z) =m(z)(sz)^{-\frac{1}{4}\sigma_3}N_0\left(I+\begin{pmatrix}1&6i\\6i&-1\end{pmatrix}
\frac{1}{48(sz)^{3/2}}+\bigO\left(\frac{1}{(sz)^{3}}\right)\right)e^{-\frac{2}{3}(sz)^{\frac{3}{2}}\sigma_3},\quad  z \to \infty.
\end{aligned}
\end{equation}

\subsection{The Identity}
We now proceed with the derivation of the differential identity.
Consider the case when $p=b$, the identities at the points $a$, $c$ are obtained similarly.
\begin{equation}\label{derdi}
    \frac{d}{db}\log P^{\mathrm{Ai}}(sJ)=\mathrm{tr} (I-K^{\mathrm{Ai}})^{-1}\frac{dK^{\mathrm{Ai}}}{db}=-((I-K^{\mathrm{Ai}})^{-1}K^{\mathrm{Ai}})(b,b)=-R_s(b,b).
\end{equation}
where, for the final equality, we used (\ref{resolvent}). The negative sign at $R_s(b,b)$ comes from the fact that $b$ is the lower limit of integration. (For $p=a,c$, we have the opposite sign.) Now, by (\ref{resolventexpr}), we may write
\be
    R_s(p,p)  =\lim_{z,z^{\prime}\to p}\left(\frac{F_1(z)G_1(z^{\prime})+F_2(z)G_2(z^{\prime})}{z-z^{\prime}}\right)=-
    \lim_{z\to p}\left(F_1(z)G_1^{\prime}(z)+F_2(z)G_2^{\prime}(z)\right).\label{Rdi}
\ee

For definiteness, assume that the limit is taken from outside the lens and with $\Im z>0$.
By \eqref{mFG}, the definition of $X$ and that of $\Phi$ in sector II,
\be\begin{aligned}
\vec{F}(z)&=m_+(z)\vec{f}(z)=
X_+(z)\begin{pmatrix} 1 & 0 \cr -1 & 1\end{pmatrix}\Phi_+^{-1}(sz)\begin{pmatrix} y_0(sz)\cr y_0'(sz)\end{pmatrix}\\
&=\sqrt{2\pi}e^{-i\pi/4}X_+(z)\begin{pmatrix} 1 & 0 \cr -1 & 1\end{pmatrix}
\begin{pmatrix} -y_2'(sz) & y_2(sz) \cr y_1'(sz) & -y_1(sz)\end{pmatrix}\begin{pmatrix} y_0(sz)\cr y_0'(sz)\end{pmatrix},
\end{aligned}\label{Fprelim}
\ee
where we used the fact that $\det\Phi(z)\equiv 1$. Using it once again, we reduce \eqref{Fprelim} to
\be\label{Fdi}
\vec{F}(z)=\frac{1}{\sqrt{2\pi}e^{-i\pi/4}}X_+(z)\begin{pmatrix} 1\cr 0\end{pmatrix}=\frac{1}{\sqrt{2\pi}e^{-i\pi/4}}
\begin{pmatrix} X_{11+}(z)\cr X_{21+}(z)\end{pmatrix}.
\ee

Similarly, we obtain
\be\label{Gdi}
\begin{aligned}
\vec{G}(z)&=(m_+^{-1}(z))^T\vec{g}(z)=
\sqrt{2\pi}e^{-i\pi/4}
(X_+^{-1}(z))^T\begin{pmatrix} 1 & 1 \cr 0 & 1\end{pmatrix}
\begin{pmatrix} -y_1(sz) & -y_1'(sz) \cr -y_2(sz) & -y_2'(sz)\end{pmatrix}\begin{pmatrix} y_0'(sz)\cr -y_0(sz)\end{pmatrix}\\
&=\frac{1}{\sqrt{2\pi}e^{-i\pi/4}}(X_+^{-1}(z))^T\begin{pmatrix} 0\cr -1\end{pmatrix}=
\frac{1}{\sqrt{2\pi}e^{-i\pi/4}}\begin{pmatrix} X_{21+}(z)\cr -X_{11+}(z)\end{pmatrix}.
\end{aligned}
\ee

Substituting \eqref{Fdi} and \eqref{Gdi} into \eqref{Rdi}, we finally obtain
\begin{equation}
     R_s(b,b)=\frac{1}{2\pi i}\lim_{z\to b}\left(X_{11}X_{21}^{\prime}-X_{21}X_{11}^{\prime} \right)(z)=
     \frac{1}{2\pi i}\lim_{z\to p}\left(X^{-1}(z)X^{\prime}(z)\right)_{21},
\end{equation}
where the limit is taken from outside the lens with $\Im z>0$;
and therefore by \eqref{derdi}
\be
\frac{d}{db}\log P^{\mathrm{Ai}}(sJ)=-R_s(b,b)=-\frac{1}{2\pi i}\lim_{z\to b}\left(X^{-1}(z)X^{\prime}(z)\right)_{21}.
\ee
At the points $a$ and $c$, we obtain the same result but with the opposite sign. 

Thus we have

\begin{lemma}[Differential identity]\label{DIFFID}
The Fredholm determinant (\ref{AIRYDET}) satisfies:
\begin{equation}\label{diffid}
    \frac{d}{dp}\log P^{\textmd{Ai}}(sJ)=\pm \frac{1}{2\pi i}\lim_{z\to p}\left(X^{-1}(z)X^{\prime}(z)\right)_{21},
\end{equation}
where the $+$ sign is taken if $p=a,c$ and the $-$ sign, if $p=b$. The limit is taken from outside the lens and with $\Im z>0$.
\end{lemma}

\section{Solution to the Riemann-Hilbert problem for $X$}\label{secRHP}
To solve the $X$-RHP for large $s$ we, as usual, apply a series of transformations. The approach, known as the steepest descent method for Riemann-Hilbert problems, was first introduced by Deift and Zhou \cite{DZ} and used and developed in many subsequent works.
The first step is to normalise the exponential behaviour at $\infty$ by multiplying from the right by a suitable function. This process is set up as follows.

\subsection{$g$-function and the RH problem for $S$}
In the introduction, we defined
\[
g(z)=\int_{a}^z\frac{q(\zeta)}{p(\zeta)^{1/2}}d\zeta, \qquad \zeta\in\mathbb{C}\setminus(-\infty,a].
\]
Here $p(z)=(z-a)(z-b)(z-c)$, the branch of the root is chosen positive for positive arguments, and the branch cut is on
$(-\infty,c)\cup(b,a)$, $c<b<a<0$; $q(z)=z^2+q_1 z + q_0$ is the second degree polynomial. 
We require that $g(z)$ satisfies conditions \eqref{gcond1} and \eqref{gcond2},
which as we now show determine the coefficients $q_0$, $q_1$.

We have
\begin{lemma}\label{lemmag}
The function
$g(z)$ is analytic in $\mathbb{C}\setminus(-\infty,a]$ and 
satisfies the jump conditions:
 \begin{equation}\label{gjumps}
 g_{+}(z)+g_{-}(z) =0, \quad z \in (-\infty,c)\cup(b,a), \qquad g_{+}(z)=g_{-}(z) +2g_+(b),  \quad z \in (c,b).
\end{equation}

The polynomial $q(z)=z^2+q_1 z + q_0$ has coefficients given by \eqref{defq1}.

As $z\to\infty$,
\begin{equation}\label{gasymptotics}
  g(z)=\frac{2}{3}z^{3/2}+\frac{\alpha_1}{z^{1/2}}+\frac{\alpha_2}{z^{3/2}}+\bigO(z^{-5/2}),
\end{equation}
where \[\alpha_1=\frac{1}{2}(ab+ac+bc)-\frac{1}{4}(a^2+b^2+c^2)-2q_0,\] and  
\be\label{alpha2}
\alpha_2=-\frac{1}{12}\left(a^3+b^3+c^2-(a+b)(a+c)(b+c)+4(a+b+c)q_0\right).\ee
\end{lemma}

\begin{proof}
Since  $\sqrt{p(z)}_++\sqrt{p(z)}_-=0$, on $(-\infty,c)\cup(b,a)$, it follows from (\ref{gcond2}) that
\[g_+(z)+g_-(z)=0,\quad \mathrm{ for }\quad z\in (-\infty,c)\cup(b,a).\] On the other hand, for $ z\in(c,b)$, we have
\[g_+(z)-g_-(z)=2\int_{a}^{b} \frac{q(\zeta)}{p(\zeta)^{1/2}_+}=2g_+(b)=2g_+(c),\]
where the last equality follows from (\ref{gcond2}). Thus we have \eqref{gjumps}.

Write the expansion at infinity
\[
\frac{q(z)}{\sqrt{p(z)}}=z^{1/2}\left(1+\frac{1}{z}\left(q_1+\frac{a+b+c}{2}\right)+\bigO\left(\frac{1}{z^2}\right)\right).
\]
Setting 
\[
q_1=-\frac{a+b+c}{2}
\]
and integrating, we obtain
\be
g(z)=\frac{2}{3}z^{3/2}+g_0+\bigO(z^{-1/2}),
\ee
For some constant $g_0$. It follows from the jump condition $g_+(z)+g_-(z)=0$ on the half-line $(-\infty,c)$ that $g_0=0$.
Thus, condition \eqref{gcond1} holds with $q_1$ given in \eqref{defq1}.
The first expression in \eqref{defq1},
\[
q_0=-\frac{J_2+q_1J_1}{J_0}
\]
follows immediately from \eqref{gcond2}. Next, taking the integral over the $B_1$ cycle (see Figure \ref{Figcycles})
on the Riemann surface of $p(z)^{1/2}$,
we have
\begin{equation}\label{newq0}
    0=- \oint_{B_1} \frac{d}{dz}\sqrt{p(z)}dz=\int_c^b \frac{p^{\prime}(z)dz}{\sqrt{p(z)}}=3J_2+4q_1J_1+(ab+ac+bc)J_0.
\end{equation}
This allows one to write $J_2$ in terms of $J_0$, $J_1$, and thus we obtain the second equality in \eqref{defq1}.

A straightforward series expansion of $g(z)$ at infinity verifies the values of $\alpha_1$, $\alpha_2$. 
\end{proof}

Now set
\begin{equation}\label{definitionofS}
S(z)=X(z)e^{s^{3/2}g(z)\sigma_3}.
\end{equation}
It follows from the lemma above and the $X$-RH problem, in (\ref{XRHP}), that $S$ satisfies the following problem. For $\Sigma$ as in Figure \ref{FigS}, we have
\subsection*{RH problem for $S$}
\begin{equation}\label{SRHP}
\begin{aligned}
&S(z)\;\mathrm{ is\; analytic\; in\; } \mathbb{C}\setminus \Sigma,\\ 
&S_{+}(z)=S_{-}(z)\begin{pmatrix} 0&1\\ -1&0 \end{pmatrix},\qquad z \in (-\infty,c)\cup (b,a), \\
&S_{+}(z)=S_{-}(z)e^{2s^{3/2}g_+(b)\sigma_3},\qquad z \in (c,b), \\
&S_{+}(z)=S_{-}(z)\begin{pmatrix} 1&0\\ e^{2s^{3/2}g(z)}&1 \end{pmatrix},\qquad z \in  \Gamma_1 \cup \Gamma_3 \cup \Gamma_4 \cup \Gamma_6, \\
&S(z)=s^{-\frac{1}{4}\sigma_3}
\left(S_0+\bigO\left(\frac{1}{z}\right)\right)z^{-\frac{1}{4}\sigma_3}N_0 \mathrm{, }\quad z \to \infty,\qquad\mbox{where}\quad
S_0=\begin{pmatrix}
1 & 0 \cr -\alpha_1 s^{3/2} & 1
\end{pmatrix}.
\end{aligned}
\end{equation}
The expansion of $S(z)$ at infinity is found using (\ref{gasymptotics}) and the expansions of $\Phi$ and $m$ at infinity. 
Up to the error term, it is the same in all sectors. 

We now construct approximate solutions (parametrices) for this $S$-RH problem: outside parametrix away from the points $a$, $b$, $c$, and local parametrices around these points, see Figure \ref{Us}. 
This allows us to construct an asymptotic solution to the $S$-problem,
and therefore to the $X$-problem.
\begin{figure}
    \begin{center}
    \begin{tikzpicture}

   \node[above] at (-9, 0.2)  (c)     {$\Sigma$};

      \node[above] at (-5.8, -0.5)  (c)     {$c$};

    \node[above] at (-8, 2.5)  (c)     {$\Gamma_1$};  
    \node[above] at (-8, -0.7)  (c)     {$\Gamma_2$};
    \node[above] at (-8, -2.5)  (c)     {$\Gamma_3$}; 
      
    \node[above] at (-4, 1.1)  (c)     {$\Gamma_4$};  
    \node[above] at (-3.2, -0.6)  (c)     {$\Gamma_5$};
    \node[above] at (-4, -1.6)  (c)     {$\Gamma_6$};       
       \node[below] at (0, 0)  (c)     {$0$};   
     
        \draw[thick,middlearrow={>}] (-8,3.46)--(-6,0) ;
        \draw[thick,middlearrow={>}] (-8,-3.46)--(-6,0);
        \draw[thick,middlearrow={>}] (-10,0)--(-6,0);
        \draw[thick,middlearrow={>}] (-6,0)--(-4.5,0) node[anchor=north east] {$b$}; 
        \draw[thick,middlearrow={>}] (-4.5,0)--(-2,0) node[anchor=north west] {$a$}; 
        
         \draw[thick,middlearrow={<}] (-2,0) arc (0:180:1.25);
        \draw[thick,middlearrow={>}] (-4.5,0) arc (180:360:1.25);

    \end{tikzpicture}
     \caption{Jump contour for the $S$-RH problem.}\label{FigS}
\end{center} 
\end{figure}
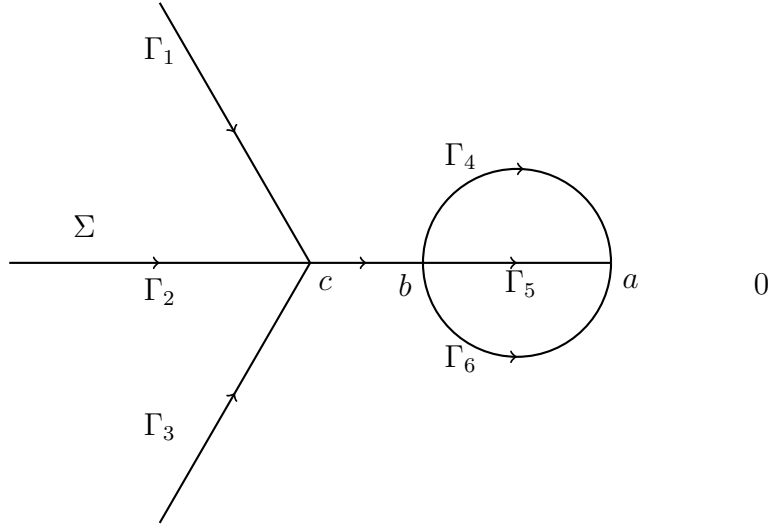

\subsection{Outside parametrix. Jacobi $\theta$-functions}
In this section we construct a parametrix for the $S$-RHP away from the points $p=a,b,c$, which models the behaviour of $S(z)$ at infinity. We proceed as follows.

\subsubsection{$\beta$ model problem}
Set 
\[\beta(z)=\left(\frac{(z-a)(z-c)}{z-b}\right)^{1/4},\]
with branch cuts along $(-\infty,c)\cup(b,a)$ and such that $\beta(z)>0$ when $z\to+\infty$. Consider the function
\begin{equation}\label{defN}
N(z)=N_0^{-1}\beta^{-\sigma_3}N_0=\begin{pmatrix} \frac{\beta(z)+\beta(z)^{-1}}{2} & \frac{\beta(z)-\beta(z)^{-1}}{2i}\\ -\frac{\beta(z)-\beta(z)^{-1}}{2i} & \frac{\beta(z)+\beta(z)^{-1}}{2}\end{pmatrix},
\end{equation}
where $N_0$ is given by (\ref{N0def}). 
The function $N(z)$ is analytic in $\mathbb{C}\setminus \left\{ (-\infty,c]\cup [b,a] \right\}$.
By definition, it is immediate that 
\be\label{Ninfty}
N(z)= N_0^{-1}\left(I+\frac{N_1}{z}+\frac{N_2}{z^2}+ \bigO\left( \frac{1}{z^3}\right)\right)z^{-\frac{1}{4}\sigma_3}N_0,\qquad z\to \infty,
\ee
for some diagonal matrices $N_{j}$, which may be written in terms of $a,b,c$. Moreover, since the boundary values of $\beta$ satisfy $\beta(z)_{+}=i\beta(z)_{-}$
for $z\in (-\infty,c)\cup (b,a)$, we have that
\begin{equation}
N_{+}(z)=N_{-}(z)\begin{pmatrix}0&1\\ -1&0\\ \end{pmatrix},\quad \mathrm{for } \quad z\in (-\infty,c)\cup (b,a). 
\end{equation}

\subsubsection{Abel map}

Recall $J_0$ and $\tau$ as defined in \eqref{Jkdefinition}  and (\ref{definitionofV}) respectively, in the introduction.  Consider the map $u(z)$ defined by
\begin{equation}\label{definitionofu}
u(z)=\int_{a}^z\omega=
\frac{1}{2J_0}\int_{a}^{z}\frac{dw}{p(w)^{1/2}},\qquad \omega(z)=\frac{dz}{2J_0p(z)^{1/2}},\qquad \mathbb{C}\setminus (-\infty, a].
\end{equation}
We note several relevant properties of $u(z)$ which will be required in defining the outside parametrix.
\begin{lemma}\label{ulemma} $u(z)$ is analytic in $\mathbb{C}\setminus (-\infty, a]$ and
has the following properties.
\begin{align}
        & u(a)=0, \quad u_{+}(b)=-\frac{1}{2}\tau, \quad u_{+}(c)=-\frac{1}{2}(1+\tau),\label{ucond0}
         \\
 & u_{+}(z)+u_{-}(z)=-1, \quad z\in (-\infty,c),\label{ucond1} \\
& u_{+}(z)=u_{-}(z)-\tau , \quad z\in (c,b),\label{ucond2}\\
& u_{+}(z)+u_{-}(z)=0,\quad z\in (b,a),\label{ucond3}
\end{align}
and \be\label{uinfty}
u(z)=-1/2+\frac{u_1}{z^{1/2}}+\frac{u_2}{z^{3/2}}+\frac{u_3}{z^{5/2}}+\bigO(z^{-7/2}), \qquad z\to \infty,\ee
where the coefficients $u_j$ may be written in terms of $a,b,c$.
\end{lemma}
\begin{proof}
It is immediate from the definition that on the + side of the cut
\[
\int_{c}^{b} \omega_+ =\frac{1}{2},\quad \mathrm{and}\quad \int_{b}^{a} \omega_+ =\frac{\tau}{2}.
\]
Considering the integration around the $A$ cycles (see Figure \ref{Figcycles}), we also note that
\[
\int_{-\infty}^{c} \omega_+ =-\frac{\tau}{2}.
\]
From here the statements of the lemma easily follow.
\end{proof}

Note that we consider $u(z)$ on the first sheet of the Riemann surface of the function $p(z)^{1/2}$ (where we have
$p(x)^{1/2}>0$, $\beta(x)>0$, for $x>a$ on the real line). The function $u(z)$ maps the full Riemann surface to a torus (where 
$\theta$-functions are defined).

\subsubsection{Theta function parametrix}
A $\theta$-function parametrix for the sine-kernel determinant on several intervals was first constructed in \cite{DIZ}.
We now modify the construction for our case.

Recall the function $\theta_3(z)$ given by (\ref{definitionoftheta3}) and the constant $\Omega$ given by (\ref{definitionofV}). Consider
 \begin{equation}\label{finfinitydef}
P^{\infty}(z;s)=N_0\begin{pmatrix} \frac{\theta_3(0)\theta_3(u(z)+s^{3/2}\Omega+d)}{\theta_3(s^{3/2}\Omega)\theta_3(u(z)+d)}\frac{\beta(z)+\beta(z)^{-1}}{2}&\frac{\theta_3(0)\theta_3(u(z)-s^{3/2}\Omega-d)}{\theta_3(s^{3/2}\Omega)\theta_3(u(z)-d)}\frac{\beta(z)-\beta(z)^{-1}}{2i}\\ -\frac{\theta_3(0)\theta_3(u(z)+s^{3/2}\Omega-d)}{\theta_3(s^{3/2}\Omega)\theta_3(u(z)-d)}\frac{\beta(z)-\beta(z)^{-1}}{2i}&\frac{\theta_3(0)\theta_3(u(z)-s^{3/2}\Omega+d)}{\theta_3(s^{3/2}\Omega)\theta_3(u(z)+d)}\frac{\beta(z)+\beta(z)^{-1}}{2}\end{pmatrix},
\end{equation}
 where $d$ a constant to be determined. We will also write $P^{\infty}(z;s)$ in the short hand notation 
 \begin{equation}\label{thetaij}
    P^{\infty}=N_0\begin{pmatrix}\theta_{11}N_{11} & \theta_{12}N_{12} \\ \theta_{21}N_{21} & \theta_{22}N_{22} \end{pmatrix},
\end{equation}
where $N_{jk}$ are the entries of $N(z)$ as defined in (\ref{defN}), and
\be\label{thetajk}
\begin{aligned}
\theta_{11}(u(z))=\frac{\theta_3(0)\theta_3(u(z)+s^{3/2}\Omega+d)}{\theta_3(s^{3/2}\Omega)\theta_3(u(z)+d)},\qquad
\theta_{12}(u(z))=\frac{\theta_3(0)\theta_3(u(z)-s^{3/2}\Omega-d)}{\theta_3(s^{3/2}\Omega)\theta_3(u(z)-d)},\\
\theta_{21}(u(z))=\frac{\theta_3(0)\theta_3(u(z)+s^{3/2}\Omega-d)}{\theta_3(s^{3/2}\Omega)\theta_3(u(z)-d)},\qquad
\theta_{22}(u(z))=\frac{\theta_3(0)\theta_3(u(z)-s^{3/2}\Omega+d)}{\theta_3(s^{3/2}\Omega)\theta_3(u(z)+d)}.
\end{aligned}
\ee

We prove:
\begin{prop}\label{pinfinityprop} $P^{\infty}$ satisfies the following RHP.
\begin{equation}\label{pinfinityrhp}
\begin{aligned}
&P^{\infty}(z;s)\;\mathrm{ is\; analytic\; in\; } \mathbb{C}\setminus  (-\infty,a],\\
&P^{\infty}_{+}(z;s)=P^{\infty}_{-}(z;s)\begin{pmatrix} 0&1\\ -1&0 \end{pmatrix},\qquad z \in (-\infty,c)\cup(b,a), \\
&P^{\infty}_{+}(z;s)=P^{\infty}_{-}(z;s)e^{2s^{3/2}g_+(b)\sigma_3},\qquad z \in (c,b), \\
&P^{\infty}(z) =N_0
\left( F_0 + \bigO(1/z)\right)N(z), \qquad z \to \infty, 
\end{aligned}
\end{equation}
where
\be
F_0=
\begin{pmatrix}
 \frac{\theta_3(0)\theta_3(1/2+s^{3/2}\Omega+d)}{\theta_3(s^{3/2}\Omega)\theta_3(1/2+d)} & 0\cr
 0 &  \frac{\theta_3(0)\theta_3(1/2+s^{3/2}\Omega-d)}{\theta_3(s^{3/2}\Omega)\theta_3(1/2+d)}
 \end{pmatrix}.
 \ee
\end{prop}
Note that $P^{\infty}(z)$ has the same jumps and behaviour at infinity (up to left-multiplication by a constant matrix), as $S(z)$. 

\begin{proof}
To check analyticity in $\mathbb{C}\setminus  (-\infty,a]$ one needs only to determine that no poles are introduced from the zeros of the theta functions in the denominators. In fact, as in \cite{DIZ},
we determine the constant $d$ so
that those which do occur, are precisely cancelled by the zeros of $N_{jk}(z)$. We claim that the function $\beta(z)-\beta(z)^{-1}$
has two zeros
\begin{equation}\label{betastars}
    \beta_1^{*} \in (c,b),\quad   \beta_2^{*} \in (a,+\infty).
\end{equation}
Indeed, we have that 
\[
\beta(z)\pm \beta(z)^{-1}=0 \implies \beta(z)^4=1 \iff (z-a)(z-c)=(z-b).\]
Let $h(z)=(z-a)(z-c)-(z-b)$.
We have 
\[
h(c)=b-c>0,\qquad
h(b)=-(a-b)(b-c)<0,\qquad
h(a)=-(a-b)<0.
\]
Since  $h(z) \to +\infty$, if  $z\to \pm\infty$, we deduce that $h$ has a zero, $\beta_1^{*}$ in $(c,b)$ and another $\beta_2^{*}$ in $(a,+\infty)$. Since $\beta(z)\pm \beta(z)^{-1}=0\iff \beta(z)^2=\pm 1$, the sign of the root (given our choice of branches) implies that 
$\beta_1^{*}$, $\beta_2^{*}$ are the zeros of $\beta(z)-\beta(z)^{-1}$, whereas $\beta(z)+\beta(z)^{-1}$ has no zeros (its zeros are on the second sheet of the Riemann surface).

It is well known (see, e.g., \cite{WW}) that $\theta_3(z)$ has only one zero at $z=\frac{1}{2}+\frac{\tau}{2}$ modulo its period lattice
$\Lambda=\mathbb{Z}+\tau\mathbb{Z}$. By Lemma \ref{ulemma}, we easily see that $u(z)$ maps $\mathbb{C}\setminus (-\infty,a]$ onto 
$(-1/2,0)\times (-\tau/2,\tau/2)$. 
Therefore $\theta(u(z))=0 \iff z=c$.
Define then
\begin{equation}\label{definitionofd}
    d=\int_c^{\beta_1^*}\omega.
\end{equation}
It follows that
\be\theta(u(z)-d)=0 \iff z=\beta_1^{*}.\ee
This proves that, for our specific choice of $d$, the zero of $\theta(u(z)-d)$, in $P^{\infty}(z)$, at $\beta_1^{*}$ is cancelled by the zero of 
$\beta(z)-\beta(z)^{-1}$. It follows that $P^{\infty}(z)_{12}$,  $P^{\infty}(z)_{21}$ are analytic functions in $\mathbb{C}\setminus(-\infty,a]$. 
To see the analyticity of the diagonal terms $P^{\infty}(z)_{11}$,  $P^{\infty}(z)_{22}$, we note that since $0<d<1/2$ we have, using again
the mapping $u(z)$, that $u(z)+d \neq 1/2+\tau/2 \mod \Lambda$ for all $z$. 
Thus $P^{\infty}(z;s)$ is analytic in $\mathbb{C}\setminus(-\infty,a]$. 
 
We now turn to the jump conditions. The jumps of $P^{\infty}(z)$ are verified by means of (\ref{ucond1})--(\ref{ucond3}), together with the quasi-periodicity properties of $ \theta_3 (z)$ in (\ref{qprelations}) and the jumps for $N(z)$. Indeed for the $11$ entry, on the set $(-\infty,c)\cup (b,a)$, we have
\begin{align*}
  (N_0^{-1}P^{\infty}(z;s))_{11+}& =  \frac{\theta_3(0)\theta_3(u_{+}(z)+s^{3/2}\Omega+d)}{\theta_3(s^{3/2}\Omega)\theta_3(u_{+}(z)+d)}N_{11+}=-\frac{\theta_3(0)\theta_3(-u_{-}(z)+s^{3/2}\Omega+d)}{\theta_3(s^{3/2}\Omega)\theta_3(-u_{-}(z)+d)} N_{12-}\\
   & =-(N_0^{-1}P^{\infty}(z;s))_{12-}.
\end{align*}
 where we have used $\theta_3(z+1)=\theta_3(z)$, and evenness of $\theta_3(z)$. On the interval $(c,b)$ we have
\begin{align*}
 (N_0^{-1} P^{\infty}(z;s))_{11+}&= \frac{\theta_3(0)\theta_3(u_{+}(z)+s^{3/2}\Omega+d)}{\theta_3(s^{3/2}\Omega)\theta_3(u_{+}(z)+d)}N_{11+}=\frac{\theta_3(0)\theta_3(u_{-}(z)-\tau+s^{3/2}\Omega+d)}{\theta_3(s^{3/2}\Omega)\theta_3(u_{-}(z)-\tau+d)}N_{11-}\\
 & = \frac{\theta_3(0)\theta_3(u_{-}(z)+s^{3/2}V+d)}{\theta_3(s^{3/2}\Omega)\theta_3(u_{-}(z)+d)}\frac{\exp(2\pi i(u_{-}(z)+s^{3/2}\Omega+d)-\pi i \tau )}{\exp(2\pi i(u_{-}(z)+d)-\pi i \tau )}N_{11-}\\
 &=(N_0^{-1}P^{\infty}(z;s))_{11-}\exp\left(2s^{3/2}g_+(b) \right).
\end{align*}
Repeating this computation for each of the remaining 3 entries yields the jumps of $P^{\infty}$. 

To obtain the large $z$ expansion of $P^{\infty}(z)$ we note that $P^{\infty}(z)N(z)^{-1}$ is analytic at infinity, and use the fact that 
by \eqref{uinfty}, $u(z)\to -1/2$ as $z\to\infty$.
\end{proof}

\subsubsection{Expansions of $u(z)$, $\beta(z)$, and some properties of $\theta_{jk}$}

Below we will use the following easy to obtain expansions:
\begin{equation}\label{uexp}
    u(z)=u_+(p)+u_{0,p}\sqrt{z-p}\left(1+u_{1,p}(z-p)+\bigO((z-p)^2)\right),\qquad z\to p\in\{a,b,c\},\quad \Im z>0,
\end{equation}
with 
\begin{align}
& u_{0,a}=\frac{1}{J_0\sqrt{(a-b)(a-c)}} , \qquad u_{1,a}=-\frac{1}{6}\left(\frac{1}{a-b}+\frac{1}{a-c} \right) ,\label{u0a}\\
& u_{0,b}=\frac{1}{iJ_0\sqrt{(a-b)(b-c)}} , \qquad u_{1,b}=\frac{1}{6}\left(\frac{1}{a-b}-\frac{1}{b-c} \right) ,\label{u0b}\\
& u_{0,c}=-\frac{1}{J_0\sqrt{(a-c)(b-c)}} , \qquad u_{1,c}=\frac{1}{6}\left(\frac{1}{a-c}+\frac{1}{b-c} \right) ,
\end{align}
and
\be\label{betaexp}
\begin{aligned}
   & \beta(z)=\beta_{0,p}(z-p)^{1/4}\left(1+\beta_{1,p}(z-p)+\bigO((z-p)^2)\right), \qquad z\to p\in\{a,c\},\quad\Im z>0,\\
   & \beta(z)=\beta_{0,b}(z-b)^{-1/4}(1+\beta_{1,b}(z-b)+\bigO((z-b)^2)), \qquad z\to b,\quad \Im z>0,
\end{aligned}
\ee
with
\begin{equation}\label{b0}
     \beta_{0,a}^2=\sqrt{\frac{a-c}{a-b}},\qquad
     \beta_{0,b}^2=i\sqrt{(a-b)(b-c)}, \qquad
     \beta_{0,c}^2=\sqrt{\frac{a-c}{b-c}}.
\end{equation}
The values of $\beta_{1,p}$ are unimportant in what follows. 

By the quasiperiodicity relations \eqref{qprelations} and the definition \eqref{definitionofV} of $\Omega$,
we obtain the following connections between the values of $\theta_{jk}$ in \eqref{thetajk}
on the $+$-side of $(-\infty,a)$:
\be\label{thetanoder}
\theta_{11}(u_+(b))=\theta_{12}(u_+(b))e^{2g(b)_+ s^{3/2}},\qquad
\theta_{22}(u_+(b))=\theta_{21}(u_+(b))e^{-2g(b)_+ s^{3/2}}.
\ee
Now considering the derivatives of the second relation in \eqref{qprelations}, we moreover obtain
\be\label{thdeder}
\theta_{j1}^{\prime}=-\theta_{j2}^{\prime}e^{2g(b)_+ s^{3/2}},\qquad
\theta_{j1}^{\prime\prime}=\theta_{j2}^{\prime\prime}e^{2g(b)_+ s^{3/2}},\qquad j=1,2,
\ee
and
\begin{equation}\label{thetader}
 \theta_{11}\theta_{22}^{\prime\prime}=\theta_{12}\theta_{21}^{\prime\prime},\qquad     \theta_{22}\theta_{11}^{\prime\prime}=\theta_{21}\theta_{12}^{\prime\prime}, \qquad 
 \theta_{11}^{\prime}\theta_{22}^{\prime}=\theta_{12}^{\prime}\theta_{21}^{\prime},
\end{equation}
where
\[
\theta_{jk}^{\prime}=\frac{d}{du}\theta_{jk}(u(z)),\qquad \theta_{jk}^{\prime\prime}=\frac{d^2}{du^2}\theta_{jk}(u(z)),
\]
and all $\theta_{jk}$ and their derivatives are evaluated at $u_+(b)$.

Similarly, at $u(a)=0$, we obtain
\be\label{thua}
\theta_{j1}(0)=\theta_{j2}(0),\qquad \theta_{j1}^{\prime}(0)=-\theta_{j2}^{\prime}(0),
\qquad \theta_{j1}^{\prime\prime}(0)=\theta_{j2}^{\prime\prime}(0),\qquad j=1,2.
\ee

\subsection{Identities for $\theta$-functions and elliptic integrals}\label{secidentities}
We now obtain a set of identities involving $\theta$-functions which will be used below in the proof of Theorem \ref{Mainthm}. First, recall all four Jacobi $\theta$-functions. These may be expressed in terms of $\theta_3(z)$, defined in (\ref{definitionoftheta3}), as follows.
\begin{equation}\label{theta2134relations}
    \theta_2(z)=\theta_1\left(z+\frac{1}{2}\right)=e^{-\pi i z+\pi i \tau /4}\theta_3\left(z-\frac{\tau}{2}\right), \quad \theta_4(z)=\theta_3\left(z+\frac{1}{2}\right).
\end{equation}

We note that $\theta_1(z)$ is an odd function of $z$, whereas $\theta_j(z)$, $j=2,3,4$ are all even functions. $\theta$-functions have the following periodicity properties:
\begin{equation}\label{theta period}
\begin{aligned}
\theta_1(z+1)&=-\theta_1(z), \qquad \theta_1(z+\tau)=-e^{-2\pi iz-\pi i \tau}\theta_1(z),\\
\theta_2(z+1)&=-\theta_2(z), \qquad \theta_2(z+\tau)=e^{-2\pi iz-\pi i \tau}\theta_2(z),\\
\theta_j(z+1)&=\theta_j(z), \qquad \theta_j(z+\tau)=e^{-2\pi iz-\pi i \tau}\theta_j(z),\qquad j=3,4.
\end{aligned}
\end{equation}

Denote by $\theta_k$, $\theta'_k$ the values of $\th$-functions and their derivatives at zero ($\th$-constants), i.e.,
\[\theta_k=\theta_k(0)=\theta_k(0;\tau),\qquad \theta'_k=\frac{d}{dz}\theta_k(z)|_{z=0},\qquad
 \qquad k=1,2,3,4.\]

\begin{lemma}\label{thetalemma}
Recall $J_0$, $\tau$, $\Omega$, $P^{\infty}(z;s)$ and $d$ given by (\ref{Jkdefinition}), (\ref{definitionofV}),  (\ref{finfinitydef}) and (\ref{definitionofd}), respectively. The following identities hold.
\begin{align}
    & i) \quad \theta_3^4=\frac{J_0^2}{\pi^2}(a-c), \qquad ii) \qquad \theta_4^4=\frac{J_0^2}{\pi^2}(a-b), \qquad
    iii) \quad \theta_2^4=\frac{J_0^2}{\pi^2}(b-c),\label{id123}\\
    & iv) \quad \det P^{\infty}(z;s)=1,\label{iddet1}\\
    & v)\quad \frac{\theta_3^{\prime}}{\theta_3}(u_+(b)+d)-\frac{\theta_1^{\prime}}{\theta_1}(u_+(b)+d)=J_0, \label{id5} \\
    & vi) \quad\left(\frac{\theta_1}{\theta_3}\right)^{\prime\prime\prime}(u_+(b)+d)=-3J_0\left(\frac{\theta_1}{\theta_3}\right)^{\prime\prime}(u_+(b)+d)
+\frac{6J_0(2\beta_{1,b}+u_{1,b})}{u_{0,b}^2}\frac{\theta_1}{\theta_3}(u_+(b)+d),\label{id6}\\
    & vii) \quad \frac{d\Omega}{db}=\frac{q(b)}{(a-b)(b-c)J_0}, \qquad viii) \quad \frac{d\tau}{db}=-\frac{\pi i}{J_0^2(a-b)(b-c)}.\label{idder}\\
     & ix) \quad\frac{\theta_3^{\prime}}{\theta_3}(d)-\frac{\theta_1^{\prime}}{\theta_1}(d)=J_0(a-c),\label{id5A}\\
    & x) \quad\left(\frac{\theta_1}{\theta_3}\right)^{\prime\prime\prime}(d)=-3J_0(a-c)\left(\frac{\theta_1}{\theta_3}\right)^{\prime\prime}(d)
+\frac{6J_0(a-c)(-2\beta_{1,a}+u_{1,a})}{u_{0,a}^2}\frac{\theta_1}{\theta_3}(d),\label{id6A}\\
& xi) \quad \frac{d\Omega}{da}=-\frac{q(a)}{(a-b)(a-c)J_0},\qquad xii)\quad \frac{d\tau}{da}=\frac{\pi i}{J_0^2(a-b)(a-c)}.\label{ddaOmega}
 \end{align}
\end{lemma}
\begin{proof}
The proof of this lemma is similar to the arguments in \cite{IB} wherein the authors derive analogous results in the case of the sine-kernel determinant. For the first three identities we proceed as follows. By the expansion of $u(z)$ at $a$,
\[
u(z)=\frac{1}{2J_0}\int_a^z \frac{dx}{p(x)^{1/2}}=\frac{1}{J_0}\frac{(z-a)^{1/2}}{\sqrt{(a-b)(a-c)}}\left(1+o(1)\right),\qquad z\to a,
\]
we obtain
\[
\frac{\th_3^2(u(z))}{\th_1^2(u(z))}=J_0^2 \frac{\th_3^2}{\th_1'^2}\frac{(a-b)(a-c)}{z-a}+\bigO(1),\qquad z\to a.
\]
Since $\th_1(0)=0$ is the only zero of $\th_1(z)$ modulo the lattice, it follows from Liouville's theorem that
\begin{equation}
    \frac{\theta_3^2(u(z))}{\theta_1^2(u(z))}-J_0^2\frac{\theta_3^2}{(\theta_1^{\prime})^2}\frac{(a-b)(a-c)}{z-a}=\mathrm{const},\qquad \forall z\in \mathbb{C}.
\end{equation}
We evaluate the constant by taking $z\to \infty$, making use of (\ref{theta2134relations}), and obtain
\begin{equation}\label{thetafourthpowers}
    \frac{\theta_3^2(u(z))}{\theta_1^2(u(z))}-J_0^2\frac{\theta_3^2}{(\theta_1^{\prime})^2}\frac{(a-b)(a-c)}{z-a}=\frac{\theta_4^2}{\theta_2^2}.
\end{equation}
Substituting here the value $z=b$, and making use of the identities (see, e.g., \cite{WW})
\begin{equation}
    \theta_1^{\prime}=\pi \theta_2\theta_3\theta_4, \qquad  \theta_3^4=\theta_2^4+\theta_4^4,
\end{equation}
we obtain (i). Similarly, substituting $z=c$ into \eqref{thetafourthpowers}, we obtain (ii).  The difference of (i) and (ii) gives (iii).

\bigskip
\noindent
iv) It follows from the RH problem for $P^{\infty}(z)$ by standard arguments that $\det P^{\infty}(z)$ is analytic in the complex plane and bounded at infinity. Therefore, it is a constant and, by the condition at infinity,
\begin{equation}
    \det P^{\infty}(z,s)=\det F_0, \qquad  \forall z\in \mathbb{C}.
\end{equation}
It remains to check that this value is $1$. We prove it as follows. First, by \eqref{finfinitydef},
\be\label{detP}
\begin{aligned}
    \det P^{\infty}(z;s)&=\frac{\theta_3^2\theta_3(u(z)+s^{3/2}\Omega+d)}{\theta_3^2(s^{3/2}\Omega)\theta_3^2(u(z)+d)}\theta_3(u(z)-s^{3/2}\Omega+d)\left( \frac{\beta(z)+\beta^{-1}(z)}{2}\right)^2\\
    &+\frac{\theta_3^2\theta(u(z)-s^{3/2}\Omega-d)}{\theta_3^2(s^{3/2}\Omega)\theta_3^2(u(z)-d)}\theta_3(u(z)+s^{3/2}\Omega-d)\left( \frac{\beta(z)-\beta^{-1}(z)}{2i}\right)^2.
\end{aligned}\ee

Consider the following meromorphic function on the Riemann surface of $p(z)^{1/2}$:
\be
\xi (z)=\beta^2(z)-1 =\frac{(z-a)(z-c)}{p(z)^{1/2}}-1.
\ee
By considerations in the previous section, $\xi(z)$ has 2 zeros, $\beta_1^*$, $\beta_2^*$, see \eqref{betastars}. They are located on the first sheet.
By Abel's theorem (see, e.g., \cite{FARKAS, DIZ})
\be
 -u(\infty)+u_+(\beta_1^*)+u(\beta_2^*)-u_+(b)=0 \mod \Lambda.
\ee
Since by Lemma \ref{ulemma}, $u(\infty)=-1/2$, $u_+(b)=-\tau/2$, $u_+(c)=-1/2-\tau/2$, and furthermore,
\[
u_+(\beta_1^*)=u_+(c)+d,\qquad d=\int_{c}^{\beta_1^*}\omega,
\]
we have
\[
   1/2 +u_+(c)+d+u(\beta_2^*)+\tau/2=u(\beta_2^*)+d =0      \mod \Lambda.
\]
or simply
\begin{equation}
    u(\beta_2^*)=-d \mod \Lambda.
\end{equation}
Note, moreover, that since $\beta(\beta_2^*)-\beta^{-1}(\beta_2^*)=0$  and
since by standard arguments $\det N(z)\equiv 1$, we have
\[
\left(\frac{\beta(\beta_2^*)+\beta^{-1}(\beta_2^*)}{2}\right)^2=1.
\]
We now set $z=\beta_2^*$ in the above expression \eqref{detP}  for the determinant and obtain
\begin{equation}\label{det=1}
    \det P^{\infty}(z;s)=1.
\end{equation}
One may expand the determinant at any point and deduce that all the terms apart from the constant one must be zero, yielding an infinite number of identities linking the derivatives of $\theta$ functions. In what follows, we will only need the constant terms in the expansion
at $z=b$ and $z=a$. A straightforward computation using expansions \eqref{uexp},  \eqref{betaexp} gives (in the notation of \eqref{thetaij})
for $z\to b$,
\begin{equation}\label{detatb}
    \frac{1}{4}\frac{d}{du}\left[\theta_{11}(u)\theta_{22}(u)-\theta_{12}(u)\theta_{21}(u) \right]_{u=u_+(b)}
    \beta_{0,b}^2 u_{0,b}+\theta_{11}(u_+(b))\theta_{22}(u_+(b))=1,
\end{equation}
where $u_{0,b}$, $\beta_{0,b}$ are defined in \eqref{u0b} and \eqref{b0}, respectively. We note that
\be\label{beta-u}
\beta_{0,b}^2 u_{0,b}=\frac{1}{J_0}.
\ee
Furthermore, using the identities \eqref{thetanoder}, \eqref{thdeder}, we can rewrite \eqref{detatb} in the form
\be\label{detatb2}
 \frac{1}{2J_0}\left[\theta_{11}^{\prime}(u)\theta_{22}(u)+\theta_{11}(u)\theta_{22}^{\prime}(u) \right]_{u=u_+(b)}
    +\theta_{11}(u_+(b))\theta_{22}(u_+(b))=1.
\ee
Similarly, expanding \eqref{detP} as $z\to a$, one obtains the identity
\be \label{155atA}
 \frac{1}{2J_0(a-c)}\left[\theta_{11}^{\prime}(0)\theta_{22}(0)+\theta_{11}(0)\theta_{22}^{\prime}(0) \right]
    +\theta_{11}(0)\theta_{22}(0)=1.
\ee

\bigskip
\noindent
v) Consider the function $\eta(z)$ defined by
\begin{equation}
  \eta(z)= \frac{\theta_1^2(u(z)+d)}{\theta_3^2(u(z)+d)}\left( \beta(z)+\beta^{-1}(z)\right)^2 -\frac{\theta_1^2(u(z)-d)}{\theta_3^2(u(z)-d)}\left( \beta(z)-\beta^{-1}(z)\right)^2 .
\end{equation}
Note that this expression is similar to that of $\det P^{\infty}$ above. Indeed we may employ the same methods and deduce that it is identically equal to $0$. On the other hand, by considering the constant coefficient (i.e. the coefficient of $(z-b)^0$) in its expansion for $z\to b$, we obtain the desired identity (v).
Vanishing of the term with $(z-b)^1$ in the expansion yields the identity (vi). Here we 
expanded the functions $\frac{\theta_1}{\theta_3}$, $\beta(z)$,
used \eqref{ucond0}, \eqref{theta period}, 
 the oddness
of $\theta_1$, the evenness of $\theta_3$, and the fact that
\[
\frac{\theta_1}{\theta_3}(u_+(b)-d)=\frac{\theta_1}{\theta_3}(-u_+(b)-d-\tau)=
\frac{\theta_1}{\theta_3}(-u_+(b)-d)=-\frac{\theta_1}{\theta_3}(u_+(b)+d).
\]

Similarly, expanding the expression at the point $z=a$ instead, one obtains (ix) and (x).

\bigskip
\noindent
vii) Recall that
\[
   I_k=\frac{1}{2}\oint_{A_1} \frac{z^k dz}{\sqrt{p(z)}},\qquad  J_k=-\frac{1}{2}\oint_{B_1} \frac{z^k dz}{\sqrt{p(z)}}
\]
where we integrate along the cycles. Note that
\begin{equation}
\Omega=\frac{g_+(b)}{\pi i}=\frac{i}{\pi} \left[I_2+q_1I_1+q_0I_0 \right].    
\end{equation}
By means of direct differentiation, we obtain
\begin{equation}\label{dbJ1}
       \frac{d}{db}I_1=\frac{1}{2}I_0+b\frac{d}{db}I_0,\qquad  \frac{d}{db}J_1=\frac{1}{2}J_0+b\frac{d}{db}J_0,
\end{equation}
and
\begin{equation}\label{dbJ2}
       \frac{d}{db}I_2=\frac{1}{2}I_1+\frac{b}{2}I_0+b^2\frac{d}{db}I_0,\qquad \frac{d}{db}J_2=\frac{1}{2}J_1+\frac{b}{2}J_0+b^2\frac{d}{db}J_0.
\end{equation}
Furthermore, expanding the integrand of
\be
\oint\frac{d}{dz}\left(\frac{(z-a)(z-c)}{z-b}\right)^{1/2}dz=0
\ee
along the cycles, we obtain
\begin{equation}\label{dbJ0}
   \frac{d}{db}I_0=-\frac{I_1-bI_0}{2(a-b)(b-c)},\qquad  \frac{d}{db}J_0=-\frac{J_1-bJ_0}{2(a-b)(b-c)},
\end{equation}
Differentiating $q_0$ in (\ref{defq1}) and using the above derivatives of $J_k$, we find
\begin{equation}\label{ddbq0}
    \frac{d}{db}q_0=-\frac{1}{2}(b+q_1)-q(b)\frac{\frac{d}{db}J_0}{J_0}.
\end{equation}

This identity will be used repeatedly in the proof of the theorem. The final ingredient is obtained by means of Riemann's period relations. 
Indeed these relations yield (by similar arguments as in the proof of Lemma 3.45  in \cite{DIZ} and Lemma 27 in \cite{IB})
\begin{equation}\label{riemannsrelations}
    I_1J_0-J_1I_0=2\pi i.
\end{equation}
Alternatively, to show \eqref{riemannsrelations},
we can first reduce $I_0$, $I_1$, $J_0$, $J_1$ to complete elliptic integrals (cf. next section) $K$, $E$, and then use 
the Legendre relation:
\[
K(k^{\prime})E(k)+K(k)E(k^{\prime})-K(k)K(k^{\prime})=\frac{\pi}{2},\qquad k'=\sqrt{1-k^2}.
\]

The above facts, taken together, allow us to differentiate $\Omega$ directly and obtain (vii).

Similarly, we have
\begin{equation}\label{daJ1}
       \frac{d}{da}I_1=\frac{1}{2}I_0+a\frac{d}{da}I_0,\qquad  \frac{d}{da}J_1=\frac{1}{2}J_0+a\frac{d}{da}J_0,
\end{equation}
\begin{equation}\label{daJ2}
       \frac{d}{da}I_2=\frac{1}{2}I_1+\frac{a}{2}I_0+a^2\frac{d}{da}I_0,\qquad \frac{d}{da}J_2=\frac{1}{2}J_1+\frac{a}{2}J_0+a^2\frac{d}{da}J_0,
\end{equation}
and
\begin{equation}\label{ddaJ0}
\frac{d}{da}I_0=\frac{I_1-aI_0}{2(a-b)(a-c)},\qquad \frac{d}{da}J_0=\frac{J_1-aJ_0}{2(a-b)(a-c)}.
\end{equation}
Furthermore,
\begin{equation}\label{ddaq0}
    \frac{dq_0}{da}=-\frac{1}{2}(a+q_1)-q(a)\frac{\frac{d}{da}J_0}{J_0}.
\end{equation} 

With the help of these equations, (xi) easily follows.

\bigskip
\noindent
viii) By definition of $\tau$ in (\ref{definitionofV}),
\begin{equation}
    \tau =\frac{I_0}{J_0}.
\end{equation}
Using (\ref{dbJ0}) and \eqref{riemannsrelations}, we obtain
\begin{equation}\label{dtaudb}
    \frac{d\tau}{db}=\frac{I_1J_0-J_1I_0}{2J_0^2(b-a)(b-c)}=-\frac{\pi i}{J_0^2(a-b)(b-c)},
\end{equation}
which proves (viii).

Similarly, (xii) follows by \eqref{ddaJ0}.
\end{proof}

\subsection{Local parametrices }
In this section we construct parametrices in small neighborhoods, $U_p$, of the points $p=a,b,c$, 
see Figure \ref{Us} where their boundaries are depicted,
which match the exterior parametrix to the main order as $s\to \infty$ on the boundaries $\partial U_p$. As in \cite{DIZ}, \cite{Arno}, \cite{CIK} these parametrices involve Bessel functions. We choose $U_b$, $U_c$ in such a way that the zero $b<\beta_1^*<c$ of $\beta(z)-\beta(z)^{-1}$
(cf. \eqref{betastars}) is outside $U_p$'s. 

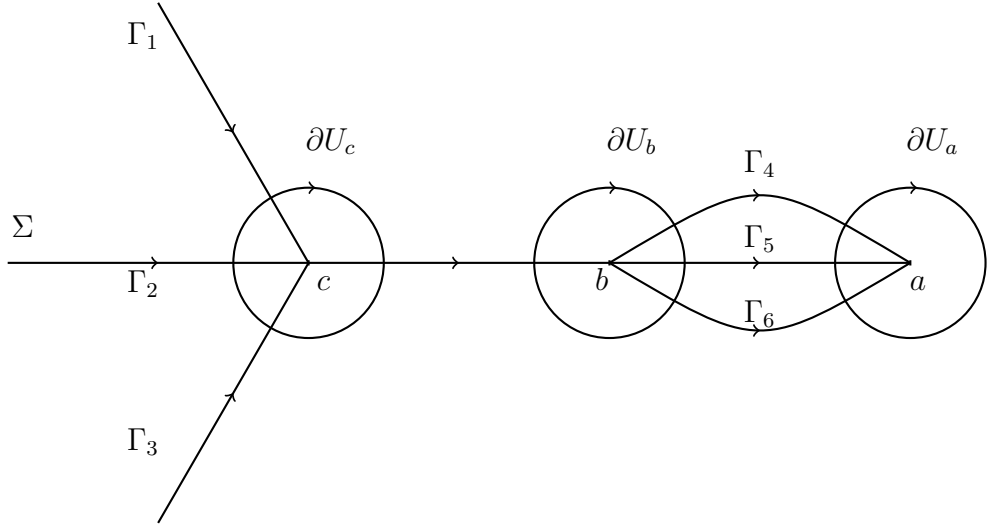
\begin{figure}
\begin{center}
    \begin{tikzpicture}
     \node[above] at (-5.8, -0.5)  (c)     {$c$};
     \node[above] at (-2.1, -0.5)  (c)     {$b$};
     \node[above] at (2.1, -0.5)  (c)     {$a$};
    \node[above] at (-9.8, 0.2)  (c)     {$\Sigma$};
    \node[above] at (-8.2,2.7)  (c)     {$\Gamma_1$};
\node[above] at (-8.2, -0.6)  (c)     {$\Gamma_2$};
\node[above] at (-8.2,-2.7)  (c)     {$\Gamma_3$};

    \node[above] at (0,1)  (c)     {$\Gamma_4$};      
    \node[above] at (0,0)  (c)     {$\Gamma_5$};  
    \node[above] at (0,-1)  (c)     {$\Gamma_6$};     

      \node[above] at (-5.7, 1.3)  (c)     {$\partial U_c$};
     \node[above] at (-1.7, 1.3)  (c)     {$\partial U_b$};
     \node[above] at (2.3, 1.3)  (c)     {$\partial U_a$};
    
    \draw[thick,middlearrow={>}] (-10,0)--(-6,0);
    \draw[thick,middlearrow={>}] (-6,0)--(-2,0);
      
    \draw[thick,middlearrow={>}] (-2,0)--(2,0);  
     \draw[thick] (-6,1pt)--(-6,-1pt);
     \draw[thick] (-2,1pt)--(-2,-1pt);
     \draw[thick] (2,1pt)--(2,-1pt);
     
        \draw[thick,middlearrow={>}] (-8,3.46)--(-6,0) ;
        \draw[thick,middlearrow={>}] (-8,-3.46)--(-6,0);
      
  \draw[thick,middlearrow={<}] (-6,-1) arc (-90:270:1);
     
         \draw[thick,middlearrow={<}] (-2,-1) arc (-90:270:1);

        \draw[thick,middlearrow={<}] (2,-1) arc (-90:270:1);
        
        \draw[thick,middlearrow={>}] (-2,0).. controls (0,1.2) .. (2,0);
        \draw[thick,middlearrow={>}] (-2,0).. controls (0,-1.2) .. (2,0);
 
    \end{tikzpicture}
        \caption{Jump contour with neighbourhoods $U_p$.}\label{Us}
\end{center}
\end{figure}

\subsubsection{Signs of $q(p)$}
Here we prove that for $0>a>b>c$,
\begin{equation}\label{signsqp}
    q(a)<0, \quad q(b)<0,\quad \mathrm{and} \quad q(c)>0.
\end{equation}
This will be used in subsequent analysis.

First note that, since
\begin{equation}
      J_k =(-1)^{k+1}\int_c^b \frac{\lvert z\rvert^k dz}{\sqrt{\lvert p(z) \rvert }},
\end{equation}
we have that
\begin{equation}
   \lvert b \rvert^k\lvert J_0 \rvert \leq \lvert J_k \rvert \leq \lvert c \rvert^k \lvert J_0 \rvert.
\end{equation}
Substituting these inequalities into the second form of $q_0$ in (\ref{defq1}), we find that
\begin{equation}\label{signofqb}
    q(b) \leq -\frac{1}{3}(a-b)(b-c) <0,
\end{equation}
and
\begin{equation}\label{signofqc}
    q(c) \geq \frac{1}{3}(a-c)(b-c) >0.
\end{equation}
To determine the sign of $q(a)$ we require a sharper bound. Let us, first, write
\begin{equation}
    J_1=\int_c^b \frac{a+z-a}{\sqrt{p(z)}}dz=aJ_0+\int_c^b \sqrt{\frac{a-z}{(b-z)(z-c)}} dz
\end{equation}
Then note the following standard reductions, with $k=\sqrt{\frac{b-c}{a-c}}$.
\begin{align*}
    & J_0  =-\int_c^b\frac{dz}{\sqrt{(a-z)(b-z)(z-c)}}=\left[ z=c+(b-c)w^2 \right]=-\frac{2}{\sqrt{a-c}}K(k),\\
    & K(k)=\int_0^1 \frac{dw}{\sqrt{1-k^2w^2}\sqrt{1-w^2}},
\end{align*}
and, similarly,
\begin{equation*}
    \int_c^b \sqrt{\frac{a-z}{(b-z)(z-c)}}dz =2\sqrt{a-c}E(k),\qquad
    E(k)=\int_0^1 \frac{\sqrt{1-k^2 w^2}}{\sqrt{1-w^2}}dw,
\end{equation*}
where $K(k)$ and $E(k)$ are the complete elliptic integrals of the first and second kind, respectively. 

Note that $0<k=\sqrt{\frac{b-c}{a-c}}<1$. With these expressions for $J_0$ and $J_1$ we find that $q(a)$ takes the following form.
\begin{equation}
    q(a)=\frac{1}{3}(a-c)\left(a-b-q_1\frac{E(k)}{K(k)}\right), \quad k=\sqrt{\frac{b-c}{a-c}}.
\end{equation}
We want to show that $q(a)<0$, i.e. that
\begin{equation}\label{qaconditionfornegative}
     \frac{E(k)}{K(k)}>\frac{a-b}{q_1}.
\end{equation}
As the authors in \cite{BCL} we will make use of the estimate
\begin{equation}\label{EKbound}
    \frac{E(k)}{K(k)}>\sqrt{1-k^2}.
\end{equation}
This reduces the problem to seeing that
\begin{equation}
    \sqrt{1-k^2}=\sqrt{\frac{a-b}{a-c}}>\frac{a-b}{q_1} \quad\mathrm{i.e., that}\quad (a+b+c)^2>4(a-b)(a-c) .
\end{equation}
But this is clear since, for $0>a>b>c$, we have
\begin{equation}
    (a+b+c)^2-4(a-b)(a-c)>(b+c)^2-4bc=(b-c)^2>0.
\end{equation}

The estimate (\ref{EKbound}) was proven in \cite{NeumanBounds}, in a different form, and also follows from log-convexity of the Gauss hypergeometric function \cite{C1}. We now provide a proof for the reader's convenience.

First  we have that
\begin{equation*}
   \frac{\pi^2}{4}= \left(\int_0^1\frac{dz}{\sqrt{1-z^2}} \right)^2 <\left(\int_0^1\frac{\sqrt{1-k^2z^2}}{\sqrt{1-z^2}}dz\right)\left(\int_0^1\frac{dz}{\sqrt{1-k^2z^2}\sqrt{1-z^2}}\right) =E(k)K(k),
\end{equation*}
and so
\begin{equation}
    \frac{E(k)}{K(k)}>\frac{\pi^2}{4K(k)^2}.
\end{equation}
It, thus, remains to prove the estimate
\begin{equation}\label{requiredKestimate}
    K(k)\leq \frac{\pi}{2}\frac{1}{\sqrt[4]{1-k^2}}.
\end{equation}
For $0<\alpha<\beta$, let us denote by $M(\alpha,\beta)$ the arithmetic-geometric mean. This is defined as the common limit of the following sequences.
\begin{align}
    &\alpha_0=\alpha,\quad \beta_0=\beta,\quad \alpha_{n+1}=\sqrt{\alpha_n \beta_n},\quad \alpha_n\leq \alpha_{n+1}; \qquad \beta_{n+1}=\frac{\alpha_n+\beta_n}{2},\quad \beta_{n+1}\leq \beta_n.
\end{align}
We then have
\begin{theorem}[Gauss, 1799]
\begin{equation}
    K(k) M(\sqrt{1-k^2},1)=\frac{\pi}{2}.\label{Gaussthm}
\end{equation}
\end{theorem}
\begin{proof}
Define the integral $I(\alpha,\beta)$ by
\begin{equation}
    I(\alpha,\beta)=\int_0^{\frac{\pi}{2}}\frac{d\theta}{\sqrt{\alpha^2\cos^2\theta+\beta^2\sin^2\theta}}, \quad I(1,\sqrt{1-k^2})=I(\sqrt{1-k^2},1)
   = K(k).
\end{equation}
Note that
\begin{equation}
    I(\alpha,\beta)=\left[ \sin\theta=\frac{2\alpha \sin \phi }{(\alpha+\beta)+(\alpha-\beta)\sin^2\phi}\right]=I\left(\frac{\alpha+\beta}{2},\sqrt{\alpha \beta}\right).
\end{equation}
It follows, recursively, that 
\begin{equation}
    I(\alpha,\beta)=I\left(\frac{\alpha+\beta}{2},\sqrt{\alpha \beta}\right)=I(\beta_2,\alpha_2)=\cdots
    =I(M(\alpha,\beta),M(\alpha,\beta))=\frac{\pi}{2M(\alpha,\beta)},
\end{equation}
which gives the result upon taking $\alpha=\sqrt{1-k^2}$, $\beta=1$.
\end{proof}
The required estimate, (\ref{requiredKestimate}), now follows from \eqref{Gaussthm} by noting that $\alpha_1=\sqrt{\alpha \beta}\leq M(\alpha,\beta)$ with
$\alpha=\sqrt{1-k^2}$, $\beta=1$.


\subsubsection{Bessel model RH problem}
First recall the Bessel model RH problem of \cite{CIK}, which is a slight modification of the corresponding problem in \cite{Arno}, which in turn, is related to the problem in \cite{DIZ}.

Let $I_0$, $K_0$, $H_{0}^{(1)}$, $H_{0}^{(2)}$ denote the standard Bessel functions, see e.g. \cite{AS}, with branch cut along $(-\infty,0)$. 
Let, furthermore, $I_0^{\prime}(x)=\frac{d}{dx}I_0(x)$, etc, and define

\begin{equation}\label{Psi1}
\Psi_1(\zeta)=e^{-\frac{i\pi}{4}\sigma_3}\pi^{\frac{1}{2}\sigma_3}\begin{pmatrix}I_0(\zeta^{1/2})&\frac{i}{\pi}K_0(\zeta^{1/2})\\\ i \pi\zeta^{1/2}I_0^{'}(\zeta^{1/2})&-\zeta^{1/2}K_0^{'}(\zeta^{1/2}) \end{pmatrix},
\end{equation}

 \begin{equation}
\Psi_2(\zeta)=\frac{1}{2}e^{-\frac{i\pi}{4}\sigma_3}\pi^{\frac{1}{2}\sigma_3}\begin{pmatrix}H_0^1(-i\zeta^{1/2})&H_0^2(-i\zeta^{1/2})\\\  \pi\zeta^{1/2}H_0^{1'}(-i\zeta^{1/2})&\pi \zeta^{1/2}H_0^{2'}(-i\zeta^{1/2}) \end{pmatrix}
\end{equation}
and
 \begin{equation}
\Psi_3(\zeta)=\frac{1}{2}e^{-\frac{i\pi}{4}\sigma_3}\pi^{\frac{1}{2}\sigma_3}\begin{pmatrix}H_0^2(i\zeta^{1/2})&-H_0^1(i\zeta^{1/2})\\ -\pi \zeta^{1/2}H_0^{2'}(i\zeta^{1/2})&\pi \zeta^{1/2}H_0^{1'}(i\zeta^{1/2}) \end{pmatrix}
\end{equation}

Define regions
$\Omega_1=\{\zeta\in \mathbb{C}: -\frac{2\pi}{3}<\arg \zeta < \frac{2\pi}{3} \} $,
$\Omega_2=\{\zeta\in \mathbb{C}: \frac{2\pi}{3}<\arg \zeta < \pi \} $
and
$\Omega_3=\{\zeta\in \mathbb{C}: -\pi<\arg \zeta < -\frac{2\pi}{3} \} $ (see Figure \ref{FigBessel})
and consider the function $\Psi(\zeta)$ defined in each region $\Omega_k$ by $\Psi_k$. We have
\begin{figure}
\begin{center}
    \begin{tikzpicture}
    \draw[thick,middlearrow={>}] (-2,3.46)--(0,0) node[anchor=north west] {$0$} ;
     \node[above] at (-3.5, 0.3)  (c)     {$\Gamma$};
    \draw[thick,middlearrow={>}] (-2,-3.46)--(0,0) ;
    \draw[thick,middlearrow={>}] (-4,0)--(0,0) ;
         \node[above] at (-2.3, -1.3)  (c)     {$\Omega_3$};
         \node[above] at (2, 0)  (c)     {$\Omega_1$};
         \node[above] at (-2.3, 1.3)  (c)     {$\Omega_2$};
    \end{tikzpicture}
    \caption{Jump contour for the $\Psi$-RH problem.}\label{FigBessel}
\end{center} 
\end{figure}
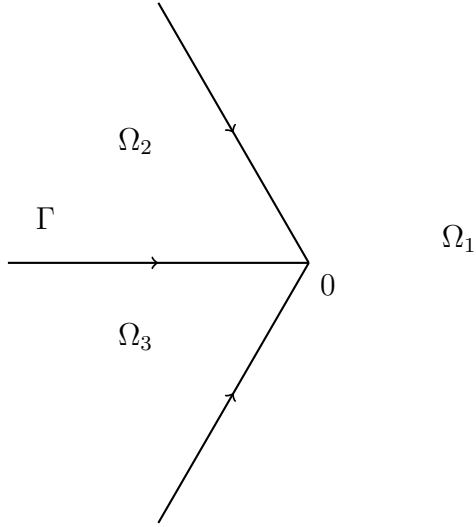
\begin{lemma} \cite{Arno,CIK}
For $\Gamma=(-\infty,0)\cup \mathbb{R_+}\exp\left(\frac{2\pi i}{3} \right)\cup \mathbb{R_+}\exp\left(-\frac{2\pi i}{3} \right)$, $\Psi(\zeta)$ 
satisfies the following RHP.
\begin{equation}
\begin{aligned}
&\Psi(\zeta)\;\mathrm{ is\; analytic\; in } \;\mathbb{C}\setminus \Gamma,\\ 
&\Psi_{+}(\zeta)=\Psi_{-}(\zeta)\begin{pmatrix}1&0\\ 1&1 \end{pmatrix} \mathrm{, } \qquad \zeta \in \mathbb{R_+}\exp\left(\frac{2\pi i}{3} \right)\cup \mathbb{R_+}\exp\left(-\frac{2\pi i}{3} \right),\\ 
&\Psi_{+}(\zeta)=\Psi_{-}(\zeta)\begin{pmatrix} 0&1\\ -1&0 \end{pmatrix} \mathrm{, }\qquad  \zeta \in (-\infty,0),\\ 
&\Psi(\zeta)=\zeta^{-\frac{1}{4}\sigma_3}N_0\left(I+\frac{1}{8\sqrt{\zeta}}\begin{pmatrix}-1&-2i\\ -2i&1 \end{pmatrix}+\bigO\left(\frac{1}{\zeta}\right)\right)e^{\zeta^{1/2}\sigma_3} \mathrm{, }\qquad \zeta \to \infty.\\
\end{aligned}
\end{equation}
Furthermore, the expansion at infinity holds uniformly as $|\zeta|\to\infty$.

\subsubsection{Construction of the parametrices}
\end{lemma}
Writing 
\[
g(z)=g_+(p)+\int_{p}^z\frac{q(\zeta)d\zeta}{p(\zeta)^{1/2}},
\]
let
\begin{equation}\label{fdef}
    f_p(z)=\left(\int_{p}^z\frac{q(\zeta)d\zeta}{p(\zeta)^{1/2}}\right)^2, \qquad f_p(z)^{1/2}=-
    \int_{p}^z\frac{q(\zeta)d\zeta}{p(\zeta)^{1/2}}=g_+(p)-g(z),\qquad
    p=a,b,c.
\end{equation}
It is easy to verify that there exist an open disk $U_p$ centered at $p$ (for each $p=a,b,c$) of a sufficiently 
small radius $\epsilon$ such that
$f_p(z)$ is a conformal mapping of $U_p$ onto a neighborhood of zero. Expanding close to $z=p$, we have, in particular,
\begin{equation}\label{glocalexp}
f_p(z)^{1/2}=
 \begin{cases}-\frac{2q(a)}{\sqrt{(a-b)(a-c)}}(z-a)^{\frac{1}{2}}(1+\bigO(z-a)),\qquad z\to a, \\  -\frac{2  q(b)}{\sqrt{(a-b)(b-c)}} (b-z)^{\frac{1}{2}} (1+\bigO(z-b)),\qquad  z\to b,  \\ \frac{2q(c)}{\sqrt{(a-c)(b-c)}} (z-c)^{\frac{1}{2}} (1+\bigO(z-c)),\qquad  z\to c .
\end{cases}
\end{equation}
Note that the signs of the quantities $q(p)$ in \eqref{signsqp} ensure the location of the cuts of $f_p(z)^{1/2}$ in $U_p$
corresponds to the contour of the $\Psi$-RHP under the mapping 
\[
\zeta(z)=s^3 f_p(z).
\]
We now choose the exact form of the jump contour $\Sigma_S$, for the $S$-RHP
problem \eqref{SRHP}, in $U_p$ so that its image under the mapping $\zeta(z)$ is direct lines. This relates the construction to the Bessel problem.

We look for local parametrices, with $z\in U_p$, in the form
\begin{equation}\label{defP}
\begin{aligned}
P_p(z;s)&=E_p(z;s)\Psi\left(s^3f_p(z)\right)\exp(s^{3/2}g(z)\sigma_3),\qquad p=a,c\\
P_b(z;s)&=E_b(z;s)\sigma_3\Psi\left(s^3f_b(z)\right)\sigma_3 \exp(s^{3/2}g(z)\sigma_3),\qquad p=b.
\end{aligned}
\end{equation}
where $E_p(z)$ are, analytic in $U_p$, matrix-valued functions chosen so that $P_p(z;s)$ matches $S(z)$ on $\partial U_p$
to the main order. We will now see that the following functions satisfy these conditions. Set
\begin{equation}\label{Ep}
E_p(z;s)=P^{\infty}(z;s)e^{-s^{3/2}g_{\pm}(p)\sigma_3}N_0^{-1}f_p(z)^{\frac{1}{4}\sigma_3}s^{\frac{3}{4}\sigma_3}, \qquad p=a,c,
\end{equation}
and
\begin{equation}\label{Eb}
E_b(z;s)=P^{\infty}(z;s)e^{-s^{3/2}g_{\pm}(b)\sigma_3}\sigma_3 N_0^{-1} \sigma_3 f_b(z)^{\frac{1}{4}\sigma_3}s^{\frac{3}{4}\sigma_3},\qquad p=b,
\end{equation}
where in $\pm$ we choose $+$, if $\Im (z-p) >0$ and $-$, if $\Im (z-p) <0$.

We have
\begin{lemma}\label{Eanalytic}
The functions $E_p(z)=E_p(z;s)$ defined by \eqref{Ep}, \eqref{Eb} are analytic in the respective neighbourhoods $U_p$.
\end{lemma}
\begin{proof} 
 Note that  \eqref{glocalexp} gives
\[
f_{a}^{\frac{1}{4}}(z)_{+}=i f_{a}^{\frac{1}{4}}(z)_{-},\quad z\in (a-\epsilon,a);\qquad
f_{c}^{\frac{1}{4}}(z)_{+}=i f_{c}^{\frac{1}{4}}(z)_{-},\quad z\in (c-\epsilon,c),
\]
and
\[
f_{b}^{\frac{1}{4}}(z)_{+}=-i f_{c}^{\frac{1}{4}}(z)_{-},\quad z\in (b,b+\epsilon).
\]

Consider $U_c$. First, we verify that $E_c(z)$ has no jumps in $U_c$.
On $(c-\epsilon,c)$, we have using the jump conditions for $P^{\infty}$ in \eqref{pinfinityrhp}, and the fact that by \eqref{gjumps}, $g_+(c)=-g_-(c)$,
\[
E_c(z)_{+}=P_{-}^{\infty}(z;s)e^{-s^{3/2}g_-(c)}\begin{pmatrix} 0&1\\ -1&0 \end{pmatrix}N_0^{-1}\begin{pmatrix}i&0\\0&-i \end{pmatrix}f_{c}^{\frac{1}{4}\sigma_3}(z)_{-}s^{\frac{3}{4}\sigma_3}=E_c(z)_{-}.
\]
On $(c,c+\epsilon)$, since $g_+(b)=g_+(c)$,
\[
E_c(z)_{+}=P_{-}^{\infty}(z;s)e^{s^{3/2}g_+(c)}N_0^{-1}f_{c}^{\frac{1}{4}\sigma_3}(z)_{-}s^{\frac{3}{4}\sigma_3}=E_c(z)_{-}.
\]
So that the singularity of $E_c(z)$ at $c$ is not a branch point.
Since $f_c^{\frac{1}{4}\sigma_3}(z)$ and the matrix elements $P^{\infty}_{jk}(z)$ contain at worst a $(z-c)^{-1/4}$ singularity each, it follows that the singularity of $E_c(z)$ at $c$ is removable. Thus $E_c(z)$ is analytic in $U_c$.

The functions $E_b(z)$ in $U_b$, and $E_a(z)$ in $U_a$ are considered similarly.
\end{proof}

Now by the large $\zeta$ expansion of $\Psi(\zeta)$ and \eqref{fdef},
 we find that on the boundaries $\partial U_p$ for large $s$, uniformly in $z$,
\begin{equation}\label{ACDelta2Error}
P_p(z;s)=P^{\infty}(z;s)e^{-s^{3/2}g_{\pm}(p)\sigma_3}\left(I+\frac{1}{8s^{3/2}f_p(z)^{1/2}}\begin{pmatrix}-1&-2i\\-2i&1\end{pmatrix}+\bigO\left(\frac{1}{s^3}\right)\right)e^{s^{3/2}g_{\pm}(p)\sigma_3},
\end{equation}
for $p=a,c$, 
and 
\be\label{BDelta2Error}
P_b(z;s)= P^{\infty}(z;s)e^{-s^{3/2}g_{\pm}(b)\sigma_3}\left(I+\frac{1}{8s^{3/2}f_b(z)^{1/2}}\begin{pmatrix}-1&2i\\2i&1\end{pmatrix}+\bigO\left(\frac{1}{s^3}\right)\right) e^{s^{3/2}g_{\pm}(b)\sigma_3},
\ee
where as before, in $\pm$, we choose $+$, if $\Im (z-p) >0$ and $-$, if $\Im (z-p) <0$.

Since $E_p(z)$ are analytic,  left-hand multiplication by them does not affect the jump conditions for $P_p(z)$, and it is easy to check
that $P_p(z)$ satisfies the same jump conditions as $S$ on $\Sigma_S\cup U_p$.

Thus, we have shown

\begin{lemma}\label{Pprhp}
With $\Sigma_S$ and $v_S$ denoting the contour and jumps, respectively, of the $S$-RHP \eqref{SRHP}; the function $P_p(z;s)$ given by \eqref{defP}
in $U_p$, $p=a,b,c$,
has the following properties.
\begin{equation}
\begin{aligned}
&P_p(z;s)\; \mathrm{ is\; analytic\; in }\; \left(\mathbb{C}\setminus \Sigma_S\right)\cap U_p,\\ 
&P_p(z;s)_{+}=P_p(z;s)_{-}v_S(z;s),\qquad z \in \Sigma_S \cap U_p.
\end{aligned}
\end{equation}
Moreover, $P_p(z;s)(P^{\infty}(z;s))^{-1}$ has the following large-$s$ expansion:
\begin{equation*}
P_p(z;s)(P^{\infty}(z;s))^{-1} =I+\Delta_1(z;s)+\bigO\left(\frac{1}{s^{3}}\right), \quad \Delta_1(z;s)=\bigO\left(\frac{1}{s^{3/2}}\right), \quad s\to \infty,
\end{equation*}
uniformly on $\partial U_p$. The function $\Delta_1(z;s)$ in $z$ is meromorphic in $U_p$.
For $p=a,c$, we have
\begin{equation*}
\Delta_1(z;s)=\frac{1}{8s^{3/2}f_p(z)^{1/2}}P^{\infty}(z;s)e^{-s^{3/2}g_{\pm}(p)\sigma_3}\begin{pmatrix}-1&-2i\\-2i&1\end{pmatrix} e^{s^{3/2}g_{\pm}(p)\sigma_3}\left( P^{\infty}(z;s)\right)^{-1},
\end{equation*}
and, for $p=b$,
\[
\Delta_1(z;s)=\frac{1}{8s^{3/2}f_b(z)^{1/2}} P^{\infty}(z;s)e^{-s^{3/2}g_{\pm}(b)\sigma_3}\begin{pmatrix}-1&2i\\2i&1\end{pmatrix} e^{s^{3/2}g_{\pm}(b)\sigma_3}\left( P^{\infty}(z;s)\right)^{-1},
\]
where in $\pm$, we choose $+$, if $\Im (z-p) >0$ and $-$, if $\Im (z-p) <0$.
\end{lemma}

\subsection{Small norm RH Problem for $R$}
Set
\begin{equation}\label{RRHPDEF}
R(z)=\left(s^{-\sigma_3/4}S_0 N_0 F_0^{-1}N_0^{-1}\right)^{-1} S(z)\times
\begin{cases}
P_p(z)^{-1},& \mathrm{for}\quad  z\in U_p,\\
P^{\infty}(z)^{-1},& \mathrm{for}\quad  z\in \mathbb{C}\setminus U_p,
\end{cases}
\end{equation}
and let $\Sigma_R$ denote the contour in Figure \ref{FigR}.
\begin{figure}
\begin{center}
    \begin{tikzpicture}
     \node[above] at (-5.7, -0.2)  (c)     {$c$};
     \node[above] at (-1.7, -0.2)  (c)     {$b$};
     \node[above] at (2.3, -0.2)  (c)     {$a$};
      
      \node[above] at (-5.7, 1.3)  (c)     {$\partial U_c$};
     \node[above] at (-1.7, 1.3)  (c)     {$\partial U_b$};
     \node[above] at (2.3, 1.3)  (c)     {$\partial U_a$};

     \draw[thick] (-6,1pt)--(-6,-1pt);
     \draw[thick] (-2,1pt)--(-2,-1pt);
     \draw[thick] (2,1pt)--(2,-1pt);
     
        \draw[thick,middlearrow={>}] (-8,3.46)--(-6.5,0.866) ;
        \draw[thick,middlearrow={>}] (-8,-3.46)--(-6.5,-0.866);
      
  \draw[thick,middlearrow={<}] (-5,0) arc (0:360:1);
     
         \draw[thick,middlearrow={<}] (-1,0) arc (0:360:1);

        \draw[thick,middlearrow={<}] (3,0) arc (0:360:1);
        
        \draw[thick,middlearrow={>}] (-2+0.707,0.707).. controls (0,1.2) .. (2-0.707,0.707);
        \draw[thick,middlearrow={>}] (-2+0.707,-0.707).. controls (0,-1.2) .. (2-0.707,-0.707);
    \end{tikzpicture}
    \caption{Jump contour for the $R$-RH problem.}\label{FigR}
\end{center}
\end{figure}
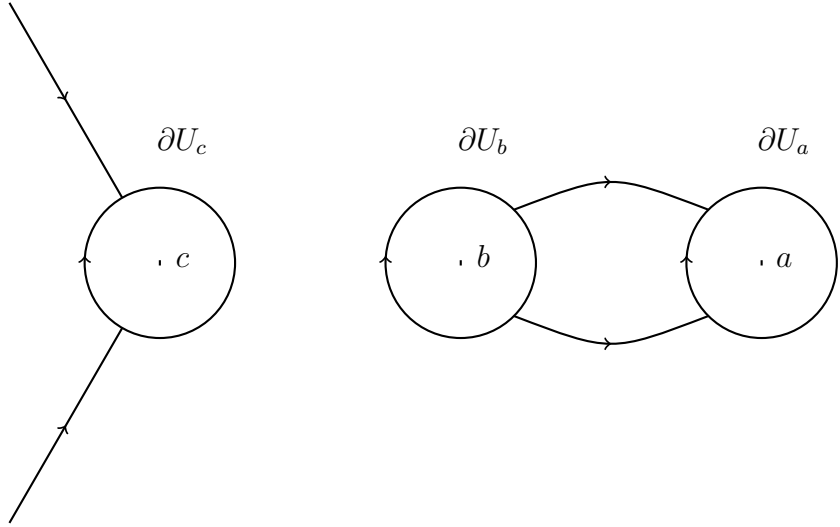
Then, $R$ satisfies the following Riemann-Hilbert problem.
\begin{equation}
\begin{aligned}
&R \;\mathrm{ is\; analytic\; in\; } \mathbb{C}\setminus \Sigma_R \\
&R_+(z)=R_-(z)P_p(z)(P^{\infty}(z))^{-1},\qquad z\in \partial U_p,\quad p\in \{a,b,c\} \\
&R_+(z)=R_-(z)P^{\infty}(z)v_S(z)(P^{\infty}(z))^{-1},\qquad z\in \Sigma_R \setminus \left(\partial U_a \cup \partial U_b \cup \partial U_c\right), \\
&R(z)\to I,\qquad z\to \infty,
\end{aligned}
\end{equation}
where
\begin{equation}
v_S(z)=\begin{pmatrix}1 & 0\cr e^{2s^{3/2}g(z)} & 1\end{pmatrix}.
\end{equation}
Indeed, the jump conditions easily follow from 
 the RH problems for $S(z)$, $P^{\infty}(z)$ and $P_p(z)$. The condition at infinity follows from \eqref{SRHP}, \eqref{pinfinityrhp},
 \eqref{Ninfty}.

We will now show that all the jumps for the $R$-RHP problem are close to the identity for large $s$ both in terms of $L^{\infty}$ and $L^2$ norms, which guarantees the solvability of the problem for large $s$ \cite{Dbook}. 
 
Let us first consider the jumps of $R$ on $\Sigma_R \setminus \left(\partial U_a \cup \partial U_b \cup \partial U_c\right)$. We claim that one may deform the contour $\Gamma_1 \cup \Gamma_3 \cup \Gamma_4 \cup \Gamma_6$, in the RH problem for $S(z)$ (\ref{SRHP}), in order to have $\Re g(z) < 0$ on this part of $\Sigma_R$. In particular, we show that this implies
\begin{equation}\label{expsmall}
P^{\infty}(z)v_S(z)(P^{\infty}(z))^{-1} =I+\bigO\left(e^{-c's^{3/2}\max\{1,|z|^{3/2}\}}\right), \quad c'>0.
\end{equation}

First, since 
\be\label{cond0}
\Re g(z) =\frac{2}{3}\lvert z \rvert^{3/2}\cos \left( \frac{3}{2}\arg z \right)+\bigO\left(\frac{1}{\sqrt{z}}\right), \quad z\to \infty \textmd{ with } \frac{\pi}{2}<\lvert \arg (z-a) \rvert <\pi,
\ee
we have that, for sufficiently large $\lvert z \rvert$, $\Re (g) <0$ in the sector, and 
hence \eqref{expsmall} holds on $\Sigma_R$ for suffcieintly large $|z|$.

On the other hand it follows from (\ref{signsqp}) that 
\[
q(z)=(z-x_1)(z-x_2),\qquad x_1\in (c,b),\quad x_2\in (a,+\infty),
\]
and therefore 
\begin{equation}
   \Im \frac{q(z)}{\sqrt{p(z)}_+}=  \frac{\partial}{\partial x}\Im g_+(x)>0, \quad x\in (-\infty,c)\cup(b,a).
\end{equation}
Since $\Re g(x)=0$, for $x\in(-\infty,c)\cup(b,a)$, using
the Cauchy-Riemann equations we have, for $0< y < \delta$, ($\delta$ sufficiently small) that
\be\label{cond1}
 \Re g(x+iy)=\int_0^{y}\frac{\partial \Re g(x+i\psi)}{\partial \psi}d\psi<0,\qquad x \in (-\infty,c)\cup(b,a).
\ee
Similarly, decreasing $\delta$ if need be, we show that
\be\label{cond2}
\Re g(x+iy)<0,\qquad x \in (-\infty,c)\cup(b,a),\quad -\delta<y<0.
\ee
Conditions \eqref{cond1} and \eqref{cond2} together with the boundedness (from infinity and zero) of $\theta$-functions on the contour 
imply that the jump matrix satisfies
\[
P^{\infty}(z)v_S(z)(P^{\infty}(z))^{-1}=I+\bigO\left(e^{-s^{3/2}\epsilon}\right)
\]
for some $\epsilon>0$ uniformly on the lenses of the $\Sigma_R$ contour around $(b,a)$, and also for sufficiently small $|z-c|$ on the lenses
around $(-\infty,c)$. It remains to obtain an $L^{\infty}$ estimate on the rest of latter lenses: between a fixed small $|z-c|$ and
a fixed large $|z-c|$ (above which we can use the estimate \eqref{cond0}). Consider the part of the lenses with $\Im z>0$ (for the $\Im z<0$ part
we argue similarly).
Clearly, it is sufficient to show that
\be\label{Imcond}
\Im \frac{q(x+iy)}{\sqrt{p(x+iy)}}>0,\qquad x \in (-\infty,c),\quad y>0.
\ee
Let $\phi_1>\phi_2>\phi_3>\phi_4>\phi_5$ denote the acute angles between the real line and the direct lines connecting
$x+iy$, $x \in (-\infty,c)$, $y>0$, and the points $c<x_1<b<a<x_2$, respectively. Then
\[
\frac{q(x+iy)}{\sqrt{p(x+iy)}}=\left|\frac{q(x+iy)}{\sqrt{p(x+iy)}}\right|\exp\left\{ i\left(\frac{\pi}{2}+\frac{\phi_1+\phi_3+\phi_4-2\phi_2-2\phi_5}{2}
\right)\right\}.
\]
It is easily seen that
\[
-\pi<\phi_1+\phi_3+\phi_4-2\phi_2-2\phi_5<\pi,
\]
and thus we obtain \eqref{Imcond}.

To summarize, we established the estimate 
\be 
P^{\infty}(z)v_S(z)(P^{\infty}(z))^{-1}=I+\bigO\left(e^{-s^{3/2}\epsilon' \max\{1,|z|^{3/2}\}}\right),\qquad \epsilon'>0,
\ee
on the lenses of the contour $\Sigma_R$.

It remains to consider the jumps of the $R$ problem on the boundaries $\partial U_p$. By means of Lemma \ref{Pprhp} we have that
\begin{equation}
v_R(z)=P_p(z)(P^{\infty}(z))^{-1} =
I+\Delta_1(z;s)+\bigO\left(\frac{1}{s^{3}}\right)=
I+\bigO\left(\frac{1}{s^{3/2}}\right), \qquad s\to \infty,
\end{equation}
uniformly on $\partial U_p$, $p=a,b,c$. 

Thus, by standard theory (see, e.g., \cite{Dbook}), the $R$-problem is a small-norm problem for large $s$ solvable by means of Neumann series. In what follows, we will need the first 2 terms. We then obtain 
\be \label{RUNIF}
R(z)=I+R^{(1)}(z)+\bigO\left(\frac{1}{s^{3}}\right), \qquad R^{(1)}(z)=\bigO\left(\frac{1}{s^{3/2}}\right),
\ee
where the estimate for $R^{(1)}(z)$ and the error term are uniform in $z$.
As usual, we substitute these series
into the $R$-RHP, and obtain that $R^{(1)}(z)$ is analytic in $\mathbb{C}\setminus\cup_p \partial U_p$,
\be
\begin{aligned}
&R_{+}^{(1)}(z)=R_{-}^{(1)}(z)+\Delta_1(z),\qquad z\in \cup_p \partial U_p,\\
&R_{+}^{(1)}(z)\to 0,\qquad z\to\infty.
\end{aligned}
\ee
This problem is solved by the Plemelj-Sokhotski formula, so that
\begin{equation}
 R^{(1)}(z)=\frac{1}{2\pi i}\int_{\partial U_a \cup \partial U_b \cup \partial U_c}\frac{\Delta_1(\zeta)}{\zeta-z}d\zeta.
\end{equation}
Note that the integration along the circles $\partial U_p$ is in the negative direction.
Below, we will find $ R^{(1)}(z)$ by computing the residues.
Retracing our transformations $R\to S\to X$ yields the asymptotic solution to the RH problem for $X(z)$.

\subsection{Extension to the limiting regime of Lemma \ref{seplemma}}\label{secRHext}
We now show that the solution to the Riemann-Hilbert problem for fixed $c<b<a<0$ extends to the regime where 
$a \to 0^{-}$ and $b-c \to 0^{+}$, as in Lemma \ref{seplemma}. We need to verify that the jumps of the $R$-RH problem remain small
in $L^{\infty}$ and $L^2$ norms. Then the $R$-RH problem is solvable by Neuman series
by standard arguments for small norm RH problems on contracting contours.

Let us first consider the jumps on $\partial U_a \cup \partial U_b \cup \partial U_c$. We will now show that in the regime
of Lemma \ref{seplemma}, by choosing the diameters of $U_a$, $U_b$, $U_c$ appropriately, we have uniformly in $z$
\be\label{estdU}
P_p(z)(P^{\infty}(z))^{-1}=I+
\begin{cases}
\bigO\left(\frac{1}{s^{3/2}\lvert a\rvert }\right),& z\in\partial U_a\\
\bigO\left(\frac{1}{s^{3/2}(b-c)}\right)& z\in\partial U_b\cup \partial U_c.
\end{cases}
\ee

First, uniformly for $a-b>\ve>0$ as $b\to c$ we have
\begin{equation}\label{J0BC}
    J_0=-\frac{\pi}{\sqrt{a-c}}\left(1+\frac{1}{2(a-c)}(b-c)+\bigO((b-c)^2)\right)
\end{equation}
and
\begin{equation}\label{J1BC}
    J_1=-\frac{c \pi}{\sqrt{a-c}}\left(1+\left(\frac{1}{c}+\frac{1}{2(a-c)}\right)(b-c)+\bigO((b-c)^2)\right),
\end{equation}
and, therefore,
\begin{equation}\label{q0expnatBC}
    q_0=\frac{ac}{2}+\frac{a}{4}(b-c)+\bigO((b-c)^2).
\end{equation}
With this expression for $q_0$ we may expand $q(p)$, $p=a,b,c$. Indeed, we find that uniformly for $a-b>\ve>0$ as $b\to c$,
\begin{equation}\label{qaBC}
    q(a)=\frac{a(a-c)}{2}+\bigO(b-c),
\end{equation}
\begin{equation}\label{qbBC}
    q(b)=-\frac{(a-2c)}{4}(b-c)+\bigO((b-c)^2),
\end{equation}
and
\begin{equation}\label{qcBC}
    q(c)=\frac{a-2c}{4}(b-c)+\bigO((b-c)^2).
\end{equation}
Note that, for $p=b,c$, \[\sqrt{z-p}=\bigO\left(\sqrt{b-c}\right),\quad b\to c,\quad z\in U_p.\]  This is because we have $\textmd{diam}(U_p)<\min \{\lvert \beta_1^*-b\rvert, \lvert \beta_1^*-c \rvert \} $, where $\beta_j^*$, $j=1,2,$ are the zeros of $\beta(z)-\beta(z)^{-1}$ (they satisfy 
\eqref{betastars}). We easily derive the expansions
\begin{equation}\label{explicitbetastars}
    \beta_1^*=c+\frac{b-c}{1+a-c}+\bigO((b-c)^2) , \qquad \beta_2^*=1+a-\frac{1}{4}\frac{b-c}{1+a-c}+\bigO((b-c)^2), \qquad b \to c.
\end{equation} 
 We may and do set $\textmd{diam}(U_p)=\epsilon^{\prime}\min \{\lvert \beta_1^*-b\rvert, \lvert \beta_1^*-c \rvert \}$, with $0<\epsilon^{\prime}<\frac{1}{2}$. Note that the diameter of $U_p$ decreases in this regime. On the other hand, we have that $\sqrt{z-a}=O(1)$, $a\to 0^{-}$, $z\in \overline{U_a}$. The diameter of $U_a$ does not need to decrease, set $\textmd{diam}(U_a)=\epsilon^{\prime}$. By \eqref{fdef} and \eqref{glocalexp} and the above estimates for $q(p)$, we have
\begin{equation}\label{173BDelta2}
\begin{aligned}
g_+(p)-g(z)&=f_p(z)^{1/2}=\bigO(b-c),\qquad   z\in \overline{U_p},\qquad p=b,c,\\
g_+(a)-g(z)&=f_a(z)^{1/2}=\bigO(|a|),\qquad   z\in \overline{U_a}.
\end{aligned}
\end{equation}
Recall also that $\Re g_+(p)=0$, $p=a,b,c$. In view of Lemma \ref{Pprhp},
to prove \eqref{estdU} it remains to show that $P^{\infty}(z)_{jk}$, $z\in \partial U_p$, are uniformly bounded in the prescribed regime.

By the first expansion in \eqref{explicitbetastars},
the functions $\beta(z)$ and $\beta(z)^{-1}$ are uniformly bounded on $\partial U_b$,  $\partial U_c$. Clearly, they are also bounded on 
$\partial U_a$.

To investigate the behaviour of $\theta$-functions, first analyze $\tau$, and for that write
\[\begin{aligned}
I_0&=
\int_{b+(b-c)^{1/2}}^a\frac{dx}{(x-c)(x-a)^{1/2}}\left(1+\bigO(\sqrt{b-c})\right)\\
&+
\frac{1}{(b-a)^{1/2}}\int_b^{b+(b-c)^{1/2}}\frac{dx}{((x-b)(x-c))^{1/2}}\left(1+\bigO(\sqrt{b-c})\right).
\end{aligned}
\]
Then it is easy to compute the integrals and obtain
\be\label{I0asymp}
I_0=\frac{1}{i\sqrt{a-c}}\left[-\log(b-c)+\log(16(a-c))\right]\left(1+\bigO(\sqrt{b-c})\right),
\ee
and, by using also \eqref{J0BC},
\be\label{tauasymp}
\tau=\frac{I_0}{J_0}=\frac{i}{\pi}\left[-\log(b-c)+\log(16(a-c))\right]\left(1+\bigO(\sqrt{b-c})\right).
\ee
Thus $i\tau \to -\infty$ logarithmically as $b\to c$. By the series representation for $\theta_3(\xi)$, this implies that the $\theta$-functions would converge to unity if their argument was any and fixed or was real.
However, since it is in general complex and depends on $a$, $b$, $c$, we need to show 
that $\theta(u(z)+(-1)^js^{3/2}\Omega+(-1)^k d)$ in the numerators in $P^{\infty}$
are bounded, and also that $\theta(u(z)\pm d)$ in the denominators are bounded away from zero. 

From the definition of $d$ in (\ref{definitionofd}), we have using \eqref{J0BC} and \eqref{explicitbetastars}  that, as $b\to c$,
\be\label{d}
\begin{aligned}
     d  &= \frac{1}{2J_0} \int_c^{\beta_1^*}\frac{dz}{((z-a)(z-b)(z-c))^{1/2}}=\left[z=c+(b-c)w \right]\\
    &=\frac{1}{2\pi}\left(1+\bigO(b-c)\right)\int_0^{\frac{\beta_1^*-c}{b-c}}\frac{dw}{(w(1-w))^{1/2}}=
    \frac{1}{2\pi}\left(1+\bigO(b-c)\right)\int_0^{\frac{1}{1+a-c}}\frac{dw}{(w(1-w))^{1/2}}<\frac{1}{2}.
\end{aligned} 
\ee
Therefore,
\be\label{dbcestimate}
0<d<\frac{1}{2}.
\ee

Using (\ref{explicitbetastars}), denote
\[ m:=\frac{1}{b-c}\min\{\lvert \beta_1^*-c \rvert, \lvert b-\beta_1^* \rvert \}=\min\{\frac{1}{1+a-c}+\bigO(b-c),\frac{a-c}{1+a-c}+\bigO(b-c)\}.  \]
Thus,
\[
\partial U_b=\{z=b+\epsilon' m (b-c)e^{i\phi}, 0\le\phi<2\pi\},\qquad \partial U_c=\{z=c+\epsilon' m (b-c)e^{i\phi}, 0\le\phi<2\pi\}.
\]
Arguing as in the case of $d$ above, we obtain
\be\label{uloc}
\begin{aligned}
u(z)&=\frac{\tau}{2}+\frac{1}{2\pi}\int_0^{\epsilon' m e^{i\phi}}\frac{dw}{(w(1-w))^{1/2}}\left( 1+ \bigO(b-c)\right),\qquad z\in \partial U_b,\\
u(z)&=\frac{1+\tau}{2}+\frac{1}{2\pi}\int_0^{\epsilon' m e^{i\phi}}\frac{dw}{(w(1-w))^{1/2}}\left( 1+ \bigO(b-c)\right),\qquad z\in \partial U_c,
\end{aligned}
\ee
and $|u(z)|\le\epsilon$ for $z\in \partial U_a$ for some $\epsilon>0$. Choosing $\epsilon'>0$ sufficiently small, we see that these equations and \eqref{dbcestimate} imply in particular that $u(z)\pm d$ is uniformly bounded away from the zero $\frac{1+\tau}{2}\mod \Lambda$ of $\theta_3(\xi)$ on $\partial U_a\cup\partial U_b\cup \partial U_c$. We then obtain for some constants
\be
\begin{aligned}
&0<\epsilon<\left|\theta(u(z)\pm d)\right|<C<\infty\qquad
\left|\theta(u(z)\pm d+r(a,b,c)))\right|<C'<\infty,\qquad \Im r(a,b,c)=0,\\
& z\in \partial U_a\cup\partial U_b\cup \partial U_c.
\end{aligned}
\ee
Thus $P^{\infty}$ is bounded on $\partial U_a\cup\partial U_b\cup \partial U_c$ and therefore \eqref{estdU} holds.

The analysis of the jumps of $R$  on $\Sigma_R\setminus \partial U_a \cup \partial U_b \cup \partial U_c$ is similar. 
(Note, in particular, that $0<\Im u(z) <|\tau|$ on the contour.)
We obtain
\begin{equation}\label{estSR}
  P^{\infty}(z;s)v_S(z;s)\left(P^{\infty}(z;s)\right)^{-1}=
  I+\bigO\left(e^{-c's^{3/2}\min\{\lvert a\rvert, b-c\}\max\{1,|z|^{3/2}\}}\right),\qquad c'>0,
\end{equation}
on $\Sigma_R\setminus \partial U_a \cup \partial U_b \cup \partial U_c$.

To conlude, we see from \eqref{estdU} and \eqref{estSR} that in the regime of Lemma \ref{seplemma}, where
$a-b>\ve>0$,
\begin{equation}\label{regime}
    a=-\frac{(\log s)^{1/8}}{s}, \qquad b=c+\frac{2(\log s)^{1/8}}{s^{3/2}},\qquad s\to +\infty,
\end{equation}
we have
\begin{equation}\label{estfinal}
\begin{aligned}
P_p(z)(P^{\infty}(z))^{-1}&=I+
\bigO\left(\frac{1}{2(\log s)^{1/8}}\right),\qquad z\in\partial U_a\cup\partial U_b\cup \partial U_c,\\
  P^{\infty}(z;s)v_S(z;s)\left(P^{\infty}(z;s)\right)^{-1}&=
  I+\bigO\left(e^{-2c'(\log s)^{1/8}\max\{1,|z|^{3/2}\}}\right),\qquad z\in\Sigma_R\setminus \partial U_a \cup \partial U_b \cup \partial U_c.
\end{aligned}
\end{equation}
Therefore the $R$-RH problem is solvable in this regime, and we can write
\be\label{Rext}
R(z)= I + R^{(1)}(z)+\bigO\left(\frac{1}{s^3\min\{|a|^2,(b-c)^2\}}\right),\qquad 
R^{(1)}(z)=\bigO\left(\frac{1}{s^{3/2}\min\{|a|,b-c\}}\right),
\ee
uniformly in $z$ and uniformly in the range from fixed $c<b<a<0$ to the regime \eqref{regime}.

\section{Proof of Theorem \ref{Mainthm}}\label{proofofthm}
 We start with the differential identity \eqref{diffid} for $p=b$, and write $P^{\mathrm{Ai}}(sJ)=P_s^{\mathrm{Ai}}(a,b)$, to indicate the values of $a$ and $b$,
\begin{equation}
 \frac{d}{db}\log P^{\mathrm{Ai}}(sJ)=
    \frac{d}{db}\log P_s^{\mathrm{Ai}}(a,b)
    =-\lim_{z\to b}\frac{1}{2\pi i}\left( X^{-1}(z)\frac{d}{dz}X(z)\right)_{21},    
\end{equation}
where the limit is taken from outside the lens and with $\Im z>0$. We first concentrate on the case of $c<b<a<0$ fixed,
and then provide an extension to the asymptotic regime of Lemma \ref{seplemma}.

By the definitions of $S(z)$ and $R(z)$, we have that
in the region $U_b\cap \{z: \Im z>0\}$, and outside the lens,
\begin{equation}
X^{-1}X^{\prime}=  e^{s^{3/2}g(z)\sigma_3}P_b^{-1}R^{-1}R^{\prime}P_b e^{-s^{3/2}g(z)\sigma_3}+e^{s^{3/2}g(z)\sigma_3}P_b^{-1}  P_b^{\prime}e^{-s^{3/2}g(z)\sigma_3}-s^{3/2}g^{\prime}(z)\sigma_3.
\end{equation}
By definition \eqref{defP} of $P_b(z)$, in particular in the preimage $\zeta^{-1}(\Omega_1)$ (which is outside the lens), 
\[ P_b(z)=E_b(z)\sigma_3 \Psi(s^3(g(z)-g(b))^2)\sigma_3 e^{s^{3/2}g(z)\sigma_3},\]      
and therefore 
\[P_b^{-1}P_b^{\prime}=e^{-s^{3/2}g(z)\sigma_3}\sigma_3\Psi^{-1}\sigma_3E_b^{-1}E_b^{\prime}\sigma_3 \Psi\sigma_3 e^{s^{3/2}g(z)\sigma_3}+e^{-s^{3/2}g(z)\sigma_3 }\sigma_3\Psi^{-1}\Psi^{\prime}\sigma_3 e^{s^{3/2}g(z)\sigma_3} +s^{3/2}g^{\prime}(z)\sigma_3.  \]
Moreover, note that for $z$ in a neighbourhood of $b$,
\begin{equation}\label{RR}
    R^{-1}(z)R^{\prime}(z)=\frac{dR^{(1)}}{dz}(z)+\bigO\left(\frac{1}{s^{3}}\right), \qquad s\to \infty,
\end{equation}
where
\begin{equation}
\frac{dR^{(1)}}{dz}(z)
   =\frac{1}{2\pi i}\int_{\partial U_a \cup \partial U_b \cup \partial U_c}\frac{\Delta_1(\zeta;s)}{(\zeta-z)^2}d\zeta.
\end{equation}
Note that the (uniform) estimate for the error term in  \eqref{RR}, follows from the 
differentiability of the asymptotic expansion of $R$, which, in turn, follows from the
uniform estimate for $R$, in \eqref{RUNIF}, and Cauchy's integral formula. 

Starting with the (differentiable) expansions, see e.g. \cite{AS}, for the Bessel functions
\[I_0(z)=1+\frac{z^2}{4}+\bigO(z^4),\qquad z\to 0,\]
and
\[K_0(z)=-\left( \log\frac{z}{2}+\gamma\right)I_0(z)+\frac{z^2}{4}+\bigO(z^4),\qquad z\to 0,\]
where $\gamma$ is Euler's constant, it is straightforward to verify that the function \eqref{Psi1}
\begin{equation}\label{Jzetaexpn}
\Psi(\zeta)=e^{-\frac{i\pi}{4}\sigma_3}\pi^{\frac{1}{2}\sigma_3} \begin{pmatrix}1+\frac{\zeta}{4}+\bigO(\zeta^2) & \frac{1}{2\pi i} \log \zeta +\bigO(1) \\ \frac{i \pi}{2}\zeta+\bigO(\zeta^{2}) & 1+\bigO(\zeta\log \zeta)  \end{pmatrix},\qquad \zeta\to 0.
\end{equation}
Using this expression, the definition
\[ E_b(z;s)=P^{\infty}(z;s)e^{-s^{3/2}g(b)\sigma_3}\sigma_3 N_0^{-1} \sigma_3 f_b(z)^{\frac{1}{4}\sigma_3}s^{\frac{3}{4}\sigma_3}, \]
and the fact that $P^{\infty}(z;s)$ is bounded, as $s\to\infty$, we estimate the error term in $\lim_{z\to b}X^{-1}(z)X^{\prime}(z)$, arising from that in \eqref{RR}, as follows
\be\label{RRError}
\lim_{z\to b}\left(
e^{s^{3/2}g(z)\sigma_3}P_b^{-1}
\bigO\left(\frac{1}{s^{3}}\right)P_b e^{-s^{3/2}g(z)\sigma_3}\right)_{21}
=\bigO\left(\frac{b-c}{s^{3/2}}\right).
\ee

Thus, in the region $W_b:=U_b\cap \{z: \Im z>0\}\cap \zeta^{-1}(\Omega_1)$,
\begin{equation}\label{218old}
- \frac{1}{2\pi i}\lim_{z\to b}\left( X^{-1}(z)\frac{d}{dz}X(z)\right)_{21}= \lim_{z\to b}(\tau_1(z)+\tau_2(z)+ \tau_3(z))_{21}+\bigO\left(\frac{b-c}{s^{3/2}}\right).
\end{equation}
where
\begin{align}
     \mathcal{\tau}_1(z)&=-\frac{1}{2\pi i}\sigma_3\Psi^{-1}(s^3 f_b(z))\frac{d\Psi}{dz}(s^3 f_b(z))\sigma_3, \label{tau121} \\
     \mathcal{\tau}_2(z)&=-\frac{1}{2\pi i}\sigma_3\Psi^{-1}(s^3 f_b(z))\sigma_3E_b^{-1}(z)\frac{dE_b}{dz}(z)\sigma_3 \Psi(s^3 f_b(z))\sigma_3,\label{tau221} \\
     \mathcal{\tau}_3(z)&=-\frac{1}{2\pi i}e^{s^{3/2}g(z)\sigma_3}P_b^{-1}(z)\frac{dR^{(1)}}{dz}(z)P_b(z) e^{-s^{3/2}g(z)\sigma_3}.\label{tau321}
\end{align}
 With these we have the formula
\begin{equation}\label{preliminaryformula}
       \frac{d}{db}\log P_s^{\mathrm{Ai}}(a,b)=  \lim_{z\to b} \left(\mathcal{\tau}_1(z)+\mathcal{\tau}_2(z)+\mathcal{\tau}_3(z)\right)_{21}+ \bigO\left(\frac{1}{s^{3/2}}\right),
\end{equation}
for fixed $c<b<a<0$,
where the limit is taken along a path in $W_b$. 

To extend \eqref{preliminaryformula} to the regime $b\to c$, $a\to 0$
of Lemma \ref{seplemma}, we rely on the results of the previous section.
Considering $R^{-1}R'$ by using \eqref{Rext}, we obtain 
\begin{equation}\label{RRext}
    R^{-1}(z)R^{\prime}(z)=\frac{dR^{(1)}}{dz}(z)+
    \bigO\left(\frac{1}{s^3\min\{|a|^2,(b-c)^2\}}\right), \qquad s\to \infty,
\end{equation}
and hence (cf. \eqref{218old})
\begin{equation}\label{prelimitingformula}
       \frac{d}{db}\log P_s^{\mathrm{Ai}}(a,b)=  \lim_{z\to b} \left(\mathcal{\tau}_1(z)+\mathcal{\tau}_2(z)+\mathcal{\tau}_3(z)\right)_{21}+ \bigO\left(\frac{1}{s^{3/2}\min\{|a|^2,(b-c)\}}\right),
\end{equation}
uniformly in the range from fixed $c<b<a<0$ to the regime \eqref{regime},
where the limit is taken along a path in $W_b$.

In the next 3 sections, we prove the following lemmata (with the notation from the introduction).
\begin{lemma}\label{mainterm}
\begin{equation}
   \lim_{z\to b} \mathcal{\tau}_1(z)_{21}=- s^3 \frac{d}{db}\alpha_2.
\end{equation}
\end{lemma}
\begin{lemma}\label{oscillatoryterm}
 \begin{equation}\lim_{z\to b}\mathcal{\tau}_2(z)_{21}= \frac{d}{db}\log \theta_3(s^{3/2}\Omega;\tau)-\frac{d\tau}{db}\frac{\partial}{\partial\tau}\log \theta_3(s^{3/2}\Omega;\tau).
 \end{equation}
\end{lemma}
\begin{lemma}\label{constantterm}With $\omega=s^{3/2}\Omega$,
\begin{equation}
    \int_0^1\lim_{z\to b}\mathcal{\tau}_3(z)_{21}d\omega=-\frac{1}{2}\frac{d}{db}\log \lvert J_0 \rvert-\frac{1}{8}\frac{d}{db}\log \lvert q(a)q(b)q(c)\rvert +\frac{d\tau}{db}\int_0^1\frac{\partial}{\partial \tau} \log\theta_3(\omega;\tau)d\omega.
\end{equation}
\end{lemma}

Using these results, we may now write the differential identity \eqref{prelimitingformula} in the following explicit form.

\begin{prop}\label{prop2b} (Asymptotic form of the differential identity for $b$)
Let $\ve>0$ be fixed.
Then uniformly for $c<b_0<b<b_1<a<0$, where
$b_0,b_1\in[c+\frac{2t_0}{s^{3/2}},a-\ve]$, $a\le -\frac{t_1}{s}$, as $s\to\infty$,
with $t_0=t_1=(\log s)^{1/8}$,
\be\label{prop2bint}
  \frac{d}{db}\log P_s^{\mathrm{Ai}}(a,b)=  \frac{d}{db} D(a,b) +
  \Theta_b+\widetilde\Theta_b+
  \bigO\left(\frac{1}{s^{3/2}|a|^2}\right)+
   \bigO\left(\frac{1}{s^{3/2}(b-c)}\right),
\ee
where
\be\label{Dform}
D(a,b)= 
-\alpha_2 s^3 +\log \theta_3(s^{3/2}\Omega;\tau)-\frac{1}{2}\log |J_0| -\frac{1}{8}\log | q(a)q(b)q(c) |,
\ee
and
\begin{align}
 \Theta_b&=\lim_{z\to b}\tau_3(z)_{21}-\int_0^1\lim_{z\to b}\tau_3(z)_{21}d\omega,\\
 \widetilde\Theta_b&=\frac{d\tau}{db}\int_0^1\frac{\partial}{\partial \tau} \log\theta_3(\omega;\tau)d\omega-
\frac{d\tau}{db}\frac{\partial}{\partial\tau}\log \theta_3(s^{3/2}\Omega;\tau).
\end{align}
Moreover,
\be\label{Omegaint}
\int_{b_0}^{b_1}\Theta_b db,\quad \int_{b_0}^{b_1}\widetilde\Theta_b db\quad =
\bigO\left(\frac{1}{s^{3/2}|a|}\right)+
   \bigO\left(\frac{1}{s^{3/2}(b_0-c)}\right).
\ee
\end{prop}

\begin{proof} By \eqref{prelimitingformula} and 3 lemmata above, we only need to prove
\eqref{Omegaint}. First we show that
\be\label{Omegaint1}
\int_{b_0}^{b_1}\Theta_b db=
\bigO\left(\frac{1}{s^{3/2}|a|}\right)+
   \bigO\left(\frac{1}{s^{3/2}(b_0-c)}\right),
\ee
i.e. replacing $\lim_{z\to b}\tau_3(z)_{21}$ by its average w.r.t. $\omega=s^{3/2}\Omega$ under the sign of the integral introduces only a small error. 
Let 
\[
f(\omega;a,b,c)=\lim_{z\to b}\tau_3(z)_{21},\qquad \omega=s^{3/2}\Omega.
\]
First, let $c<b<a<0$ be fixed. Then $f$ is an analytic function of $\omega$. Let $f_n$ denote its Fourier coefficients
\begin{equation}
    f(\omega;a,b,c)=\sum_{n\in \mathbb{Z}}f_n(a,b,c)e^{2\pi i n \omega}.
\end{equation}
For $n\neq 0$, it follows from integration by parts that
\begin{align}\label{fourierbyparts}
    \Bigl \lvert \int_{b_0}^{b_1} f_n(a,b,c) & e^{2\pi i n s^{3/2}\Omega}db \Bigr \rvert=\\
    &=\frac{1}{2\pi  \lvert n \rvert s^{3/2}}\Biggl \lvert \left[ \frac{f_n(a,b,c)e^{2\pi i n s^{3/2}\Omega}}{\frac{d\Omega}{db}} \right]_{b_0}^{b_1}-\int_{b_0}^{b_1}e^{2\pi i n s^{3/2}\Omega}\frac{\partial}{\partial b}\left(\frac{f_n(a,b,c)}{\frac{d\Omega}{db}} \right)db\Biggr \rvert.
\end{align}
From \eqref{idder} of Lemma \ref{thetalemma}, we see that $\frac{d\Omega}{db}(a,b,c)$ is a strictly positive differentiable function
of the parameters $a$, $b$, and $c$, bounded away from zero if $c<b<a<0$ are bounded away from each other. 
Furthermore, the functions $f_n$,  $\partial f_n/\partial b$ decrease with $n$ exponentially.
Therefore 
\begin{equation}\label{fourieraverageerror}
  \int_{b_0}^{b_1} f(\omega;a,b,c)db=\sum_{n\in \mathbb{Z}}
 \int_{b_0}^{b_1}  f_n(a,b,c)e^{2\pi i n \omega}db=
  \int_{b_0}^{b_1} f_0(a,b,c)db+\bigO\left(\frac{1}{s^{3/2}} \right) ,\quad s\to \infty.
\end{equation}

Now let $c<b<a<0$ be in the range of the Proposition.  We will show that
\begin{equation}\label{fe}
  \int_{b_0}^{b_1} f(\omega;a,b,c)db=
  \int_{b_0}^{b_1} f_0(a,b,c)db+
  \bigO\left(\frac{1}{s^{3/2}|a|}\right)+
   \bigO\left(\frac{1}{s^{3/2}(b_0-c)}\right),\qquad s\to \infty,
\end{equation}
which implies \eqref{Omegaint1}.

By \eqref{idder}, \eqref{J0BC}, \eqref{qbBC},
\begin{equation}
    \frac{d\Omega}{db}=\frac{a-2c}{\pi\sqrt{a-c}}+\bigO(b-c),\qquad  \frac{d^2\Omega}{db^2}=\bigO(1),
\end{equation}
as $b\to c$ uniformly for $a-b>\ep>0$. We note, in particular, that $\frac{d\Omega}{db}$
remains bounded away from zero in the regime of the Proposition. To apply the above Fourier series arguments and deduce \eqref{fe},
it suffices to show that
\be
f_n(a,b,c)=\bigO\left(\frac{1}{n(b-c)}\right)+\bigO\left(\frac{1}{n|a|}\right),\qquad 
\frac{\partial}{\partial b}f_n(a,b,c)=\bigO\left(\frac{1}{n(b-c)^2}\right)+\bigO\left(\frac{1}{n|a|}\right).
\ee
in the regime of the Proposition. Since
\begin{equation}\label{Fdom} 
\left|f_n(a,b,c)\right|=\left|\int_0^1f(\omega;a,b,c)e^{-2\pi i n \omega}d\omega\right| 
=\left|\frac{1}{2\pi n}\int_0^1\frac{\partial}{\partial\omega}f(\omega;a,b,c)e^{-2\pi i n \omega}d\omega\right|,
\qquad n\neq 0,
\end{equation}
and similarly for $\frac{\partial}{\partial b}f_n(a,b,c)$, we only need to show that
\begin{equation}\label{bsuff}
\frac{\partial}{\partial \omega}f(\omega;a,b,c)=\mathcal O\left(\frac{1}{b-c}\right)+\bigO\left(\frac{1}{|a|}\right)
, \qquad \frac{\partial}{\partial \omega}\frac{\partial}{\partial b}f(\omega;a,b,c)=\mathcal O\left(\frac{1}{(b-c)^2}\right)+\bigO\left(\frac{1}{|a|}\right).
\end{equation}

One can use directly the definition of $\lim_{z\to b}\tau_3(z)_{21}$, but it is easier to consider its simplified form in Section \ref{sec-const}
below.
By \eqref{tauasymp}, we have uniformly in $\omega$, the estimates for the $\theta$-functions 
\be\label{thetalim}
\theta_j(\omega)=1+\bigO(b-c),\qquad j=1,2,3,4,
\ee
and their derivatives $\theta'_j(\omega)=\bigO(b-c)$.
The analysis of the previous section shows that the matrix elements \eqref{thetajk}, $\theta_{jk}(u(z))$ are bounded on $\partial U_b$,
and so are their derivatives w.r.t. $\omega$.
Using \eqref{uloc}, \eqref{tauasymp}, we also conclude that
\[
 \frac{\partial u}{\partial b},\qquad
  \frac{\partial \tau}{\partial b},\qquad
   \frac{\partial \beta}{\partial b},\qquad
 \frac{\partial \beta^{-1}}{\partial b}
\]
are all $\bigO(1/(b-c))$ uniformly on $\partial U_b$. Finally, note that we can use the heat equation to differentiate $\theta$-functions w.r.t. $\tau$.
The estimates \eqref{bsuff} follow then from \eqref{tau3TT}, \eqref{firstACresidues}, and \eqref{penultimateAterm} with \eqref{AABB} below.
Thus, we obtained \eqref{fe}.




In view of the above lemmata \ref{mainterm}, \ref{oscillatoryterm}, \ref{constantterm}, and the formula \eqref{prelimitingformula},
it remains to show the second estimate in \eqref{Omegaint}, i.e. that 
\[ 
\widetilde\Theta_b=
\frac{d\tau}{db}\int_0^1\frac{\partial}{\partial\tau} \log\theta_3(\omega;\tau)d\omega-\frac{d\tau}{db}\frac{\partial}{\partial\tau}\log \theta_3(s^{3/2}\Omega;\tau)\]
becomes $\bigO\left(\frac{1}{s^{3/2}\min\{|a|,b-c\}}\right)$ after integration. But this follows from the above considerations similarly to (and simpler than) \eqref{fe}.
\end{proof}

Using the analogous results for the differential identity at $a$, we follow the above methods and arrive in Section \ref{secDEATA}, similarly, at

\begin{prop}\label{DEataasympt} (Asymptotic form of the differential identity for $a$)
Let $\ve>0$ be fixed.
Then uniformly for $c<b<a_0<a<a_1<0$, where
$a_0,a_1\in [b+\ve,-\frac{t_1}{s}]$, $b\ge c+\frac{2t_0}{s^{3/2}}$ as $s\to\infty$,
with $t_0=t_1=(\log s)^{1/8}$,
\be\label{prop2aint}
  \frac{d}{da}\log P_s^{\mathrm{Ai}}(a,b)=  \frac{d}{da} D(a,b) +
  \Theta_a+\widetilde\Theta_a+
  \bigO\left(\frac{1}{s^{3/2}|a|}\right)+
   \bigO\left(\frac{1}{s^{3/2}(b-c)^2}\right),
\ee
where $D(a,b)$ is given by \eqref{Dform}, and
\begin{align}
\Theta_a&=\lim_{z\to a}\tau_3(z)_{21}-\int_0^1\lim_{z\to a}\tau_3(z)_{21}d\omega,\\
 \widetilde\Theta_a&=\frac{d\tau}{da}\int_0^1\frac{\partial}{\partial \tau} \log\theta_3(\omega;\tau)d\omega-
\frac{d\tau}{da}\frac{\partial}{\partial\tau}\log \theta_3(s^{3/2}\Omega;\tau).
\end{align}
Moreover,
\be
\int_{a_0}^{a_1}\Theta_a da,\quad \int_{a_0}^{a_1}\widetilde\Theta_a da\quad =
\bigO\left(\frac{1}{s^{3/2}|a|}\right)+
   \bigO\left(\frac{1}{s^{3/2}(b-c)}\right).
\ee
\end{prop}

Because of the form of the error terms, we will need to move the edges $a$ and $b$ first
simultaneously. To do this we will use the following Proposition which easily follows from the 2 ones above and some similar additional arguments. For fixed $\alpha$, $\beta$, we denote $a=\alpha-x$, $b=\beta+x$ and remark that
\[
\frac{d}{dx}=\frac{\partial}{\partial b}-\frac{\partial}{\partial a}.
\]
\begin{prop}\label{DEatabasympt} (Asymptotic form of the differential identity for $a=\alpha-x$, $b=\beta+x$)
Let $\ve>0$ be fixed.
Then uniformly for $0\le x_0<x<x_1\le\frac{\alpha-\beta}{2}-\ve$, where
$a=\alpha-x$, $b=\beta+x$, and $\beta= c+\frac{2t_0}{s^{3/2}}$, $\alpha= -\frac{t_1}{s}$
as $s\to\infty$
with $t_0=t_1=(\log s)^{1/8}$,
\be\label{prop2x}
  \frac{d}{dx}\log P_s^{\mathrm{Ai}}(a,b)=  \frac{d}{dx} D(a,b) 
  -\Theta_a-\widetilde\Theta_a+\Theta_b+\widetilde\Theta_b+
  \bigO\left(\frac{1}{s^{3/2}|a|^2}\right)+
   \bigO\left(\frac{1}{s^{3/2}(b-c)^2}\right),
\ee
where $D(a,b)$ is given by \eqref{Dform}, and $\Theta_a$, $\widetilde\Theta_a$, $\Theta_b$,
$\widetilde\Theta_b$ as in Propositions \ref{DEataasympt}, \ref{prop2b} above.
Moreover,
\be
\int_{x_0}^{x_1}\Theta_a dx,\quad \int_{x_0}^{x_1}\widetilde\Theta_a dx,\quad 
\int_{x_0}^{x_1}\Theta_b dx,\quad \int_{x_0}^{x_1}\widetilde\Theta_b dx\quad 
=
\bigO\left(\frac{1}{s^{3/2}|a|}\right)+
   \bigO\left(\frac{1}{s^{3/2}(b-c)}\right).
\ee
\end{prop}

With these propositions, we may prove Theorem \ref{Mainthm}. We will integrate to reach the desired values of $a=a_0$, $b=b_0$ of the Theorem.
Set $\alpha= -\frac{t_1}{s}$, $\beta=c+\frac{2t_0}{s^{3/2}}$, $t_0=t_1=(\log s)^{1/8}$.
First, by Proposition \ref{DEatabasympt}, we integrate over $x$ from $x=0$ to $x=\alpha-a_0$,
which fixes the desired value of $a$:
\be\begin{aligned}
&\log P_s^{\mathrm{Ai}}(a_0,\beta+\alpha-a_0)-\log P_s^{\mathrm{Ai}}(\alpha,\beta)=
\int_{0}^{\alpha-a_0} \frac{d}{dx}\log P_s^{\mathrm{Ai}}(a,b)dx=\\
& D(a_0,\beta+\alpha-a_0)-D(\alpha,\beta)
+\bigO\left(\frac{1}{s^{3/2}|\alpha|}\right)+
   \bigO\left(\frac{1}{s^{3/2}(\beta-c)}\right).
   \end{aligned}
\ee
Next, using Proposition \ref{prop2b}, we integrate over $b$ from the present value $b=\beta+\alpha-a_0$ to $b=b_0$:
\be\begin{aligned}
&\log P_s^{\mathrm{Ai}}(a_0,b_0)-\log P_s^{\mathrm{Ai}}(a_0,\beta+\alpha-a_0)=
\int_{\beta+\alpha-a_0}^{b_0} \frac{d}{db}\log P_s^{\mathrm{Ai}}(a,b)db=\\
& D(a_0,b_0)-D(a_0,\beta+\alpha-a_0)
+\bigO\left(\frac{1}{s^{3/2}}\right).
   \end{aligned}
\ee

The sum of these expressions gives
\begin{equation}\label{Psum}
\log P_s^{\mathrm{Ai}}(a_0,b_0)=D(a_0,b_0)-D(\alpha,\beta)+\log P_s^{\mathrm{Ai}}(\alpha,\beta)+\bigO\left(\frac{1}{(\log s)^{1/8}}\right).
\end{equation}

Note that $\log P_s^{\mathrm{Ai}}(\alpha,\beta)$ in \eqref{Psum} may be replaced with the expression in Lemma \ref{seplemma}. Thus it remains only to expand the expression for $D(\alpha,\beta)$ when $s\to\infty$. To this end we prove the following 
\begin{prop}\label{limitregimeprop}
 For $\alpha=-\frac{t_1}{s}$ and $\beta =c+\frac{2t_0}{s^{3/2}}$, $t_0=t_1=(\log s)^{1/8}$,
\begin{equation}\label{Dalbe} 
    D(\alpha,\beta)=-\frac{\lvert c \rvert t_0^2}{2}-\frac{1}{4}\log(\sqrt{\lvert c \rvert} t_0) -\frac{t_1^3}{12}-\frac{1}{8}\log t_1 +\frac{1}{2}\log s+\frac{1}{8}\log 2 -\frac{1}{2}\log \pi+o(1), \quad s\to \infty.
\end{equation}
\end{prop}

\begin{proof}The formula is obtained by expanding all terms in $D(a,b)$ with $a=-\frac{t_1}{s}$ and $b=c+\frac{2t_0}{s^{3/2}}$,
$s\to\infty$. First, 
by definition of $\alpha_2$ in \eqref{definitionofalpha2} and the exansion of $q_0$ (\ref{q0expnatBC}),
\be\label{121}
\alpha_2 =-\frac{a^3}{12}+\frac{(a-2c)^2}{32(a-c)}(b-c)^2+\bigO((b-c)^3)=
   \frac{t_1^3}{12s^3}+\frac{\lvert c \rvert t_0^2}{2s^3}\left(1+o(1)\right).
\ee
Next, by \eqref{thetalim}, $\theta_3(s^{3/2}\Omega;\tau)=1+\bigO(b-c)$.
The expansion (\ref{J0BC}) implies that
\begin{equation}\label{118}
    -\frac{1}{2}\log \lvert J_0 \rvert =-\frac{1}{2}\log \pi +\frac{1}{4}\log \lvert c \rvert +\bigO(b-c).
\end{equation}
Finally, the expansions (\ref{qaBC}), (\ref{qbBC}), and (\ref{qcBC}) imply that
\begin{equation}\label{125}
\begin{aligned}
    -\frac{1}{8}\log \lvert q(a)q(b)q(c)\rvert &=-\frac{1}{8}\log \left| \frac{a(a-c)}{2}\left(\frac{a-2c}{4}(b-c)\right)^2(1+\bigO(b-c))\right| \\
    &= 
    -\frac{1}{4}\log(\sqrt{\lvert c \rvert} t_0)-\frac{1}{8}\log t_1+\frac{1}{2}\log s +\frac{1}{8}\log 2-\frac{1}{4}\log \lvert c \rvert+o(1).
    \end{aligned}
\end{equation}
Combining (\ref{121}), (\ref{118}), and (\ref{125}), we obtain \eqref{Dalbe}.
\end{proof}

On substituting the expressions \eqref{seplemmaexpansion} of Lemma \ref{seplemma} and \eqref{Dalbe} of
Proposition \ref{limitregimeprop}, into \eqref{Psum}, Theorem \ref{Mainthm} follows
with $a=a_0$ and $b=b_0$.

In the following sections we prove the aforementioned lemmata \ref{mainterm}, \ref{oscillatoryterm}, and \ref{constantterm}, 
and their analogues for the differential identity in $a$.

\subsection{The main term. Proof of Lemma \ref{mainterm}.}
From \eqref{Jzetaexpn} we easily obtain, 
recalling that $\zeta=\zeta(z)=s^3 f_b(z)=s^3 (g(z)-g(b))^2$ and using \eqref{glocalexp},
\begin{equation}\label{mainprelim}
\begin{aligned}
\lim_{z\to b}{\tau}_1(z)_{21}&=
-\frac{1}{2\pi i}\lim_{z\to b}\left(\sigma_3\Psi^{-1}(s^3f_b(z))\frac{d\Psi}{dz}(s^3f_b(z))\sigma_3\right)_{21}\\
&=-\frac{1}{2\pi i}\lim_{\zeta\to 0}\left(\sigma_3\Psi^{-1}(\zeta)\frac{d\Psi}{d\zeta}(\zeta)\sigma_3\right)_{21} s^{3} f'_b(b)
=-s^3\frac{q(b)^2}{(a-b)(b-c)}.
\end{aligned}
\end{equation}

On the other hand, using \eqref{defq1}, \eqref{ddbq0}, and \eqref{dbJ0}, 
we differentiate  $\alpha_2$ given by \eqref{alpha2} in Lemma \ref{lemmag}, and obtain
\begin{equation}\label{ddbalpha2}
    \frac{d}{db}\alpha_2=\frac{q(b)^2}{(a-b)(b-c)}.
\end{equation}
Comparing this with \eqref{mainprelim}, we obtain Lemma \ref{mainterm}.

\begin{remark}
Identity (\ref{ddbalpha2}) is suitable to prove that $\alpha_2>0$. Note that this property is important since it determines exponential \emph{decay} of the determinant itself. It is clear that the right hand side of \eqref{ddbalpha2} is positive for all values of $a$, $b$, and $c$. Since (cf Section \ref{secRHext})
$\alpha_2 \to -\frac{a^3}{12}>0$ as  $b \to c $, we have that
\begin{equation}
    \alpha_2(b)=\alpha_2(b=c)+\int_c^b\frac{q(b)^2}{(a-b)(b-c)}db>0.
\end{equation}
\end{remark}

\subsection{The oscillatory term. Proof of Lemma \ref{oscillatoryterm}.}

We now study the term
\begin{equation}
    \mathcal{\tau}_2(z)=-\frac{1}{2\pi i}\sigma_3\Psi^{-1}(\zeta)\sigma_3E_b^{-1}(z)E_b^{\prime}(z)\sigma_3 \Psi(\zeta)\sigma_3.
\end{equation}
From the definition
\[ E_b(z)=P^{\infty}(z)e^{-s^{3/2}g_+(b)\sigma_3}\sigma_3 N_0^{-1} \sigma_3 f_b(z)^{\frac{1}{4}\sigma_3}s^{\frac{3}{4}\sigma_3},\]
we have that
\begin{align}
&E_b^{-1}(z)E_b^{\prime}(z)=
 \frac{\sigma_3}{4}\frac{f_b^{\prime}(z)}{f_b(z)}\\
&+ s^{-\frac{3}{4}\sigma_3} f_b^{-\frac{1}{4}\sigma_3}(z)  \sigma_3 N_0 \sigma_3 e^{s^{3/2}g_+(b)\sigma_3}\left(P^{\infty}(z)\right)^{-1} \left( P^{\infty}(z) \right)^{\prime} e^{-s^{3/2}g_+(b)\sigma_3}\sigma_3 N_0^{-1} \sigma_3 f_b^{\frac{1}{4}\sigma_3}(z)  s^{\frac{3}{4}\sigma_3}.
\notag
\end{align}
It follows from the expansions of $\Psi(\zeta)$ in (\ref{Jzetaexpn}) and of  $f_b(z)$ in \eqref{glocalexp}, that the contribution to $\tau_2$ from the first term
in the above expression is
\begin{equation}\label{EBsomething}
  -\frac{1}{2\pi i}  \frac{1}{4}\frac{f_b^{\prime}(z)}{f_b(z)}\left(\sigma_3\Psi^{-1}(\zeta)\sigma_3 \Psi(\zeta)\sigma_3\right)_{21}\to s^3\frac{q(b)^2}{2(a-b)(b-c)},\quad z\to b.
\end{equation}

Next,
with the notation for $P^{\infty}$ in \eqref{thetaij}, we obtain,
using the fact that
$\det P^{\infty}(z)=1$, $\det (P^{\infty})^{\prime}(z)=0$ (see Lemma \ref{thetalemma} (iv)),
\begin{equation}\label{ABC}
    e^{s^{3/2}g_+(b)\sigma_3}\left( P^{\infty}(z)\right)^{-1}\left(P^{\infty}(z)\right)^{\prime} e^{-s^{3/2}g_+(b)\sigma_3}=\begin{pmatrix} A & B \\ C & -A \end{pmatrix},
\end{equation}
with
\begin{equation}\label{A0}
    A=\theta_{22}\frac{d\theta_{11}}{dz}N_{11}^2+\theta_{11}\theta_{22}N_{11}\frac{d N_{11}}{dz}+\theta_{12}\frac{d\theta_{21}}{dz} N_{12}^2+\theta_{12}\theta_{21}N_{12}\frac{d N_{12}}{dz},
\end{equation}
\begin{equation}\label{B0}
    B=\left[
    \left(\theta_{22}\frac{d\theta_{12}}{dz}-\theta_{12}\frac{d\theta_{22}}{dz}\right)N_{11}N_{12}+
    \theta_{12}\theta_{22}\left(N_{11}\frac{dN_{12}}{dz}-N_{12}\frac{d N_{11}}{dz}\right)\right]e^{2g_+(b)s^{3/2}},
\end{equation}
and
\begin{equation}\label{C0}
    C=\left[
    \left(\theta_{21}\frac{d\theta_{11}}{dz}-\theta_{11}\frac{d\theta_{21}}{dz}\right)N_{11}N_{12}-
    \theta_{21}\theta_{11}\left(N_{11}\frac{dN_{12}}{dz}-N_{12}\frac{d N_{11}}{dz}\right)\right]e^{-2g_+(b)s^{3/2}},
\end{equation}
where we omit the arguments for brevity.

In this notation, we determine using the expansion of $\Psi(\zeta)$ in (\ref{Jzetaexpn}) that
\begin{align}\label{EBEXPN}
   \lim_{z\to b}{\tau}_2(z)_{21}&=\lim_{z\to b}\left(\frac{1}{4}\left(2A+i(B+C)\right)s^{3/2}f_b^{1/2}(z) -\frac{i}{4}(B-C)s^3f_b(z)\right)+s^3\frac{q(b)^2}{2(a-b)(b-c)},
\end{align}
where we substituted (\ref{EBsomething}). We first evaluate the term with $B-C$ in this expression.
Writing
\be
\frac{d}{dz}\theta_{jk}=\theta'_{jk}\frac{du}{dz},\qquad \theta'_{jk}(u)=\frac{d}{du}\theta_{jk}(u(z)),
\ee
we have by \eqref{uexp}
\[
\frac{du}{dz}=\frac{u_{0,b}}{2}(z-b)^{-1/2}\left(1+\bigO(z-b)\right),\qquad z\to b.
\]
Replacing $\frac{d}{dz}\theta_{jk}$ by $\theta'_{jk}\frac{du}{dz}$ in the above expressions for $B$ and $C$, we obtain
by the expansion of $\beta(z)$ in \eqref{betaexp} and the identities for the values of $\theta_{jk}$ at $u_+(b)$ in \eqref{thetanoder},
\be\begin{aligned}
B-C&=-\frac{1}{2i}\frac{1}{z-b}\left( \theta_{11}\theta_{22}+\frac{u_{0,b}\beta_{0,b}^2}{2}(\theta_{11}\theta'_{22}+\theta_{22}\theta'_{11})\right)\left|_{u_+(b)}\right.\left(1+\bigO(z-b)\right)\\
&=-\frac{1}{2i}\frac{1}{z-b}\left(1+\bigO(z-b)\right),
\end{aligned}
\ee
where we used the identity \eqref{detatb} in the last equation. Then, by the expansion of $f_b(z)$ in \eqref{glocalexp}, 
\begin{equation}
    -\frac{i}{4}(B-C)s^3f_b(z) \to -s^3 \frac{q(b)^2}{2(a-b)(b-c)}, \qquad z\to b,
\end{equation}
which cancels the last term in (\ref{EBEXPN}). 

We now evaluate $2A+i(B+C)$. We obtain as before, but now using also \eqref{thetader}, that
\be
B+C=-\frac{u_{0,b}}{2i}\frac{1}{(z-b)^{1/2}}\left(
\theta_{11}\theta'_{22}-\theta_{22}\theta'_{11}+\frac{u_{0,b}\beta_{0,b}^2}{2}(\theta_{11}\theta''_{22}-\theta_{22}\theta''_{11})\right)\left|_{u_+(b)}\right.
+\bigO(1),
\ee
and similarly,
\be
A=-\frac{u_{0,b}}{4}\frac{1}{(z-b)^{1/2}}\left(
\theta_{11}\theta'_{22}-\theta_{22}\theta'_{11}+\frac{u_{0,b}\beta_{0,b}^2}{2}(\theta_{11}\theta''_{22}-\theta_{22}\theta''_{11})\right)\left|_{u_+(b)}\right.
+\bigO(1).
\ee
Recalling also \eqref{beta-u}, we obtain
\be
2A+i(B+C)=-\frac{u_{0,b}}{(z-b)^{1/2}}\left(
\theta_{11}\theta'_{22}-\theta_{22}\theta'_{11}+\frac{1}{2J_0}(\theta_{11}\theta''_{22}-\theta_{22}\theta''_{11})\right)\left|_{u_+(b)}
+\bigO(1).\right.
\ee
Substituting this expression into \eqref{EBEXPN} and using again \eqref{glocalexp} and the value of $u_{0,b}$ in \eqref{u0b}, we have
\be\label{t23id}
 \lim_{z\to b}{\tau}_2(z)_{21}=
 s^{3/2}\frac{q(b)}{(a-b)(b-c)}\frac{1}{4J_0^2}\left[\theta_{11}\theta_{22}
 \left(2J_0\left(\frac{\theta'_{11}}{\theta_{11}}-\frac{\theta'_{22}}{\theta_{22}}\right)
 +\frac{\theta''_{11}}{\theta_{11}}-\frac{\theta''_{22}}{\theta_{22}}\right)\right]_{u_+(b)}.
 \ee

Using parts (v) and (vii) of Lemma \ref{thetalemma}, and also expanding $\theta_{jk}$ by their definition \eqref{thetajk},
we obtain
\be\label{222}
 \lim_{z\to b}{\tau}_2(z)_{21}=
 \frac{s^{3/2}}{2}\frac{d\Omega}{db}T_1(s^{3/2}\Omega),
 \ee
 where 
 \be\label{T1omegadef}
 \begin{aligned}
 T_1(\omega)&=\frac{1}{2J_0}\left[\theta_{11}\theta_{22}
 \left(2J_0\left(\frac{\theta'_{11}}{\theta_{11}}-\frac{\theta'_{22}}{\theta_{22}}\right)
 +\frac{\theta''_{11}}{\theta_{11}}-\frac{\theta''_{22}}{\theta_{22}}\right)\right]_{u_+(b)}\\
 &=
 -\frac{\theta_3^2\theta_3(\omega+\widehat{d})\theta_3(\omega-\widehat{d})}{J_0\theta_3^2(\widehat{d})\theta_3^2(\omega)}\left[ \frac{\theta_1^{\prime}}{\theta_1}(\widehat{d})\left(\frac{\theta_3^{\prime}}{\theta_3}(\omega+\widehat{d})+\frac{\theta_3^{\prime}}{\theta_3}(\omega-\widehat{d})\right)-\frac{1}{2}\left(\frac{\theta_3^{\prime\prime}}{\theta_3}(\omega+\widehat{d})- \frac{\theta_3^{\prime\prime}}{\theta_3}(\omega-\widehat{d})\right) \right],
   \end{aligned}
\ee
with $\widehat{d}=u_+(b)+d$.

The function $T_1(\omega)$ has the same structure as (219) in \cite{IB}. 
As in the proof of part (c) of Proposition 23 in \cite{IB}, we verify that (note: $J_0<0$)
\begin{equation}\label{prop23CIB}
    T_1(\omega)=2\frac{\theta_3^{\prime}}{\theta_3}(\omega).
\end{equation}
Substituting this into \eqref{222} and writing out the total derivative, we obtain Lemma \ref{oscillatoryterm}.

\subsection{The constant term. Proof of Lemma \ref{constantterm}.}\label{sec-const}
We are left to consider the term
\begin{equation}
    \mathcal{\tau}_3(z)=-\frac{1}{2\pi i}e^{s^{3/2}g(z)\sigma_3}P_b^{-1}(z)\frac{dR^{(1)}}{dz}(z)P_b(z) e^{-s^{3/2}g(z)\sigma_3}.
\end{equation}
Note that
\begin{equation}\label{R1primeB}
\begin{aligned}
 \frac{d}{dz}  R^{(1)}(z)&=\frac{1}{2\pi i}\int_{\partial U_a \cup \partial U_b \cup \partial U_c}\frac{\Delta_1(\xi)}{(\xi-z)^2}d\xi\\
 &=\frac{1}{2\pi i}\int_{\partial U_a \cup \partial U_b \cup \partial U_c}\frac{\Delta_1(\xi)}{(\xi-b)^2}d\xi\left(1+\bigO(z-b)\right),\qquad
 z\to b.
 \end{aligned}
\end{equation}
Furthermore, with notation as in (\ref{thetaij}) and using the identities \eqref{thetanoder} for $\theta_{jk}$ at $u_+(b)$, 
and the expansions for $\Psi(\xi)$ in \eqref{Jzetaexpn}, $f_b(z)$ in \eqref{glocalexp}, $\beta(z)$ in \eqref{betaexp},
we have for the first column
of $P_b(z)e^{-s^{3/2}g(z)\sigma_3}$,
\begin{equation}\label{PBEXPn}
   \lim_{z\to b} \begin{pmatrix} (P_b(z)e^{-s^{3/2}g(z)\sigma_3})_{11}\\
   (P_b(z)e^{-s^{3/2}g(z)\sigma_3})_{21}\end{pmatrix}
   =\frac{\beta_{0,b}}{\sqrt{2}}\left( s^3 \pi ^2 f_b^{\prime}(b)\right)^{1/4}N_0e^{-s^{3/2}g_+(b)\sigma_3}\begin{pmatrix}
   \theta_{11}(u_+(b)) \\ i\theta_{22}(u_+(b))\end{pmatrix},
\end{equation}
Set
\begin{equation}
    F(\xi)=\begin{pmatrix} \theta_{11}(u(\xi)) N_{11}(\xi) & \theta_{12}(u(\xi)) N_{12}(\xi) \\ \theta_{21}(u(\xi)) N_{21}(\xi) & \theta_{22}(u(\xi)) N_{22}(\xi) \end{pmatrix}.
\end{equation}
The above observations and the fact that $\det P^{\infty}=\det F\equiv 1$ imply
\be
 \lim_{z\to b} \tau_3(z)_{21}=
 \frac{1}{2\pi i}\int_{\partial U_a \cup \partial U_b \cup \partial U_c}\frac{L(\xi)}{(\xi-b)^2}d\xi,
 \ee
 where
 \be\label{defL}
 L(\xi)=\frac{\beta^2_{0,b}}{-4\pi i}\left( s^3 \pi ^2 f_b^{\prime}(b)\right)^{1/2}
\begin{pmatrix}-i\theta_{22}(u_+(b)) & \theta_{11}(u_+(b)) \end{pmatrix}e^{s^{3/2}g_+(b)\sigma_3}\widehat{\Delta_1(\xi)}e^{-s^{3/2}g_+(b)\sigma_3}\begin{pmatrix} \theta_{11}(u_+(b)) \\ i\theta_{22}(u_+(b))\end{pmatrix},
\ee
with, by recalling Lemma \ref{Pprhp},
\begin{align}\label{231}
    \widehat{\Delta_1(\xi)} &=\frac{1}{8s^{3/2}f_a(\xi)^{1/2}}F(\xi) \begin{pmatrix}-1 & -2i \\ -2i & 1 \end{pmatrix} F(\xi)^{-1}, \quad \mathrm{near }\quad \xi=a,\\
    &=\frac{1}{8s^{3/2}f_b(\xi)^{1/2}}F(\xi) e^{-s^{3/2}g_\pm(b)\sigma_3}\begin{pmatrix}-1 & 2i \\ 2i & 1 \end{pmatrix}e^{s^{3/2}g_\pm(b)\sigma_3} F(\xi)^{-1}, \quad \mathrm{near } \quad\xi=b,\label{Deltab}
    \\
    &=\frac{1}{8s^{3/2}f_c(\xi)^{1/2}}F(\xi) e^{-s^{3/2}g_\pm(c)\sigma_3}\begin{pmatrix}-1 & -2i \\ -2i & 1 \end{pmatrix}e^{s^{3/2}g_\pm(c)\sigma_3} F(\xi)^{-1}, \quad \mathrm{near }\quad \xi=c,
    \label{233}
\end{align}
where in $\pm$, $+$ is taken if $\Im\xi>0$, and $-$ if $\Im\xi<0$. (In fact, $g_\pm(b)=g_\pm(c)$.)
Recall that $\Delta_1(\xi)$ (and hence $\widehat{\Delta_1(\xi)}$) is a meromorphic function near each point $a$, $b$, $c$. In fact,
it has a first order pole at these points. 
Note that since there is an additional pole at $b$ introduced by the denominator in (\ref{R1primeB}), we must expand $\widehat{\Delta_1(\xi)}$ up to the third term at $b$, but only up to the first term at $a$ and $c$ to compute the residues. Therefore, it is convenient to define
\be\label{T2T3def}
 T_2(\omega)= \frac{1}{2\pi i}\int_{\partial U_a \cup \partial U_c}\frac{L(\xi)}{(\xi-b)^2}d\xi,\qquad 
T_3(\omega)=  \frac{1}{2\pi i}\int_{\partial U_b }\frac{L(\xi)}{(\xi-b)^2}d\xi,
 \ee
where we denote $\omega=s^{3/2}\Omega$. Then we have
\be\label{tau3TT}
\lim_{z\to b} \tau_3(z)_{21}=T_2(\omega)+T_3(\omega).
\ee

\subsubsection{Evaluation of $T_2(\omega)$}
Since $u(a)=0$, we have expanding $F(\xi)$ (using \eqref{betaexp})
\[
F(\xi)=\begin{pmatrix}\theta_{11}(0) & -i \theta_{12}(0)\\ -i \theta_{21}(0) & \theta_{22}(0) \end{pmatrix}
\frac{1}{2\beta_{0,a}}(\xi-a)^{-1/4}\left(1+o(1)\right),\qquad \xi\to a.
\]
Subsituting this into \eqref{231} and expanding $f_a(\xi)$ by \eqref{glocalexp}, we obtain
\begin{equation}
    \widehat{\Delta_1(\xi)}=\frac{1}{\xi-a}\frac{\beta_{0,a}^{-2}}{16s^{3/2}f_a^{\prime}(a)^{1/2}}\begin{pmatrix}\theta_{11}(0)\theta_{22}(0) & -i \theta_{11}(0)^2 \\ -i \theta_{22}(0)^2 & -\theta_{11}(a)\theta_{22}(0) \end{pmatrix}+\bigO(1), \qquad 
    \xi\to a,
\end{equation}
Similarly, using the identities 
\be
\theta_{j1}(u_+(c))=\theta_{j2}(u_+(c))e^{2s^{3/2}g_+(b)},\qquad j=1,2,
\ee
which are derived as \eqref{thetanoder}, 
we obtain
\begin{equation}\begin{aligned}
    \widehat{\Delta_1(\xi)}&=\frac{1}{\xi-c}\frac{\beta_{0,c}^{-2}}{16s^{3/2}f_c^{\prime}(c)^{1/2}}e^{-s^{3/2}g_+(b)\sigma_3} \begin{pmatrix}\theta_{11}(u_+(c))\theta_{22}(u_+(c)) & -i \theta_{11}(u_+(c))^2 \\ -i \theta_{22}(u_+(c))^2 & -\theta_{11}(u_+(c))\theta_{22}(u_+(c)) \end{pmatrix}e^{s^{3/2}
    g_+(b)\sigma_3}\\
    &+\bigO(1), \qquad 
    \xi\to c.
    \end{aligned}
\end{equation}
Therefore by \eqref{defL}, writing out $\theta_{jk}$ by \eqref{thetajk}, using relations in (\ref{theta2134relations}) and the following summation formulae,
\begin{align}\label{theta23}
\theta_2(x+y)\theta_3(x-y)+\theta_2(x-y)\theta_3(x+y)&=\frac{2}{\theta_2\theta_3}\theta_2(x)\theta_2(y)\theta_3(x)\theta_3(y),\\
\label{theta34}
\theta_1(x+y)\theta_2(x-y)-\theta_1(x-y)\theta_2(x+y)&=\frac{2}{\theta_3\theta_4}\theta_1(y)\theta_2(y)\theta_3(x)\theta_4(x),
\end{align}
(where we recall the convention that the $\theta$-functions without argument stand for their values at zero),
we obtain
\begin{equation}
L(\xi)=\frac{1}{\xi-a}\frac{1 }{16}\frac{\beta_{0,b}^2}{\beta_{0,a}^2}\frac{f_b^{\prime}(b)^{1/2}}{f_a^{\prime}(a)^{1/2}}\frac{\theta_3^2}{\theta_2^2}\frac{\theta_2^2(s^{3/2}\Omega)}{\theta_3^2(s^{3/2}\Omega)},\qquad \xi\to a,
\end{equation}
and
\begin{equation}
L(\xi)=\frac{1}{\xi-c}\frac{1 }{16}\frac{\beta_{0,b}^2}{\beta_{0,c}^2}\frac{f_b^{\prime}(b)^{1/2}}{f_c^{\prime}(c)^{1/2}}\frac{\theta_3^2}{\theta_4^2}\frac{\theta_4^2(s^{3/2}\Omega)}{\theta_3^2(s^{3/2}\Omega)},\qquad \xi\to c.
\end{equation}
Now note that an analysis of branches in \eqref{glocalexp} gives that $f_a^{\prime}(a)$, $f_c^{\prime}(c)>0$, while $f_b^{\prime}(b)=-i\left|f_b^{\prime}(b)\right|$. Using their values and those of $\beta_{0,p}$ in \eqref{b0}, we write the residues and determine $T_2$ (note the negative direction of integration in \eqref{T2T3def}):
\begin{equation}\label{firstACresidues}
   T_2(\omega)= -\frac{1}{16}\left( \frac{1}{a-b}\left| \frac{q(b)}{q(a)}\right|\frac{\theta_3^2}{\theta_2^2}\frac{\theta_2^2(\omega)}{\theta_3^2(\omega)} + \frac{1}{b-c}\left| \frac{q(b)}{q(c)}\right| \frac{\theta_3^2}{\theta_4^2}\frac{\theta_4^2(\omega)}{\theta_3^2(\omega)}\right).
\end{equation}

We now average $T_2(\omega)$ over $\omega$. We use the integrals (238) obtained in \cite{IB}. Namely,
\begin{equation}\label{238}
    \int_0^1 \frac{\theta_2^2(\omega)}{\theta_3^2(\omega)}d\omega =\frac{\theta_4^2}{(\theta_1^{\prime})^2}\frac{\theta_3^{\prime \prime}}{\theta_3}+\frac{\theta_2^2}{\theta_3^2} ,\qquad \int_0^1 \frac{\theta_4^2(\omega)}{\theta_3^2(\omega)}d\omega =-\frac{\theta_2^2}{(\theta_1^{\prime})^2}\frac{\theta_3^{\prime \prime}}{\theta_3}+\frac{\theta_4^2}{\theta_3^2}.
\end{equation}
By the identity $\theta_1^{\prime}=\pi\theta_2\theta_3\theta_4$ \cite{WW}, we can write the first integral as follows
\be\label{intinterim}
\frac{\theta_3^2}{\theta_2^2}
\int_0^1 \frac{\theta_2^2(\omega)}{\theta_3^2(\omega)}d\omega =\frac{1}{\pi^2\theta_2^4}\frac{\theta_3^{\prime \prime}}{\theta_3}+1.
\ee
To simplify this  further, we use the heat equation for $\theta$-functions,  
\begin{equation}
    \frac{\theta_3^{\prime\prime}}{\theta_3}=4\pi i \frac{\partial}{\partial \tau}\log \theta_3=\pi i \frac{1}{\frac{\partial\tau}{\partial b}}\frac{\partial}{\partial b}\log \theta_3^4=-2J_0(a-b)(b-c)\frac{d}{db}J_0=J_0(J_1-bJ_0),
\end{equation}
where we made use of parts $(i)$ and $(viii)$ of Lemma \ref{thetalemma} to express $\theta_3^4$ and $\frac{d\tau}{db}$ in terms
of elliptic integrals, and used \eqref{dbJ0} to evaluate $\frac{dJ_0}{db}$.
Substituting this expression as well as the one for $\theta_2^4$ from Lemma \ref{thetalemma} into the r.h.s. of \eqref{intinterim}, we obtain
\be
\frac{\theta_3^2}{\theta_2^2}
\int_0^1 \frac{\theta_2^2(\omega)}{\theta_3^2(\omega)}d\omega =
\frac{J_1-cJ_0}{J_0(b-c)}.
\ee
A similar analysis of the second integral in \eqref{238} gives
\be
\frac{\theta_3^2}{\theta_4^2}
\int_0^1 \frac{\theta_4^2(\omega)}{\theta_3^2(\omega)}d\omega =
-\frac{J_1-aJ_0}{J_0(a-b)}.
\ee
Therefore, integrating \eqref{firstACresidues} we obtain
\begin{equation}\label{secondACresidues}
 \int_0^1 T_2(\omega)d\omega=   -\frac{1}{16(a-b)(b-c)}\left(\left|\frac{q(b)}{q(a)}\right| \left( \frac{J_1}{J_0}-c\right) -\left|\frac{q(b)}{q(c)}\right| \left( \frac{J_1}{J_0}-a\right) \right).
\end{equation}
Next, by means of direct differentiation of $q(a)$ and $q(c)$, recalling the derivative $\frac{dq_0}{db}$ in (\ref{ddbq0}), we have that
\begin{equation}
    \frac{\frac{J_1}{J_0}-c}{2(a-b)(b-c)}=\frac{\frac{d}{db}q(a)}{q(b)}+\frac{a+b-c}{4q(b)}+\frac{1}{2(a-b)},
\end{equation}
and
\begin{equation}
    \frac{\frac{J_1}{J_0}-a}{2(a-b)(b-c)}=\frac{\frac{d}{db}q(c)}{q(b)}+\frac{b+c-a}{4q(b)}-\frac{1}{2(b-c)}.
\end{equation}
Substituting these expressions into (\ref{secondACresidues}), noting the signs of $q(p)$ obtained earlier in (\ref{signsqp}), 
and using the simple identities
\be
q(b)=q(a)-\frac{1}{2}(a-b)(a+b-c),\qquad q(b)=q(c)+\frac{1}{2}(b-c)(b+c-a),
\ee
we find
\begin{equation}\label{ACRresidues}
 \int_0^1 T_2(\omega) d\omega =  -\frac{1}{8}\frac{d}{db}\log \frac{\lvert q(a)q(c)\rvert}{\sqrt{(a-b)(b-c)}}.
\end{equation}

\subsubsection{Evaluation of $T_3(\omega)$}
In view of the definition \eqref{T2T3def}, we will now simplify the expression for $L(\xi)$ in \eqref{defL}, with $\widehat{\Delta_1(\xi)}$
given by \eqref{Deltab}. We use the following argument from \cite{IB} to decompose matrix products. In a neighborhood of $b$, define functions $A_j(z)$ and $B_j(z)$,  $j=1,2$, by
\be\label{AABB}
\begin{aligned}
    A_1(z)&=\frac{1}{2}\left(\theta_{11}(u)(\beta(z)+\beta^{-1}(z))e^{\mp s^{3/2}g_+(b)} +\theta_{12}(u)(\beta(z)-\beta^{-1}(z))e^{\pm s^{3/2}g_+(b)} \right),\\
    A_2(z)&=\frac{1}{2}\left(\theta_{22}(u)(\beta(z)+\beta^{-1}(z))e^{\pm s^{3/2}g_+(b)} +\theta_{21}(u)(\beta(z)-\beta^{-1}(z))e^{\mp s^{3/2}g_+(b)} \right),\\
    B_1(z)&=\frac{1}{2}\left(\theta_{11}(u)(\beta(z)+\beta^{-1}(z))e^{\mp s^{3/2}g_+(b)} -\theta_{12}(u)(\beta(z)-\beta^{-1}(z))e^{\pm s^{3/2}g_+(b)} \right),\\
    B_2(z)&=\frac{1}{2}\left(\theta_{22}(u)(\beta(z)+\beta^{-1}(z))e^{\pm s^{3/2}g_+(b)} -\theta_{21}(u)(\beta(z)-\beta^{-1}(z))e^{\mp s^{3/2}g_+(b)} \right),
    \end{aligned}\ee
where the upper sign in $\pm$, $\mp$ is taken for $\Im(z-b)>0$, and the lower, for $\Im(z-b)<0$.
It follows from the jump relations (cf. proof of  Proposition \ref{pinfinityprop}), that 
\begin{equation}\label{analytAB}
    A_j(z)(z-b)^{1/4},
     \qquad   B_j(z)(z-b)^{-1/4},\qquad  j=1,2,
\end{equation}
are analytic functions in $U_b$. 
In this notation, we have
\[
\widehat{\Delta_1(z)}=\frac{1}{8s^{3/2}f_b(z)^{1/2}}F(z) e^{-s^{3/2}g_\pm(b)\sigma_3}\begin{pmatrix}-1 & 2i \\ 2i & 1 \end{pmatrix}e^{s^{3/2}g_\pm(b)\sigma_3} F(z)^{-1}, 
\]
with 
\be
  F e^{-s^{3/2}g_{\pm}(b)\sigma_3}=\frac{1}{2}\begin{pmatrix}A_1+B_1 & -i(A_1-B_1) \\ i(A_2-B_2) & A_2+B_2 \end{pmatrix}.
\ee
This is a convenient way of expressing $\widehat{\Delta_1(z)}$ near $b$ for the following reason.
Since $\widehat{\Delta_1(z)}$ is meromorphic near $b$, its expansion at $z=b$ contains only integer powers. On account of the $1/f_b(z)^{1/2}$ term, containing only half-powers, it follows from the analyticity of \eqref{analytAB}
that all 'cross-terms', containing products $A_j B_k$, must be zero. Thus we have the following decomposition:
\be\begin{aligned}
    \widehat{\Delta_1(z)} &=\frac{1}{4\cdot 8s^{3/2}f_b(z)^{1/2}}\begin{pmatrix}A_1 & -iA_1 \\ iA_2 & A_2 \end{pmatrix} \begin{pmatrix}-1 & 2i \\ 2i & 1 \end{pmatrix}\begin{pmatrix}A_2 & iA_1 \\ -iA_2 & A_1 \end{pmatrix}+\\
    & +\frac{1}{4\cdot 8s^{3/2}f_b(z)^{1/2}}\begin{pmatrix}B_1 & iB_1 \\ -iB_2 & B_2 \end{pmatrix} \begin{pmatrix}-1 & 2i \\ 2i & 1 \end{pmatrix}\begin{pmatrix}B_2 & -iB_1 \\ iB_2 & B_1 \end{pmatrix}.
    \end{aligned}\ee
It is now a straightforward calculation to see that we may write (\ref{defL}) near $b$ in the form
\begin{equation}\label{AANDBFORM}
\begin{aligned}
L(z)&=
-\frac{\beta^2_{0,b}  f_b^{\prime}(b)^{1/2}}{64 f_b(z)^{1/2}}
   \left( [A_1(z) \theta_{22}(u_+(b))e^{s^{3/2}g_+(b)}-A_2(z)\theta_{11}(u_+(b))e^{-s^{3/2}g_+(b)}]^2\right.\\
   &\left.+3[B_1(z)\theta_{22}(u_+(b))e^{s^{3/2}g_+(b)}+B_2(z)\theta_{11}(u_+(b))e^{-s^{3/2}g_+(b)}]^2\right),
   \end{aligned}
\end{equation}

By means of (\ref{uexp}), (\ref{betaexp}), and the identities for $\theta_{jk}$ in \eqref{thetanoder}, \eqref{thdeder}, we expand $\beta^{-1}(z)A_j(z)$, $z\to b$, for $j=1,2$, and determine that the analytic (by \eqref{analytAB}) in $U_b$ function
\be\begin{aligned}
    &\beta^{-1}(z)[A_1(z) \theta_{22}(u_+(b))e^{s^{3/2}g_+(b)}-A_2(z)\theta_{11}(u_+(b))e^{-s^{3/2}g_+(b)}]\\
    &=\left\{\beta_{0,b}^{-2}u_{0,b}(\theta_{11}^{\prime}\theta_{22}-\theta_{11}\theta_{22}^{\prime})+ 
    +\frac{u_{0,b}^2}{2}(\theta_{11}^{\prime\prime}\theta_{22}-\theta_{11}\theta_{22}^{\prime\prime})\right\}(z-b)+\bigO((z-b)^2)\\
    &=\frac{u_{0,b}^2}{2}\left\{2J_0(\theta_{11}^{\prime}\theta_{22}-\theta_{11}\theta_{22}^{\prime})+
    \theta_{11}^{\prime\prime}\theta_{22}-\theta_{11}\theta_{22}^{\prime\prime}\right\}(z-b)+\bigO((z-b)^2)
    ,\qquad z\to b,
\end{aligned}\ee
where all $\theta_{jk}$ and their derivatives with omitted argument are evaluated at $u_+(b)$, and we used the identity \eqref{beta-u}
in the last equation.
Notice that the $z-b$ term is, up to a prefactor, is what we found earlier in (\ref{t23id}). 
By the equations \eqref{T1omegadef} and \eqref{prop23CIB},
\begin{equation}\label{T3useful}
\begin{aligned}
 &\beta^{-1}(z)[A_1(z) \theta_{22}(u_+(b))e^{s^{3/2}g_+(b)}-A_2(z)\theta_{11}(u_+(b))e^{-s^{3/2}g_+(b)}]\\
 &=2u_{0,b}^2 J_0 \frac{\theta_3^{\prime}}{\theta_3}(\omega) (z-b)+\bigO((z-b)^2),\qquad z\to b,\qquad \omega=s^{3/2}\Omega.
 \end{aligned}
\end{equation}
 Substituting the above expression into the first term in (\ref{AANDBFORM}) and that in turn into \eqref{T2T3def}, we find that
\be
\begin{aligned}
T_3(\omega)&=  \frac{1}{2\pi i}\int_{\partial U_b }\frac{L(z)}{(z-b)^2}dz=
-\frac{\beta^2_{0,b}  f_b^{\prime}(b)^{1/2}}{64}\left(\int_{\partial U_b }  \frac{dz}{2\pi i}
\frac{\beta^{2}(z)}{f_b(z)^{1/2}}\left\{\left(2u_{0,b}^2 J_0 \frac{\theta_3^{\prime}}{\theta_3}(\omega)\right)^2+\bigO(z-b)\right\}
\right.
\\
 &\left.+\int_{\partial U_b }  \frac{dz}{2\pi i}
 \frac{3}{(z-b)^2f_b(z)^{1/2}}[B_1(z)\theta_{22}(u_+(b))e^{s^{3/2}g_+(b)}+B_2(z)\theta_{11}(u_+(b))e^{-s^{3/2}g_+(b)}]^2\right).
\end{aligned}
 \ee
 Since the meromorphic in $U_b$ function
 \[
 \frac{\beta^2(z)}{f_b(z)^{1/2}}=\frac{1}{z-b}\frac{\beta^2_{0,b} }{f_b^{\prime}(b)^{1/2}}\left(1+\bigO(z-b)\right),\qquad z\to b,
 \]
we immediately compute the first integral by taking the residue (recall the negative direction of the integration), and recalling the expressions \eqref{b0}, \eqref{u0b} for $\beta_{0,b}$, $u_{0,b}$, we obtain
\be\label{penultimateAterm}
\begin{aligned}
T_3(\omega)&=- \frac{1}{16 J_0^2(a-b)(b-c)}\left(\frac{\theta_3^{\prime}}{\theta_3}(\omega)\right)^2\\
\\
 &
 -\frac{\beta^2_{0,b}  f_b^{\prime}(b)^{1/2}}{64}\int_{\partial U_b }  \frac{dz}{2\pi i}
 \frac{3}{(z-b)^2f_b(z)^{1/2}}[B_1(z)\theta_{22}(u_+(b))e^{s^{3/2}g_+(b)}+B_2(z)\theta_{11}(u_+(b))e^{-s^{3/2}g_+(b)}]^2.
\end{aligned}
 \ee
Recall that we are interested in the average of $T_3(\omega)$ over $\omega$. As was noticed in a similar situation in \cite{IB},
it is easier to do the averaging of the second term in \eqref{penultimateAterm} {\it before} computing its residue. Furthermore,
to average the first term, we can use the integral (A.19) in Lemma 26 from \cite{IB}:
\be\label{onethetaint}
 \int_0^1 \left( \frac{\theta_3^{\prime}}{\theta_3}(\omega) \right)^2d\omega =\frac{\pi^2}{3}+\frac{\theta_1^{\prime\prime\prime}}{3\theta_1^{\prime}}.
 \ee
 By the heat equation, the identity $\theta_1^{\prime}=\pi \theta_2 \theta_3\theta_4$, and the identities of  Lemma \ref{thetalemma}, we have that
\be
    \frac{\theta_1^{\prime \prime \prime}}{\theta_1^{\prime}}=4\pi i \frac{\partial }{\partial \tau}\log \theta_1^{\prime}=\frac{\pi i}{\frac{d\tau}{d b}}\frac{d}{d b}\log( \theta_2^4 \theta_3^4 \theta_4^4)
    =-J_0^2 (a-b)(b-c)\frac{d}{d b} \log \left[\left(J_0^6(a-b)(b-c)\right) \right].
\ee
 
 Thus, we obtain
\be\label{264}
\frac{1}{J_0^2(a-b)(b-c)} \int_0^1 \left( \frac{\theta_3^{\prime}}{\theta_3}(\omega) \right)^2d\omega=
\frac{\pi^2/3}{J_0^2(a-b)(b-c)}-\frac{1}{3}\frac{d}{d b} \log \left[\left(J_0^6(a-b)(b-c)\right)\right].
\ee

On the other hand, since
\be
 \frac{\partial }{\partial \tau}\log \theta_3(\omega)=\frac{1}{4\pi i} \frac{\theta_3^{\prime \prime}(\omega)}{\theta_3(\omega)}
 \ee
and
\[
0= \int_0^1 \left( \frac{\theta_3^{\prime}}{\theta_3}(\omega) \right)^{\prime}d\omega=
\int_0^1\left[\frac{\theta_3^{\prime \prime}(\omega)}{\theta_3(\omega)}-\left( \frac{\theta_3^{\prime}}{\theta_3}(\omega) \right)^2\right]
d\omega,
\]
we obtain, using the expression for $d\tau/db$ from Lemma \ref{thetalemma},
\be\label{addition}
\frac{d\tau}{d b} \int_0^1 \frac{\partial }{\partial \tau}\log \theta_3(\omega)d\omega=
-\frac{\pi i}{J_0^2(a-b)(b-c)}\int_0^1 \frac{\partial }{\partial \tau}\log \theta_3(\omega)d\omega=
-\frac{1}{4J_0^2(a-b)(b-c)}\int_0^1   \left( \frac{\theta_3^{\prime}}{\theta_3}(\omega) \right)^2d\omega.
\ee

Thus, integrating \eqref{penultimateAterm} and using \eqref{addition} and \eqref{264}, we obtain
 \be\label{T3prelim}
 \begin{aligned}
&\int_0^1 T_3(\omega)d\omega- \frac{d\tau}{d b} \int_0^1 \frac{\partial }{\partial \tau}\log \theta_3(\omega)d\omega=
 \frac{\pi^2/16}{J_0^2(a-b)(b-c)}-\frac{1}{16}\frac{d}{d b} \log \left[\left(J_0^6(a-b)(b-c)\right)\right]\\
  &
 -\frac{\beta^2_{0,b}  f_b^{\prime}(b)^{1/2}}{64}\int_{\partial U_b }  \frac{dz}{2\pi i}
 \frac{3}{(z-b)^2f_b(z)^{1/2}}\int_0^1 d\omega [B_1(z)\theta_{22}(u_+(b))e^{s^{3/2}g_+(b)}+B_2(z)\theta_{11}(u_+(b))e^{-s^{3/2}g_+(b)}]^2.
\end{aligned}
 \ee

We now turn to computation of the average of the term with $B$'s in this expression.
First note that the analytic in $U_b$ function
\begin{equation}\label{fb}
    \frac{1}{\beta^2(z)f_b(z)^{1/2}}=\frac{1}{\beta_{0,b}^2f_b^{\prime}(b)^{1/2}}\left(1-\left(\frac{f_{1,b}}{2}+2\beta_{1,b}\right)(z-b)+\bigO((z-b)^2) \right), \qquad z\to b,
\end{equation}
where we used (\ref{betaexp}) and wrote $f_b(z)=f_b^\prime(b)(z-b)(1+f_{1,b}(z-b)+\bigO((z-b)^2))$ recalling \eqref{glocalexp}. 
In the last expansion
\be\label{f1b}
f_{1,b}=\frac{2}{3}\left(\frac{q_1+2b}{q(b)}-\frac{1}{2}\left(\frac{1}{b-c}-\frac{1}{a-b}\right)\right).
\ee

Expanding $B_j(z)$ similarly to $A_j(z)$ above, and using the identity \eqref{detatb2}, we obtain for the analytic  (by \eqref{analytAB}) in $U_b$ function
\begin{equation}
    \beta(z)[B_1(z)\theta_{22}(u_+(b))e^{s^{3/2}g_+(b)}+B_2(z)\theta_{11}(u_+(b))e^{-s^{3/2}g_+(b)}]-2=\bigO(z-b), \qquad z\to b.
\end{equation}
Taking the square of this expression, we may write
\begin{equation}\label{Btermexpn}
\begin{aligned}
 &\beta(z)^2[B_1(z)\theta_{22}(u_+(b))e^{s^{3/2}g_+(b)}+B_2(z)\theta_{11}(u_+(b))e^{-s^{3/2}g_+(b)}]^2\\
  &=-4+4\beta(z)
  [B_1(z)\theta_{22}(u_+(b))e^{s^{3/2}g_+(b)}+B_2(z)\theta_{11}(u_+(b))e^{-s^{3/2}g_+(b)}]+\bigO((z-b)^2), \qquad z\to b.  
  \end{aligned}
\end{equation}
The advantage of this representation is that we see (taking into account \eqref{fb}) that only the first 2 terms on the r.h.s. may give a nonzero contribution to
the residue in \eqref{T3prelim}, and the second term is not squared.

We now evaluate the average
\begin{equation}
    \int_0^1 [B_1(z)\theta_{22}(u_+(b))e^{s^{3/2}g_+(b)}+B_2(z)\theta_{11}(u_+(b))e^{-s^{3/2}g_+(b)}]
     d\omega=\int_0^1 [\widehat{q(\omega)}+\widehat{q(-\omega)}]d\omega, \qquad \omega=s^{3/2}\Omega,
\end{equation}
where, by expanding the terms and using quasiperiodicity of $\theta$-functions,
\begin{equation}
\widehat{q(\omega)}=
        \frac{\theta_3^2}{2\theta_3^2(\omega)}\frac{\theta_3(-\omega+\widehat{d})}{\theta_3(\widehat{d})}\left( \frac{\theta_3(\widehat{u(z)}+\omega+\widehat{d})}{\theta_3(\widehat{u(z)}+\widehat{d})}(\beta(z)+\beta^{-1}(z))-\frac{\theta_3(\widehat{u(z)}-\omega-\widehat{d})}{\theta_3(\widehat{u(z)}-\widehat{d})}(\beta(z)-\beta^{-1}(z))\right),
\end{equation}
with
\begin{equation}
    \widehat{u(z)}=u(z)-u_+(b), \quad \widehat{d}=d+u_+(b),\qquad u_+(b)=-\frac{\tau}{2}.
\end{equation}
Since $\widehat{q(-\omega)}=\widehat{q(1-\omega)}$, we have that $\int_0^1 [\widehat{q(\omega)}+\widehat{q(-\omega)}]d\omega=
2\int_0^1 \widehat{q(\omega)}d\omega$.
We notice that our expression for $\widehat{q(\omega)}$ has the form of $\widetilde{q(\omega)}$, given by (264) in Section 9.2 of \cite{IB}, with $u(z) \rightarrow \widehat{u(z)}$ and $d \rightarrow \widehat{d}$. Therefore, our analysis now is very similar to that of  \cite{IB}.

Applying Lemma 26 of \cite{IB} to evaluate $\int_0^1\widehat{q(\omega)}d\omega$, we obtain:
\begin{multline}\label{formulaT38}
\beta(z) \int_0^1 [B_1(z)\theta_{22}(u_+(b))e^{s^{3/2}g_+(b)}+B_2(z)\theta_{11}(u_+(b))e^{-s^{3/2}g_+(b)}]
     d\omega=
\frac{\pi \theta_3^2 \widehat g(\widehat d)}{\left(\theta_1'\right)^2\sin(\pi \widehat u)}\\
\times\left\{ \left(\beta(z)^2+1\right) 
\widehat g(\widehat d+\widehat u)[\widehat f(\widehat d)-\widehat f(\widehat d+\widehat u)]
+\left(\beta(z)^2-1\right)\widehat g(\widehat d-\widehat u)[\widehat f(\widehat d)-\widehat f(\widehat d-\widehat u)]
\right\},
\end{multline}
where
\begin{equation}\label{formulaT39}
\widehat g(x)=\frac{\theta_1(x)}{\theta_3(x)},\qquad
\widehat f(x)=\frac{\theta_1'(x)}{\theta_1(x)}.
\end{equation}
Furthermore, expanding $\widehat{u(z)}$ and using (270), (271) in \cite{IB} gives, as $z\to b$,
\begin{multline}\label{T39c}
\beta(z) \int_0^1 [B_1(z)\theta_{22}(u_+(b))e^{s^{3/2}g_+(b)}+B_2(z)\theta_{11}(u_+(b))e^{-s^{3/2}g_+(b)}]
     d\omega\\
     =\widehat g(\widehat d)\left(1+\frac{\pi^2}{6}u_{0,b}^2(z-b)+ \mathcal O((z-b)^2)\right)\\ 
     \times\left[H_0-u_{0,b}\beta_{0,b}^2(1+(z-b)(u_{1,b}+2\beta_{1,b}))H_1+(z-b)(u_{0,b}^2H_2+u_{0,b}^3\beta_{0,b}^2H_3)+\mathcal O((z-b)^{3/2})\right],
\end{multline}
 where
\begin{equation}\begin{aligned}
H_0&=2\widehat g(\widehat d)\left(\frac{1}{\widehat g(\widehat d)^2}-\frac{\theta_3''}{\theta_3}\left(\frac{\theta_3}{\theta_1'}\right)^2\right),&
H_1&=2\widehat g'(\widehat d)\frac{\theta_3''}{\theta_3}\left(\frac{\theta_3}{\theta_1'}\right)^2,\\
H_2&=\widehat g''(\widehat d)\left(\frac{1}{3\widehat g(\widehat d)^2}-\frac{\theta_3''}{\theta_3}\left(\frac{\theta_3}{\theta_1'}\right)^2\right),&
H_3&=\frac{\widehat g'''(\widehat d)}{6}\left(\frac{1}{\widehat g(\widehat d)^2}-2\frac{\theta_3''}{\theta_3}\left(\frac{\theta_3}{\theta_1'}\right)^2\right)-\frac{1}{6}\frac{\widehat g''(\widehat d)\widehat g'(\widehat d)}{\widehat g(\widehat d)^3}.
\end{aligned}\end{equation}
Recalling \eqref{id5} of Lemma \ref{thetalemma} and \eqref{beta-u}, we note that
\be
\widehat g'(\widehat d)=-J_0 \widehat g(\widehat d)=-\frac{1}{u_{0,b}\beta^2_{0,b}}
\widehat g(\widehat d).
\ee

By further applying 
\eqref{id5} and \eqref{id6} of Lemma \ref{thetalemma}, we simplify the combinations of the $H_j$ as follows:
\begin{equation} 
H_0-u_{0,b}\beta_{0,b}^{2}H_1=\frac{2}{\widehat g(\widehat d)},\qquad u_{0,b}^2H_2+u_{0,b}^3\beta_{0,b}^2H_3=\frac{2\beta_{1,b}+u_{1,b}}{\widehat g(\widehat d)}\left(1-2\widehat g(\widehat d)^2\frac{\theta_3''}{\theta_3}\left(\frac{\theta_3}{\theta_1'}\right)^2\right),
\end{equation}
which allows us to write \eqref{T39c} in the form
\begin{equation}\label{Bavgexpn}
\begin{aligned}
\beta(z) \int_0^1 [B_1(z)\theta_{22}(u_+(b))e^{s^{3/2}g_+(b)}+B_2(z)\theta_{11}(u_+(b))e^{-s^{3/2}g_+(b)}]
     d\omega
=\\
2+\left(\frac{\pi^2}{3}u_{0,b}^2+u_{1,b}+2\beta_{1,b} \right)(z-b)+\bigO((z-b)^2), \qquad z\to b,
\end{aligned}
\end{equation}

Substituting this into \eqref{T3prelim} and calculating the residue by
(\ref{Btermexpn}) and (\ref{fb}) we obtain
\be\label{T3almost}
\begin{aligned}
-\frac{\beta^2_{0,b}  f_b^{\prime}(b)^{1/2}}{64}\int_{\partial U_b }  \frac{dz}{2\pi i}
 \frac{3}{(z-b)^2f_b(z)^{1/2}}\int_0^1 d\omega [B_1(z)\theta_{22}(u_+(b))e^{s^{3/2}g_+(b)}+B_2(z)\theta_{11}(u_+(b))e^{-s^{3/2}g_+(b)}]^2\\
         =\frac{1}{16}\left(\pi^2u_{0,b}^2+3u_{1,b}-3f_{1,b}/2\right)=
         -\frac{1}{16}\left(\frac{\pi^2}{(a-b)(b-c)J_0^2}+\frac{q_1+2b}{q(b)} \right).
\end{aligned}\ee
For the last equation here we used \eqref{u0b}, \eqref{f1b}.
On the other hand, 
differentiating  $q(b)$ w.r.t. $b$ and using (\ref{ddbq0}),  we obtain
\begin{equation}
  \frac{q_1+2b}{2q(b)}= \frac{d}{db}\log \lvert q(b)\rvert+\frac{d}{db}\log 
  \vert J_0\rvert.
\end{equation}
With this we may write the r.h.s. of \eqref{T3almost} as
\begin{equation}\label{Btermfinalcontribution}
   - \frac{\pi^2}{16(a-b)(b-c)J_0^2}-\frac{1}{8}\frac{d}{db}\log \lvert q(b)J_0 \rvert.
\end{equation}
Substituting this expression into (\ref{T3prelim}) we obtain
\begin{equation}\label{BRESIDUE}
  \int_0^1 T_3(\omega)d\omega-\frac{d\tau}{db}\int_0^1\frac{\partial}{\partial \tau}\log \theta_3(\omega)d\omega
  = -\frac{1}{2}\frac{d}{db}\log\lvert J_0\rvert-\frac{1}{8}\frac{d}{db}\log\lvert q(b) \rvert -\frac{1}{16}\frac{d}{db}\log((a-b)(b-c)).
\end{equation}
Together with (\ref{ACRresidues}) this equation gives
\begin{align*}
  \lim_{z\to b}\int_0^1\mathcal{\tau}_3(z)_{21}&d\omega-\frac{d\tau}{db}\int_0^1 \frac{\partial }{\partial \tau}\log \theta_3(\omega)d\omega
  = \int_0^1 T_2(\omega)d\omega+ \int_0^1 T_3(\omega)d\omega-\frac{d\tau}{db}\int_0^1 \frac{\partial }{\partial \tau}\log \theta_3(\omega)d\omega
   \\
  &=-\frac{1}{2}\frac{d}{db}\log \lvert J_0 \rvert-\frac{1}{8}\frac{d}{db}\log \lvert q(a)q(b)q(c)\rvert,
\end{align*}
which proves Lemma \ref{constantterm}.

\subsection{Asymptotic form of the differential identity at $a$. Proof of Proposition \ref{DEataasympt}.}\label{secDEATA}
In this section we prove that the differential identity is symmetric in $a$ and $b$ as shown in Proposition \ref{DEataasympt}. We start from the differential identity at $a$, namely
\begin{equation}
    \frac{d}{da}\log P^{\textmd{Ai}}(sJ)=\frac{1}{2\pi i}\lim_{z\to a}\left( X^{-1}(z)X^{\prime}(z)\right)_{21},
\end{equation}
where the limit is taken from outside the lens and with $\Im z>0$.
By the definitions of $S(z)$ and $R(z;s)$, we have that
\begin{equation}
X^{-1}X^{\prime}=  e^{s^{3/2}g(z)\sigma_3}P_a^{-1}R^{-1}R^{\prime}P_a e^{-s^{3/2}g(z)\sigma_3}+e^{s^{3/2}g(z)\sigma_3}P_a^{-1}  P_a^{\prime}e^{-s^{3/2}g(z)\sigma_3}-s^{3/2}g^{\prime}(z)\sigma_3.
\end{equation}
Then, by definition of $P_a(z;s)$, noting that $g(a)=0$, we have
\[ P_a=E_a \Psi(s^3g(z)^2) e^{s^{3/2}g(z)\sigma_3},\]       
\[P_a^{-1}P_a^{\prime}=e^{-s^{3/2}g(z)\sigma_3}\Psi^{-1}E_a^{-1}E_a^{\prime} \Psi e^{s^{3/2}g(z)\sigma_3}+e^{-s^{3/2}g(z)\sigma_3 }\Psi^{-1}\Psi^{\prime} e^{s^{3/2}g(z)\sigma_3} +s^{3/2}g^{\prime}(z)\sigma_3.  \]
Moreover, we have
\begin{equation}\label{RRexta}
    R^{-1}(z)R^{\prime}(z)=\frac{dR^{(1)}}{dz}(z)+
    \bigO\left(\frac{1}{s^3\min\{|a|^2,(b-c)^2\}}\right), \qquad s\to \infty,
\end{equation}
where
\begin{equation}\label{R1Aprime}
    R^{(1)\prime}(z;s)=\frac{1}{2\pi i}\int_{\partial U_a \cup \partial U_b \cup \partial U_c}\frac{\Delta_1(\zeta;s)}{(\zeta-z)^2}d\zeta.
\end{equation}
It follows, similarly to the argument leading to \eqref{prelimitingformula}, that in the regime of Proposition \ref{DEataasympt}
\begin{equation}\label{prelim2}
       \frac{d}{da}\log P^{\mathrm{Ai}}(sJ)=  \lim_{z\to a} \left(\mathcal{\wt\tau}_1(z)+\mathcal{\wt\tau}_2(z)+\mathcal{\wt\tau}_3(z)\right)_{21}+ \bigO\left(\frac{1}{s^{3/2}\lvert a \rvert}\right)+\bigO\left(\frac{1}{s^{3/2}(b-c)^2}\right)
\end{equation}
with
\begin{align}
    & \mathcal{\wt\tau}_1(z)=\frac{1}{2\pi i}\Psi^{-1}(s^3 f_a(z))\Psi^{\prime}(s^3 f_a(z)), \label{tauA121} \\
    & \mathcal{\wt\tau}_2(z)=\frac{1}{2\pi i}\Psi^{-1}(s^3 f_a(z))E_a^{-1}(z)E_a^{\prime}(z) \Psi(s^3 f_a(z)),\label{tauA221} \\
    & \mathcal{\wt\tau}_3(z)=\frac{1}{2\pi i}e^{s^{3/2}g(z)\sigma_3}P_a^{-1}(z)\frac{R^{(1)}}{dz}(z)P_a(z) e^{-s^{3/2}g(z)\sigma_3}.\label{tauA321}
\end{align}

We now compute the limit in \eqref{prelim2}. As before, we summarise the contributions to the differential identity, from each term, in the following lemmata.
\begin{lemma}\label{maintermA}
\begin{equation}
   \lim_{z\to a} \mathcal{\wt\tau}_1(z)_{21}=- s^3 \frac{d}{da}\alpha_2.
\end{equation}
\end{lemma}
\begin{proof}
The proof of Lemma \ref{maintermA} runs almost exactly the same as that of Lemma \ref{mainterm}:\\
From \eqref{Jzetaexpn} we easily obtain, 
recalling that $\zeta=\zeta(z)=s^3 f_a(z)=s^3 g(z)^2$ and using \eqref{glocalexp},
\begin{equation}\label{mainprelimA}
\begin{aligned}
\lim_{z\to a}{\mathcal{\wt\tau}}_1(z)_{21}&=
\frac{1}{2\pi i}\lim_{z\to a}\left(\Psi^{-1}(s^3f_a(z))\frac{d\Psi}{dz}(s^3f_a(z))\right)_{21}\\
&=\frac{1}{2\pi i}\lim_{\zeta\to 0}\left(\Psi^{-1}(\zeta)\frac{d\Psi}{d\zeta}(\zeta)\right)_{21} s^{3} f'_a(a)
=s^3\frac{q(a)^2}{(a-b)(a-c)},
\end{aligned}
\end{equation}
where we have noted that, in this case,
\[f_a^{\prime}(a)=\frac{4q(a)^2}{(a-b)(a-c)}. \]

At this stage we must prove, similarly to \eqref{ddbalpha2}, the identity
\begin{equation}\label{ddaalpha2}
    \frac{d}{da}\alpha_2=-\frac{q(a)^2}{(a-b)(a-c)}.
\end{equation}
For this we, as before, differentiate \eqref{alpha2}. Then, 
using \eqref{ddaJ0} and \eqref{ddaq0}, we obtain
 obtain (\ref{ddaalpha2}). Comparing (\ref{ddaalpha2}) with \eqref{mainprelimA}, we obtain Lemma \ref{maintermA}.
\end{proof}

\begin{lemma}\label{oscillatorytermA}
 \begin{equation}\lim_{z\to a}\mathcal{\wt\tau}_2(z)_{21}= \frac{d}{da}\log \theta_3(s^{3/2}\Omega;\tau)-\frac{d\tau}{da}\frac{\partial}{\partial\tau}\log \theta_3(s^{3/2}\Omega;\tau).
 \end{equation}
\end{lemma}
\begin{proof}
To prove Lemma \ref{oscillatorytermA} we proceed as in the proof of Lemma \ref{oscillatoryterm}. We first have that
\begin{equation}\label{EBsomethingA}
  \frac{1}{2\pi i}  \frac{1}{4}\frac{f_a^{\prime}(z)}{f_a(z)}\left(\Psi^{-1}(\zeta)\sigma_3 \Psi(\zeta)\right)_{21}\to -s^3\frac{q(a)^2}{2(a-b)(a-c)},\qquad z\to a.
\end{equation}
Then, similarly to \eqref{EBEXPN} we obtain 
\be\label{EBEXPNa}
   \lim_{z\to a}{\mathcal{\wt\tau}}_2(z)_{21}=\lim_{z\to a}\left(\frac{1}{4}\left(2A-i(B+C)\right)s^{3/2}f_a^{1/2}(z) +\frac{i}{4}(B-C)s^3f_a(z)\right)-s^3\frac{q(a)^2}{2(a-b)(a-c)}.
\ee
where $A$, $B$, $C$ are given by \eqref{A0}--\eqref{C0} but without the $e^{\pm 2g_+(b)s^{3/2}}$
factors and with $\theta_{jk}$ evaluated at $u(a)=0$ rather than at $u_+(b)$.
Using the identities \eqref{thua} for $\theta_{jk}$ and its derivatives, we see that
\be
B-C=\frac{1}{2i}\frac{1}{z-a}\left( \theta_{11}\theta_{22}+\frac{u_{0,a}\beta_{0,a}^{-2}}{2}(\theta_{11}\theta'_{22}+\theta_{22}\theta'_{11})\right)\left|_{u(a)=0}\right.\left(1+\bigO(z-a)\right).
\ee
Using the fact that
\be
u_{0,a}\beta_{0,a}^{-2}=\frac{1}{J_0(a-c)}
\ee
and the identity \eqref{155atA}, we obtain
\be
B-C=\frac{1}{2i}\frac{1}{z-a}\left(1+\bigO(z-a)\right),
\ee
so that by the expansion of $f_a(z)$ in \eqref{glocalexp}, 
\begin{equation}
    \frac{i}{4}(B-C)s^3f_a(z) \to s^3 \frac{q(a)^2}{2(a-b)(a-c)}, \qquad z\to a,
\end{equation}
which cancels the last term in (\ref{EBEXPNa}), as before. 

Furthermore, we obtain similarly 
\be
2A-i(B+C)=-\frac{u_{0,a}}{(z-a)^{1/2}}\left(
\theta_{11}\theta'_{22}-\theta_{22}\theta'_{11}+\frac{1}{2J_0(a-c)}(\theta_{11}\theta''_{22}-\theta_{22}\theta''_{11})\right)\left|_{0}
+\bigO(1).\right.
\ee

By identities \eqref{id5A} and \eqref{id6A} of Lemma \ref{thetalemma}, we then obtain
\be \label{T1prelimA}
\lim_{z\to a}\mathcal{\wt\tau}_2(z)_{21}=s^{3/2}\frac{d\Omega}{da}\wt T_1(s^{3/2}\Omega),
\ee
where
\be\label{T1omegaata}
 \begin{aligned}
 \wt T_1(\omega)&
 =\frac{1}{2J_0(a-c)}\left[\theta_{11}\theta_{22}
 \left(2J_0(a-c)\left(\frac{\theta'_{11}}{\theta_{11}}-\frac{\theta'_{22}}{\theta_{22}}\right)
 +\frac{\theta''_{11}}{\theta_{11}}-\frac{\theta''_{22}}{\theta_{22}}\right)\right]_{0}\\
 &=
 -\frac{\theta_3^2\theta_3(\omega+d)\theta_3(\omega-d)}{J_0(a-c)\theta_3^2(d)\theta_3^2(\omega)}\left[ \frac{\theta_1^{\prime}}{\theta_1}(d)\left(\frac{\theta_3^{\prime}}{\theta_3}(\omega+d)+\frac{\theta_3^{\prime}}{\theta_3}(\omega-d)\right)-\frac{1}{2}\left(\frac{\theta_3^{\prime\prime}}{\theta_3}(\omega+d)- \frac{\theta_3^{\prime\prime}}{\theta_3}(\omega-d)\right) \right].
   \end{aligned}
\ee
The same argument as for \eqref{T1omegadef} above gives
\begin{equation}\label{T1viathetaata}
   \wt T_1(\omega)=2\frac{\theta_3^{\prime}}{\theta_3}(\omega).
\end{equation}
Substituting this into \eqref{T1prelimA} and writing out the total derivative, we obtain the Lemma.
\end{proof}

\begin{lemma}\label{constanttermA}With $\omega=s^{3/2}\Omega$,
\begin{equation}
   \int_0^1 \lim_{z\to a}\mathcal{\wt\tau}_3(z)_{21}d\omega=-\frac{1}{2}\frac{d}{da}\log \lvert J_0 \rvert-\frac{1}{8}\frac{d}{da}\log \lvert q(a)q(b)q(c)\rvert +\frac{d\tau}{da}\int_0^1\frac{\partial}{\partial \tau} \log\theta_3(\omega;\tau)d\omega .
\end{equation}
\end{lemma}

\begin{proof}
We have already computed most of the necessary ingredients in the proof of Lemma \ref{constantterm}. For $\omega=s^{3/2}\Omega$ let us denote the residue contribution to $\lim_{z\to a}\mathcal{\tau}_3(z)_{21}$ from the points $b$ and $c$ by $T_{2,a}(\omega)$, and from the point $a$ by $T_{3,a}(\omega)$. Thus, similarly to \eqref{tau3TT} we have
\be\label{T+T}
 \lim_{z\to a} \mathcal{\wt\tau}_3(z)_{21}=T_{2,a}+T_{3,a}
 \ee
 where
 \be\label{T2T3defa}
 T_{2,a}= \frac{1}{2\pi i}\int_{\partial U_b \cup \partial U_c}\frac{L_a(\xi)}{(\xi-a)^2}d\xi,\qquad 
T_{3,a}=  \frac{1}{2\pi i}\int_{\partial U_a }\frac{L_a(\xi)}{(\xi-a)^2}d\xi,
 \ee
 with  (note that $u(a)=0$ is the argument of $\theta_{jk}$)
 \be\label{defLa}
 L_a(\xi)=\frac{\beta^{-2}_{0,a}}{4\pi i}\left( s^3 \pi ^2 f_a^{\prime}(a)\right)^{1/2}
\begin{pmatrix}i\theta_{22}(0) & \theta_{11}(0) \end{pmatrix}\widehat{\Delta_1(\xi)}\begin{pmatrix} \theta_{11}(0) \\ -i\theta_{22}(0)\end{pmatrix},
\ee
and $\widehat{\Delta_1(\xi)}$ is given by \eqref{231}--\eqref{233}.

For $T_{2,a}$ we may proceed as in the evaluation of $T_2$ in the proof of Lemma \ref{constantterm}.
It is straightforward to derive, analogously to \eqref{secondACresidues}, that
\be\label{T2a}
\int_0^1 T_{2,a}(\omega)d\omega=\frac{1}{16(a-b)(a-c)}\left[ \frac{q(a)}{q(c)}\left(\frac{J_1}{J_0}-b\right)+\frac{q(a)}{q(b)}\left(\frac{J_1}{J_0}-c\right) \right] .    \ee

Using \eqref{ddaJ0} and \eqref{ddaq0}, we differentiate $q(b)$ and $q(c)$ directly and find that
\[-\frac{\frac{J_1}{J_0}-b}{2(a-b)(a-c)}=\frac{\frac{d}{da}q(c)}{q(a)}+\frac{a+c-b}{4q(a)}-\frac{1}{2(a-c)}, \]
and
\[-\frac{\frac{J_1}{J_0}-c}{2(a-b)(a-c)}=\frac{\frac{d}{da}q(b)}{q(a)}+\frac{a+b-c}{4q(a)}-\frac{1}{2(a-b)}. \]
Substituting these into \eqref{T2a} and noting the simple relations
\[q(a)=q(b)+\frac{1}{2}(a-b)(a+b-c), \qquad q(a)=q(c)+\frac{1}{2}(a-c)(a+c-b), \]
we arrive at the expression
\begin{equation}\label{BCRresiduesATA}
 \int_0^1 T_{2,a}(\omega) d\omega =  -\frac{1}{8}\frac{d}{da}\log \frac{\lvert q(b)q(c)\rvert}{\sqrt{(a-b)(a-c)}}.
\end{equation}

Concerning $T_{3,a}$ we, again, follow the proof of Lemma \ref{constantterm}. We first redefine
\be\label{AABB2}
\begin{aligned}
    A_1(z)&=\frac{1}{2}\left(\theta_{11}(u)(\beta(z)+\beta^{-1}(z)) +\theta_{12}(u)(\beta(z)-\beta^{-1}(z))\right),\\
    A_2(z)&=\frac{1}{2}\left(\theta_{22}(u)(\beta(z)+\beta^{-1}(z)) +\theta_{21}(u)(\beta(z)-\beta^{-1}(z)) \right),\\
    B_1(z)&=\frac{1}{2}\left(\theta_{11}(u)(\beta(z)+\beta^{-1}(z)) -\theta_{12}(u)(\beta(z)-\beta^{-1}(z)) \right),\\
    B_2(z)&=\frac{1}{2}\left(\theta_{22}(u)(\beta(z)+\beta^{-1}(z)) -\theta_{21}(u)(\beta(z)-\beta^{-1}(z)) \right),
    \end{aligned}\ee
and then derive similarly to \eqref{AANDBFORM} that for $z$ close to $a$
\begin{equation}\label{AANDBFORM2}
L_a(z)=
-\frac{\beta^{-2}_{0,a}  f_a^{\prime}(a)^{1/2}}{64 f_a(z)^{1/2}}
   \left( [B_1(z) \theta_{22}(0)-B_2(z)\theta_{11}(0)]^2+3[A_1(z)\theta_{22}(0)+A_2(z)\theta_{11}(0)]^2\right),
\end{equation}
Note that $u(a)=0$, $g_+(a)=0$, and the roles of $A$ and $B$ are interchanged.

We then obtain similarly to \eqref{T3useful} that, in terms of $\wt T_1(\omega)$ given by \eqref{T1omegaata}, \eqref{T1viathetaata},
\be\begin{aligned}
    &\beta(z)[B_1(z) \theta_{22}(0)-B_2(z)\theta_{11}(0)]\\
    &=\left\{\beta_{0,a}^{2}u_{0,a}(\theta_{11}^{\prime}\theta_{22}-\theta_{11}\theta_{22}^{\prime})+ 
    \frac{u_{0,a}^2}{2}(\theta_{11}^{\prime\prime}\theta_{22}-\theta_{11}\theta_{22}^{\prime\prime})\right\}(z-a)+\bigO((z-a)^2)\\
    &=\beta_{0,a}^2 u_{0,a}\wt T_1(\omega)(z-a)+\bigO((z-a)^2)\\
    &=2u_{0,a}^2 J_0 (a-c)\frac{\theta_3^{\prime}}{\theta_3}(\omega) (z-a)+\bigO((z-a)^2),\qquad z\to a,\qquad \omega=s^{3/2}\Omega.
\end{aligned}\ee
It follows that
\be
\begin{aligned}
T_{3,a}(\omega)&=  \frac{1}{2\pi i}\int_{\partial U_a }\frac{L_a(z)}{(z-a)^2}dz=
-\frac{\beta^{-2}_{0,a}  f_a^{\prime}(a)^{1/2}}{64}\left(\int_{\partial U_a }  \frac{dz}{2\pi i}
\frac{\beta^{-2}(z)}{f_a(z)^{1/2}}\left\{\left(2u_{0,a}^2 J_0 (a-c)\frac{\theta_3^{\prime}}{\theta_3}(\omega)\right)^2+\bigO(z-a)\right\}
\right.
\\
 &\left.+\int_{\partial U_a }  \frac{dz}{2\pi i}
 \frac{3}{(z-a)^2f_a(z)^{1/2}}[A_1(z)\theta_{22}(0)+A_2(z)\theta_{11}(0)]^2\right).
\end{aligned}
 \ee
Since the meromorphic in $U_a$ function
 \[
 \frac{\beta^{-2}(z)}{f_a(z)^{1/2}}=\frac{1}{z-a}\frac{\beta^{-2}_{0,a} }{f_a^{\prime}(a)^{1/2}}\left(1+\bigO(z-a)\right),\qquad z\to a,
 \]
we immediately compute the first integral by taking the residue (recall the negative direction of the integration), and recalling the expressions \eqref{b0}, \eqref{u0a} for $\beta_{0,a}$, $u_{0,a}$, we obtain
\be\label{penultimateAtermATA}
\begin{aligned}
T_{3,a}(\omega)&= \frac{1}{16 J_0^2(a-b)(a-c)}\left(\frac{\theta_3^{\prime}}{\theta_3}(\omega)\right)^2\\
\\
 &
 -\frac{\beta^{-2}_{0,a}  f_a^{\prime}(a)^{1/2}}{64}\int_{\partial U_a }  \frac{dz}{2\pi i}
 \frac{3}{(z-a)^2f_a(z)^{1/2}}[A_1(z)\theta_{22}(0)+A_2(z)\theta_{11}(0)]^2.
\end{aligned}
 \ee
To average the first term, we again use \eqref{onethetaint}.
By the identities of  Lemma \ref{thetalemma}, we have that
\be
    \frac{\theta_1^{\prime \prime \prime}}{\theta_1^{\prime}}=4\pi i \frac{\partial }{\partial \tau}\log \theta_1^{\prime}=\frac{\pi i}{\frac{d \tau}{d a}}\frac{d}{d a}\log( \theta_2^4 \theta_3^4 \theta_4^4)
    =J_0^2 (a-b)(a-c)\frac{d}{d a} \log \left[\left(J_0^6(a-b)(a-c)\right) \right].
\ee
Thus, we obtain
\be\label{264ATA}
\frac{1}{J_0^2(a-b)(a-c)} \int_0^1 \left( \frac{\theta_3^{\prime}}{\theta_3}(\omega) \right)^2d\omega=
\frac{\pi^2/3}{J_0^2(a-b)(a-c)}+\frac{1}{3}\frac{d}{d a} \log \left[\left(J_0^6(a-b)(a-c)\right)\right].
\ee

Similarly to the derivation of \eqref{addition}, we find
\be\label{additionATA}
\frac{d\tau}{d a} \int_0^1 \frac{\partial }{\partial \tau}\log \theta_3(\omega)d\omega=
\frac{\pi i}{J_0^2(a-b)(a-c)}\int_0^1 \frac{\partial }{\partial \tau}\log \theta_3(\omega)d\omega=
\frac{1}{4J_0^2(a-b)(a-c)}\int_0^1   \left( \frac{\theta_3^{\prime}}{\theta_3}(\omega) \right)^2d\omega.
\ee

Thus, integrating \eqref{penultimateAtermATA} and using \eqref{additionATA} and \eqref{264ATA}, we obtain
 \be\label{T3prelimATA}
 \begin{aligned}
&\int_0^1 T_{3,a}(\omega)d\omega- \frac{d\tau}{d a} \int_0^1 \frac{\partial }{\partial \tau}\log \theta_3(\omega)d\omega=-
 \frac{\pi^2/16}{J_0^2(a-b)(a-c)}-\frac{1}{16}\frac{d}{d a} \log \left[\left(J_0^6(a-b)(a-c)\right)\right]\\
  &
 -\frac{\beta^{-2}_{0,a}  f_a^{\prime}(a)^{1/2}}{64}\int_{\partial U_a }  \frac{dz}{2\pi i}
 \frac{3}{(z-a)^2f_a(z)^{1/2}}\int_0^1 d\omega [A_1(z)\theta_{22}(0)+A_2(z)\theta_{11}(0)]^2.
\end{aligned}
 \ee

We now turn to computation of the average of the term with $A$'s in this expression.
First note that the analytic in $U_a$ function
\begin{equation}\label{fa}
    \frac{1}{\beta^{-2}(z)f_a(z)^{1/2}}=\frac{1}{\beta_{0,a}^{-2}f_a^{\prime}(a)^{1/2}}\left(1-\left(\frac{f_{1,a}}{2}-2\beta_{1,a}\right)(z-a)+\bigO((z-a)^2) \right), \qquad z\to a,
\end{equation}
where we used (\ref{betaexp}) and wrote $f_a(z)=f_a^\prime(a)(z-a)(1+f_{1,a}(z-a)+\bigO((z-a)^2))$ recalling \eqref{glocalexp}. 
In the last expansion
\be\label{f1a}
f_{1,a}=\frac{2}{3}\left(\frac{q_1+2a}{q(a)}-\frac{1}{2}\left(\frac{1}{a-b}+\frac{1}{a-c}\right)\right).
\ee

Expanding $B_j(z)$ and using the identity \eqref{155atA}, we obtain for the analytic in $U_a$ function
\begin{equation}
    \beta(z)^{-1}[A_1(z)\theta_{22}(0)+A_2(z)\theta_{11}(0)]-2=\bigO(z-a), \qquad z\to a.
\end{equation}
Taking the square of this expression, we may write
\begin{equation}\label{BtermexpnATA}
\begin{aligned}
 &\beta(z)^{-2}[A_1(z)\theta_{22}(0)+A_2(z)\theta_{11}(0)]^2\\
  &=-4+4\beta(z)^{-1}
  [A_1(z)\theta_{22}(0)+A_2(z)\theta_{11}(0)]+\bigO((z-a)^2), \qquad z\to a.  
  \end{aligned}
\end{equation}
We now evaluate the average
\begin{equation}
    \int_0^1 [A_1(z)\theta_{22}(0)+A_2(z)\theta_{11}(0)]
     d\omega=\int_0^1 [\widehat{q(\omega)}_a+\widehat{q(-\omega)}_a]d\omega, \qquad \omega=s^{3/2}\Omega,
\end{equation}
where, by writing out the terms and using periodicity properties of $\theta$-functions,
\begin{equation}
\widehat{q(\omega)}_a=
        \frac{\theta_3^2}{2\theta_3^2(\omega)}\frac{\theta_3(-\omega+d)}{\theta_3(d)}\left( \frac{\theta_3(u(z)+\omega+d)}{\theta_3(u(z)+d)}(\beta(z)+\beta^{-1}(z))+\frac{\theta_3(u(z)-\omega-d)}{\theta_3(u(z)-d)}(\beta(z)-\beta^{-1}(z))\right),
\end{equation}
Since $\widehat{q(-\omega)}_a=\widehat{q(1-\omega)}_a$, we have that $\int_0^1 [\widehat{q(\omega)}_a+\widehat{q(-\omega)}_a]d\omega=
2\int_0^1 \widehat{q(\omega)}_a d\omega$.
We notice that the expression for $\widehat{q(\omega)}_a$ has the form of $\widetilde{q(\omega)}$, given in Section 9.2 of \cite{IB}.  In particular, we may proceed in the same way (ans similarly to the derivation of \eqref{Bavgexpn} above). Expanding $u(z)$ and $\beta^{-2}(z)$ as $z\to a$ by \eqref{uexp}, \eqref{betaexp}, and now using the identities 
\eqref{id5A} and \eqref{id6A} of Lemma \ref{thetalemma},  we obtain
\begin{equation}\label{BavgexpnATA}
\begin{aligned}
&\beta(z)^{-1} \int_0^1 [A_1(z)\theta_{22}(0)+A_2(z)\theta_{11}(0)]
     d\omega
=\\
&=2+\left(\frac{\pi^2}{3}u_{0,a}^2+u_{1,a}-2\beta_{1,a} \right)(z-a)+\bigO((z-a)^2), \qquad z\to a,
\end{aligned}
\end{equation}

Substituting \eqref{BtermexpnATA} into \eqref{T3prelimATA} and calculating the residue by
(\ref{BavgexpnATA}) and (\ref{fa}), we obtain
\be\label{T3almostATA}
\begin{aligned}
-\frac{\beta^2_{0,a}  f_a^{\prime}(a)^{1/2}}{64}&\int_{\partial U_a }  \frac{dz}{2\pi i}
 \frac{3}{(z-a)^2f_a(z)^{1/2}}\int_0^1 d\omega [A_1(z)\theta_{22}(0)+A_2(z)\theta_{11}(0)]^2\\
    &=\frac{1}{16}\left(\pi^2u_{0,a}^2+3u_{1,a}-3f_{1,a}/2\right)=
         \frac{1}{16}\left(\frac{\pi^2}{J_0^2(a-b)(a-c)}-\frac{q_1+2a}{q(a)} \right).
\end{aligned}\ee
For the last equation here we used \eqref{u0a} and \eqref{f1a}. Next, 
we differentiate $q(a)$ w.r.t. $a$ and use \eqref{ddaq0} which gives the equation 
\be\frac{q_1+2a}{2q(a)}=\frac{d}{da}\log \lvert q(a) \rvert +\frac{d}{da}\log \lvert J_0 \rvert.  \ee
Substituting this into \eqref{T3almostATA} and that, in turn, into \eqref{T3prelimATA} we conlude that
 \be\label{T3A}
 \begin{aligned}
&\int_0^1 T_{3,a}(\omega)d\omega- \frac{d\tau}{d a} \int_0^1 \frac{\partial }{\partial \tau}\log \theta_3(\omega)d\omega=\\
  &
-\frac{1}{2}\frac{d}{da}\log \lvert J_0 \rvert -\frac{1}{8}\frac{d}{da}\log \lvert q(a) \rvert -\frac{1}{16}\frac{d}{da}\log((a-b)(a-c)).
\end{aligned}
 \ee
Substituting \eqref{BCRresiduesATA} and \eqref{T3A} into the average of \eqref{T+T},
we finally obtain
\be\begin{aligned}
&\int_0^1 \lim_{z\to a}\mathcal{\wt\tau}_3(z)_{21}d\omega-\frac{d\tau}{da}\int_0^1 \frac{\partial}{\partial \tau}\log \theta_3(\omega)d\omega&=\int_0^1 T_{2,a}(\omega)d\omega+\int_0^1 T_{3,a}(\omega)d\omega-\frac{d\tau}{da}\int_0^1\frac{\partial}{\partial \tau }\log \theta_3(\omega)d\omega\\
&=-\frac{1}{2}\frac{d}{da}\log \lvert J_0 \rvert -\frac{1}{8}\frac{d}{da}\log \lvert q(a)q(b)q(c)\rvert,
\end{aligned}\ee
which proves Lemma \ref{constanttermA}.
\end{proof}

\section*{Acknowlegements}
The work of I.K. was partly supported by the Leverhulme Trust research project grant
RPG-2018-260. The authors are grateful to Benjamin Fahs and Alexander Its for useful discussions.

\end{document}